\documentclass[10pt, a4paper]{article}
\usepackage{amsmath,amsthm,stackengine}
\usepackage{amsfonts}
\usepackage{mathrsfs}
\usepackage{mathtools}
\usepackage{hyperref}
\usepackage{amssymb}
\usepackage{epsfig}
\usepackage{dsfont}
\usepackage{stmaryrd}
\usepackage[cal=cm]{mathalfa}
\usepackage{nicefrac}
\usepackage{float}
\restylefloat{table}
\restylefloat{figure}
\usepackage{graphicx}
\usepackage{bm}
\usepackage[nottoc]{tocbibind}
\usepackage[top=1in,bottom=1in,left=1.25in,right=1.25in]{geometry}
\usepackage{enumitem}
\usepackage{pgfplots}
\usepackage{pgfplotstable}
\usepackage{lipsum}
\usepackage{siunitx}
\usepackage[belowskip=-20pt]{caption} 
\usepackage{tikz}
\usetikzlibrary{arrows,shapes.geometric}
\usetikzlibrary{patterns}
\usepackage{pst-eucl}
\usepackage{multido,fp}
\pgfplotsset{width=7cm,compat=1.8}
\usepackage{caption}
\usepackage{subcaption}
\usepackage{fancyhdr}
\usepackage[numbers,sort&compress]{natbib}

\newcommand{\bk}{\textbf{A}}
\newcommand{\bb}{\textbf{b}}
\newcommand{\R}{\mathbb{R}}
\newcommand{\ve}{V_h(P)}
\newcommand{\p}{\mathcal{P}}
\newcommand{\T}{\mathcal{T}}
\newcommand{\e}{\mathcal{E}}
\newcommand{\tT}{\widehat{{\cal T}}}
\newcommand{\pid}{\Pi^\nabla}
\newcommand{\bn}{\textbf{n}}
\newcommand{\po}{\Pi_0}

\def\h{\hspace{1mm}}

\newcommand{\dv}{\text{div}}

\newcommand{\ol}{\overline}

\newcommand{\pw}{\mathrm{pw}}
\newcommand{\s}{\mathrm{s}}
\newcommand{\ds}{\displaystyle}
\newcommand{\bs}{\bm{\sigma}}
\newcommand{\PF}{\mathrm{PF}}
\newcommand{\F}{\mathrm{F}}
\newcommand{\NC}{\mathrm{NC}}
\def\Xint#1{\mathchoice
	{\XXint\displaystyle\textstyle{#1}}%
	{\XXint\textstyle\scriptstyle{#1}}%
	{\XXint\scriptstyle\scriptscriptstyle{#1}}%
	{\XXint\scriptscriptstyle\scriptscriptstyle{#1}}%
	\!\int}
\def\XXint#1#2#3{{\setbox0=\hbox{$#1{#2#3}{\int}$ }
		\vcenter{\hbox{$#2#3$ }}\kern-.6\wd0}}

\def\dashint{\Xint-}

\newtheorem{propn}{Proposition}[section]
\newtheorem{thm}[propn]{Theorem}
\newtheorem{lemma}[propn]{Lemma}
\newtheorem*{lemma*}{Lemma}
\newtheorem{cor}[propn]{Corollary}
\newtheorem*{thm*}{Theorem}
\theoremstyle{definition}
\newtheorem{defn}[propn]{Definition}

\newtheorem{example}{Example}[section]
\newtheorem{rem}{Remark}
\usepackage{xcolor}
\usepackage{titlesec}

\titleformat{\paragraph}
{\normalfont\normalsize\bfseries}{\theparagraph}{1em}{}
\titlespacing*{\paragraph}
{0pt}{3.25ex plus 1ex minus .2ex}{1.5ex plus .2ex}

\title{A priori and a posteriori error analysis of the lowest-order NCVEM for  second-order linear indefinite elliptic problems}
\author{Carsten Carstensen\thanks{ Department of Mathematics, Humboldt-Universit\"{a}t zu Berlin, 10099 Berlin, Germany.
	Email: cc@math.hu-berlin.de},  Rekha Khot\thanks{Department of Mathematics, Indian Institute of Technology Bombay, Powai, Mumbai, 400076. Email: rekhamp@math.iitb.ac.in, akp@math.iitb.ac.in} \h and Amiya K. Pani\footnotemark[2]}

\begin{document}
	\maketitle
	\begin{abstract}
		The  nonconforming virtual element method (NCVEM) for the approximation of the weak solution 
		to a general linear second-order non-selfadjoint  indefinite elliptic PDE in a polygonal domain $\Omega$ is analyzed under reduced elliptic regularity. The main tool in the \textit{a priori}  error analysis is the connection between the nonconforming virtual element space and the Sobolev space $H^1_0(\Omega)$ by a right-inverse $J$ of the interpolation operator $I_h$.  The stability of the discrete solution allows for the proof of existence of a unique discrete solution, of a discrete inf-sup estimate and, consequently, for optimal error estimates in the $H^1$ and $L^2$ norms. The explicit residual-based \textit{a posteriori} error estimate for the NCVEM is reliable and efficient up to the oscillation terms. Numerical experiments on different types of polygonal meshes illustrate the robustness of an error estimator and support the improved convergence rate of an adaptive mesh-refinement  in comparison to the uniform mesh-refinement.
	\end{abstract}
 \par polytopes,  enrichment, stability, \textit{a priori} error estimates, a residual-based \textit{a posteriori} error \par estimate, adaptive mesh-refinement. 
		\numberwithin{equation}{section}
		\numberwithin{figure}{section}
		\section{Introduction}
		The nonconforming virtual element method  approximates the weak solution $u\in H^1_0(\Omega)$ to the second-order linear elliptic boundary value problem 
		\begin{align}
		{\cal L} u:=-\dv(\bk \nabla u + \bb u)+\gamma u = f \quad\mbox {in}\quad\Omega\label{1}
		\end{align}
		for a given $f\in L^2(\Omega)$ in a bounded polygonal Lipschitz domain $\Omega \subset {\R}^2$ subject to homogeneous Dirichlet boundary conditions.
		\subsection{General introduction}
		The virtual element method (VEM) introduced in \cite{1} is one of the well-received polygonal methods for  approximating the  solutions to partial differential equations (PDEs) in the continuation of the mimetic finite difference method \cite{da2014mimetic}. This method is becoming increasingly popular \cite{9,4,3,6,5,2} for its ability to deal with  fairly general polygonal/polyhedral meshes. On the account of its versatility in shape of polygonal domains, the local finite-dimensional space (the space of shape functions) comprises non-polynomial functions. The novelty of this approach lies in the fact that it does not demand for the explicit construction of non-polynomial functions and the knowledge of degrees of freedom along with suitable projections onto polynomials is sufficient to implement the method. \par
		Recently,  Beir{\~a}o da Veiga \textit{et al.} discuss a conforming VEM for the indefinite problem (\ref{1}) in \cite{4}. Cangiani \textit{et al.} \cite{5} develop a nonconforming VEM under the additional condition 
		\begin{align}
		0\leq\gamma-\frac{1}{2}\dv(\bb),\label{1.1}
		\end{align}
		which makes the bilinear form coercive and significantly simplifies the analysis.
		The two papers \cite{4,5} prove \textit{a priori} error estimates for a solution $u\in H^2(\Omega)\cap H^1_0(\Omega)$ in a convex domain $\Omega$. 	The \textit{a priori} error analysis for the nonconforming VEM  in \cite{5} can be extended to the case when the exact solution $u\in H^{1+\sigma}(\Omega)\cap H^1_0(\Omega)$ with $\sigma>1/2$ as it is based on  traces. This paper shows it for all $\sigma>0$ and circumvents any trace inequality.  Huang \textit{et al.} \cite{HUANG2021113229} discuss \textit{a priori} error analysis of the nonconforming VEM  applied to Poisson and Biharmonic problems for $\sigma>0$. 
		  An \textit{a posteriori} error estimate in \cite{6}  explores  the conforming VEM for (\ref{1}) under the assumption (\ref{1.1}). There are a few contributions \cite{6,da2015residual,mora2015virtual} on residual-based \textit{a posteriori} error control for the conforming VEM. This paper presents \textit{a priori} and \textit{a posteriori} error estimates for the nonconforming VEM without (\ref{1.1}), but under the assumption that the Fredholm operator ${\cal L}$ is injective.
	\par
		\subsection{Assumptions on (\ref{1})}
		This paper solely imposes the following assumptions \ref{A1}-\ref{A3} on the coefficients $\bk, \bb, \gamma$ and the operator ${\cal L}$ in (\ref{1}) with $f\in L^2(\Omega)$.
		\begin{enumerate}[label=(\textbf{A\arabic*})]
			\item\label{A1} The coefficients $\bk_{jk}, \bb_{j},\gamma$ for $j,k=1,2$ are piecewise Lipschitz continuous functions.  For any decomposition $\cal{T}$ (admissible in the sense of Subsection $2.1$) and any polygonal domain $P\in\cal{T}$, the coefficients $\bk, \bb, \gamma$  are bounded pointwise a.e. by $\|\bk\|_{\infty}, \|\bb\|_{\infty}, \|\gamma\|_{\infty}$ and their piecewise first derivatives by $|\bk|_{1,\infty},|\bb|_{1,\infty}, |\gamma|_{1,\infty}$.
			\item \label{A2} There exist  positive constants $a_0$ and $a_1$ such that, for a.e. $x\in \Omega$, $\bk(x)$ is SPD and 
			\begin{align}
			a_0|\xi|^2 \leq \sum_{j,k=1}^{2} \bk_{jk}(x)\xi_{j}\xi_{k}\leq a_1|\xi|^2\quad\text{for all}\; \xi \in {\R}^2. \label{2}
			\end{align}
			\item \label{A3}The linear operator ${\cal L}:H^{1}_0(\Omega)\to H^{-1}(\Omega)$ is injective, {\it i.e.}, zero is not an eigenvalue of ${\cal L}$ .
		\end{enumerate}	
	Since the bounded linear operator ${\cal L}$ is a Fredholm operator \cite[p.~321]{Evans}, \ref{A3} implies that ${\cal L}$ is bijective with bounded inverse ${\cal L}^{-1}:H^{-1}(\Omega)\to H^1_0(\Omega)$.  The Fredholm theory also entails the existence of a unique solution to the adjoint problem, that is, for every $g\in L^2(\Omega)$, there exists a unique solution $\Phi\in H^1_0(\Omega)$ to
	\begin{align}
	{\cal L}^*\Phi:=-\dv(\bk\nabla \Phi)+\bb\cdot\nabla \Phi+\gamma \Phi=g.\label{5}
	\end{align}
	The bounded polygonal Lipschitz domain $\Omega$, the homogeneous Dirichlet boundary conditions, and \ref{A1}-\ref{A2} lead to some $0<\sigma\leq 1$ and positive constants $C_{\text{reg}}$ and $C^*_{\text{reg}}$ (depending only on $\sigma, \Omega$ and coefficients of ${\cal L}$)  such that, for any $f,g\in L^2(\Omega)$, the unique solution $u$ to (\ref{1}) and the unique solution $\Phi$ to (\ref{5}) belong to $H^{1+\sigma}(\Omega)\cap H^1_0(\Omega)$ and satisfy
	\begin{align}
	\|u\|_{1+\sigma,\Omega}\leq C_{\text{reg}}\|f\|_{L^2(\Omega)}\;
	\text{	and}\;\;
	\|\Phi\|_{1+\sigma,\Omega}\leq C^*_{\text{reg}}\|g\|_{L^2(\Omega)}.\label{6}
	\end{align}
	(The restriction $\sigma\leq 1$ is for convenience owing to the limitation to  first-order convergence of the scheme.)
	\subsection{Weak formulation}
	Given the coefficients $\bk, \bb,\gamma$ with \ref{A1}-\ref{A2}, define, for all $u,v\in V:=H^1_0(\Omega)$, 
	\begin{equation}
	a(u,v):=(\bk \nabla u,\nabla v)_{L^2(\Omega)},\hspace{0.5cm} b(u,v):=(u,\bb \cdot\nabla v)_{L^2(\Omega)},\hspace{0.5cm}c(u,v):=(\gamma u, v)_{L^2(\Omega)}\label{7}
	\end{equation}
	and 
	\begin{equation}
	B(u,v):=a(u,v)+b(u,v)+c(u,v)\label{8}
	\end{equation}
	(with piecewise versions $a_{\pw}, b_{\pw}, c_{\pw}$ and $ B_{\pw}$ for $\nabla$ replaced by the piecewise gradient $\nabla_{\pw}$ and local contributions $a^P, b^P , c^P$ defined in Subsection~3.1 throughout this paper). The weak formulation of the problem (\ref{1}) seeks $u\in V$ such that
	\begin{equation}
	B(u,v) = (f,v) \quad\text{for all}\; v \in V.\label{9}
	\end{equation}
	Assumptions \ref{A1}-\ref{A3} imply that the bilinear form $B(\cdot,\cdot)$ is continuous and satisfies an inf-sup condition \cite{braess2007finite}
	\begin{align}
	 0<\beta_0:=\inf_{0\neq v\in V}\sup_{0\neq w\in V}\frac{B(v,w)}{\|v\|_{1,\Omega}\|w\|_{1,\Omega}}.\label{9.1}
	\end{align}
	\subsection{Main results and outline}
	Section $2$ introduces the VEM and guides the reader to the first-order nonconforming VEM on polygonal meshes. It explains the continuity of the interpolation operator and related error estimates in detail. Section $3$ starts with the discrete bilinear forms and their properties, followed by some preliminary estimates for the consistency error and the nonconformity error. The  nonconformity error uses a new conforming companion operator resulting in the well-posedness of the discrete problem for sufficiently fine meshes. Section $4$ proves the discrete inf-sup estimate and optimal \textit{a priori} error estimates. Section $5$ discusses both reliability and efficiency of an explicit residual-based \textit{a posteriori} error estimator. Numerical experiments in Section $6$ for three computational benchmarks illustrate the performance of an error estimator and show the improved convergence rate in   adaptive mesh-refinement.
	\subsection{Notation}
	Throughout this paper, standard notation applies to Lebesgue and Sobolev spaces $H^m$ with norm $\|
	\cdot\|_{m,\cal{D}}$ (resp. seminorm $|\cdot|_{m,\cal{D}}$) for $m>0$, while $(\cdot,\cdot)_{L^2({\cal D})}$ and $\|\cdot\|_{L^2({\cal D})}$  denote the $L^2$ scalar product and $L^2$ norm on a domain ${\cal D}$. The space $C^0(\cal D)$ consists of all continuous functions  vanishing on the boundary of a domain ${\cal D}$. The dual space of $H^1_0(\Omega)$ is denoted by $H^{-1}(\Omega)$ with dual norm $\|\cdot\|_{-1}$. An inequality $A\lesssim B$ abbreviates $A\leq CB$ for   a generic constant $C$,  that may depend on the coefficients of ${\cal L}$, the universal constants $\sigma$, $\rho$ (from \ref{M2} below), but that is independent of the mesh-size. Let $\p_k({\cal D})$ denote  the set of polynomials of degree at most $k\in\mathbb{N}_0$ defined on a domain ${\cal D }$ and let $\Pi_k$ denote the piecewise $L^2$ projection on $\p_k({\cal T})$ for any admissible partition $\T\in\mathbb{T}$ (hidden in the notation $\Pi_k$).  The notation $H^s(P):= H^s(\text{int}P)$ for a compact polygonal domain $P$ means the Sobolev space $H^s$ \cite{Evans} defined in the interior $\text{int}(P)$ of $P$ throughout this paper.  The outward normal derivative is denoted by $\frac{\partial\;\bullet}{\partial\bn_P}=\bn_P\cdot\nabla\bullet$ for the exterior  unit normal vector $\bn_P$ along the boundary $\partial P$ of the domain $P$.

	\section{First-order virtual element method on a polygonal mesh}
	This section describes class of admissible partitions of $\Omega$ into polygonal domains and the lowest-order nonconforming virtual element method for the problem (\ref{1}) \cite{5,2}. 
	\subsection{Polygonal meshes}
	A  polygonal domain $P$ in this paper is a non-void compact simply-connected set $P$ with polygonal boundary $\partial P $ so that $\text{int}(P)$ is a Lipschitz domain. The polygonal boundary $\partial P$ is a simple closed polygon described by a finite sequence of distinct points. The set ${\cal N}(\partial P)=\{z_1,z_2,\dots,z_J\}$ of nodes of a polygon $P$ is enumerated with $z_{J+1}:=z_1$ such that $E(j):=\text{conv}\{z_j,z_{j+1}\}$ defines an edge and all $J$ edges cover the boundary $\partial P=E(1)\cup\dots\cup E(J)$ with an intersection $E(j)\cap E(j+1)=\{z_{j+1}\}$ for $j=1,\dots,J-1$ and $E(J)\cap E(1)={z_1}$ with $\text{dist}(E(j),E(k)) >0$ for all distinct indices $j\neq k$.
	\bigskip
\\	Let $\mathbb{T}$ be a family of partitions of  $\overline{\Omega}$  into polygonal domains, which satisfies the  conditions \ref{M1}-\ref{M2} with a universal positive constant $\rho$.
	\begin{enumerate}[label={(\bfseries M\arabic*)}]
		\item \label{M1}Admissibility.  Any two distinct polygonal domains $P$ and $P'$ in $\T\in\mathbb{T}$ are disjoint or share  a  finite number of edges  or vertices.
		\begin{figure}[H]
				\centering
				\includegraphics[width=0.3\linewidth]{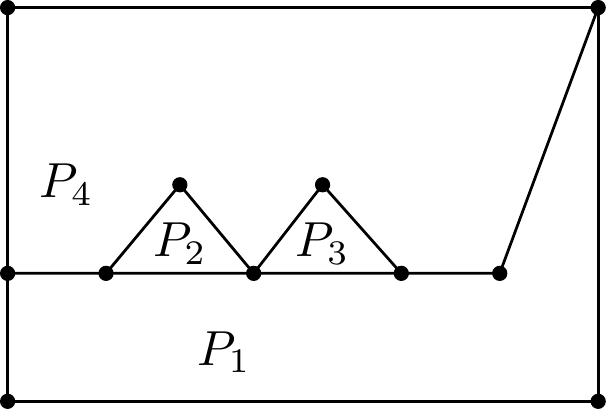}
				\caption{}
		\end{figure}
		\item \label{M2}Mesh regularity.   Every polygonal domain $P$ of diameter $h_P$ is star-shaped with respect to every point of a ball of radius greater than equal to $\rho h_P$ and every edge $E$ of $P$ has a length $|E|$ greater than equal to $\rho h_P$.
	\end{enumerate}
Here and throughout this paper,  $h_\T|_P:=h_P$ denotes the piecewise constant mesh-size   and   $\mathbb{T}(\delta):=\{\T\in\mathbb{T} : h_{\text{max}}\leq\delta\leq 1\}$ with  the maximum diameter $h_{\text{max}}$ of the polygonal domains in $\T$ denotes the subclass of partitions of $\overline{\Omega}$ into polygonal domains of maximal mesh-size $\leq\delta$. Let $|P|$ denote the area of polygonal domain $P$ and $|E|$ denote the length of an edge $E$. With a fixed orientation to a polygonal domain $P$, assign the outer unit normal $\bn_{P}$ along the boundary $\partial P$ and $\bn_E:=\bn_P|_{E}$ for an edge $E$ of $P$. Let $\e$ (resp. $\widehat{\e}$) denote the set of edges $E$ of $\T$ (resp. of $\tT$) and $\e(P)$ denote the set of edges of polygonal domain $P\in\T$.   For a polygonal domain $P$, define \begin{align*}
\text{mid}(P):=\frac{1}{|P|}\int_P x\,dx\quad\text{and}\quad\text{mid}(\partial P):=\frac{1}{|\partial P|}\int_{\partial P}x\, ds. \end{align*}\\ Let
$\p_k({\cal T}):=\{v\in L^2(\Omega):\forall P\in\T\quad v|_{P}\in\p_k(P)\}$ for $k\in\mathbb{N}_0$ and $\Pi_k$ denote the piecewise $L^2$ projection onto $\p_k({\cal T})$. The notation $\Pi_k$ hides its dependence on $\T$ and also assume $\Pi_k$ applies componentwise to vectors.
 Given a decomposition ${\cal T}\in\mathbb{T}$ of $\Omega$ and a function $f\in L^2(\Omega)$, its oscillation reads
\begin{align*}
\mathrm{osc}_k(f,P):= \|h_P(1-\Pi_k)f\|_{L^2(P)}\quad\text{and}\quad
\mathrm{osc}_k(f,{\cal T}):=\left(\sum_{P\in{\cal T}}\|h_P(1-\Pi_k)f\|_{L^2(P)}^2\right)^{\displaystyle\nicefrac{1}{2}}
\end{align*}
with $\mathrm{osc}(f,\bullet):=\mathrm{osc}_0(f,\bullet)$.
\begin{rem}[consequence of mesh regularity assumption]\label{2.4c}
		There exists an interior node $c$   in the sub-triangulation $\tT(P):=\{T(E)=\text{conv}(c,E): E\in\e(P)\}$ of a polygonal domain $P$  with $h_{T(E)}\leq h_P\leq C_{\text{sr}}h_{T(E)}$ as illustrated in  Figure \ref{tz}. Each polygonal domain $P$ can be divided into triangles so that the resulting sub-triangulation $\tT|_P:=\tT(P)$ of $\T$ is shape-regular. The minimum angle in the sub-triangulation solely depends on $\rho$  \cite[Sec.~2.1]{14}. 
		\begin{figure}[H]
			\centering
			\begin{subfigure}{.5\textwidth}
				\centering
				\includegraphics[width=0.5\linewidth]{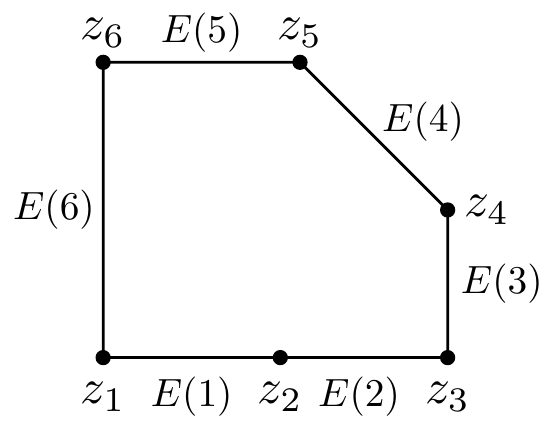}               
				\caption{}
				\vspace{1cm}
				\label{d1}
			\end{subfigure}%
			\begin{subfigure}{.5\textwidth}
				\centering
				\includegraphics[width=0.5\linewidth]{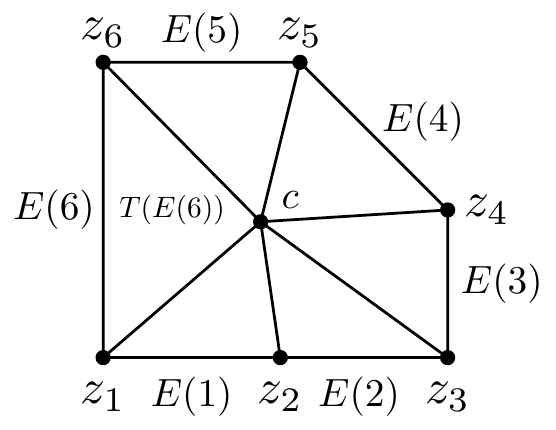}
				\caption{}
				\vspace{1cm}
				\label{d2}
			\end{subfigure}
			\caption{(a) Polygon $P$ and (b) its sub-triangulation $\tT(P)$.}
			\label{tz}
		\end{figure}
\end{rem}
\begin{lemma}[Poincar\'e-Friedrichs inequality]\label{2.4b} 
	  There exists a positive constant $C_\PF$, that  depends solely on $\rho$,  such that
	\begin{align}
	\|f\|_{L^2(P)}\leq C_\PF h_P|f|_{1,P} \label{eq}
	\end{align}
	 holds for any $f\in H^1(P)$ with 
	 $ \sum_{j\in J}\int_{E(j)} f\,ds=0$ for a nonempty subset $J\subseteq \{1,\dots,m\}$ of indices in the notation $\partial P=E(1)\cup \dots\cup E(m)$ of Figure~\ref{tz}. The constant $C_\PF$ depends exclusively on the number $m:=|\e(P)|$ of the edges in the polygonal domain $P$ and the quotient of the maximal area divided by the minimal area of a triangle in the triangulation $\tT(P)$.
\end{lemma}
Some comments on $C_\PF$ for anisotropic meshes are in order before the proof gives an explicit expression for $C_\PF$.
\begin{example}
	Consider a rectangle $P$ with a large aspect ratio divided into four congruent sub-triangles all with vertex $c=\text{mid}(P)$. Then, $m=4$ and the quotient of the maximal area divided by the minimal area of a triangle in the criss-cross triangulation $\tT(P)$ is one. Hence $C_\PF\leq 1.4231$ (from the proof below) is independent of the aspect ratio of $P$.
\end{example}
\begin{proof}[Proof of Lemma~\ref{2.4b}]
	The case $J=\{1,\dots,m\}$ with $f\in H^1(P)$ and $\int_{\partial P}f\,ds=0$ is well-known cf. e.g. \cite[Sec.~2.1.5]{14}, and follows from the Bramble-Hilbert lemma \cite[Lemma~4.3.8]{7} and the trace inequality \cite[Sec.~2.1.1]{14}. The remaining part of the proof shows the inequality \eqref{eq}  for the  case $J\subseteq\{1,\dots,m\}$. The polygonal domain $P$ and its triangulation $\tT(P)$ from Figure~\ref{tz}	has the center $c$ and the nodes $z_1,\dots,z_m$ for the $m:=|\e(P)|=|\tT(P)|$ edges $E(1),\dots,E(m)$ and the triangles $T(1),\dots,T(m)$ with $T(j)=T(E(j))=\text{conv}\{c,E(j)\}=\text{conv}\{c,z_j,z_{j+1}\}$ for $j=1,\dots,m$. Here and throughout this proof, all indices are understood modulo $m$,  e.g., $z_{0}=z_m$. The proof uses the trace identity
\begin{align}
\dashint_{E(j)} f\,ds = \dashint_{T(j)}f\,dx+\frac{1}{2}\dashint_{T(j)}(x-c)\cdot\nabla f(x)\,dx\label{e1}
\end{align}
for $f\in H^1(P)$ as in the lemma. This follows from an integration by parts and the observation that $(x-c)\cdot\bn_F = 0$ on $F\in\e(T(j))\backslash E(j)$ and the height $(x-c)\cdot\bn_{E(j)}= \frac{2|T(j)|}{|E(j)}$  of the edge $E(j)$   in the triangle $T(j)$, for $x\in E(j)$; cf. \cite[Lemma~2.1]{carstensen2012explicit} or \cite[Lemma~2.6]{8} for the remaining details. Another version of the trace identity \eqref{e1} concerns $\text{conv}\{z_j,c\}=: F(j)=\partial T(j-1)\cap\partial T(j)$ and reads
\begin{align}
\dashint_{F(j)} f\,ds& = \dashint_{T(j-1)}f\,dx+\frac{1}{2}\dashint_{T(j-1)}(x-z_{j-1})\cdot\nabla f(x)\,dx\nonumber\\&=\dashint_{T(j)}f\,dx+\frac{1}{2}\dashint_{T(j)}(x-z_{j+1})\cdot\nabla f(x)\,dx\label{e2}
\end{align}
 in $T(j-1)$ and $T(j)$. The three trace identities in \eqref{e1}-\eqref{e2} are rewritten with the following abbreviations, for $j=1,\dots m$,
\begin{align*}
&x_j:=\dashint_{E(j)}f\,ds,\quad f_j:=\dashint_{T(j)}f\,dx,\quad a_j:=\frac{1}{2}\dashint_{T(j)}(x-c)\cdot\nabla f(x)\,dx,\\ &b_j:=\frac{1}{2}\dashint_{T(j)}(x-z_{j})\cdot\nabla f(x)\,dx,\quad c_j:=\frac{1}{2}\dashint_{T(j)}(x-z_{j+1})\cdot\nabla f(x)\,dx.
\end{align*}
Let $ t_{\text{min}}=\min_{T\in\tT(P)}|T|$ and $ t_{\text{max}}=\max_{T\in\tT(P)}|T|$  abbreviate the minimal and maximal area of a triangle in $\tT(P)$ and let $\widehat{\Pi}_0f\in\p_0(\tT(P))$ denote the piecewise integral means of $f$ with respect to the triangulation $\tT(P)$. The Poincar\'e inequality in a triangle with the constant $C_\text{P}:=1/j_{1,1}$ and the first positive root $j_{1,1}  \approx 3.8317$ of the Bessel function $J_1$   from \cite[Thm.~2.1]{carstensen2012explicit} allows for
\begin{align*}
\|f-\widehat{\Pi}_0f\|_{L^2(T(j))}\leq C_\text{P}h_{T(j)} |f|_{1,T(j)}\quad\text{for}\;j=1,\dots,m.
\end{align*}
Hence $\|f-\widehat{\Pi}_0f\|_{L^2(P)}\leq C_\text{P}h_{P} |f|_{1,P}$. This and the Pythagoras theorem (with $f-\widehat{\Pi}_0f\perp\p_0(\tT(P))$ in $L^2(P)$) show 
\begin{align}
\|f\|^2_{L^2(P)}=\|\widehat{\Pi}_0f\|^2_{L^2(P))}+\|f-\widehat{\Pi}_0f\|^2_{L^2(P))}\leq \|\widehat{\Pi}_0f\|^2_{L^2(P))}+C_\text{P}^2h_P^2|f|^2_{1,P}.\label{2.4}
\end{align}
 It remains to bound the term $\|\widehat{\Pi}_0f\|^2_{L^2(P))}$. The assumption on $f$ reads $\sum_{j\in J}\int_{E(j)}f\,ds=\sum_{j\in J}|E(j)|x_j=0$
for a subset $J\subset \{1,\dots,m\}$ so that $0\in\text{conv}\{|E(1)|x_1,\dots,|E(m)|x_m\}$. It follows $0\in\text{conv}\{x_1,\dots,x_m\}$ and it is known that this implies 
\begin{align}
\sum_{k=1}^{m}x_k^2\leq {\cal{M}}\sum_{k=1}^{m}(x_{k}-x_{k-1})^2\label{e3}
\end{align}
for a constant ${\cal{M}} = \frac{1}{2(1-\cos(\pi/m))}$ that depends exclusively on $m$  \cite[Lemma~4.2]{8}. Recall \eqref{e1} in the form $x_j=f_j+a_j$ to deduce from a triangle inequality and \eqref{e3} that
\begin{align*}
\frac{1}{2}\sum_{j=1}^{m}f_j^2\leq \sum_{k=1}^m x_k^2+\sum_{\ell=1}^{m}a_\ell^2\leq {\cal{M}} \sum_{k=1}^{m}(x_{k}-x_{k-1})^2+\sum_{\ell=1}^{m}a_\ell^2.
\end{align*}
This shows that
\begin{align*}
t_{\text{max}}^{-1}\|\widehat{\Pi}_0f\|^2_{L^2(P)} =t_{\text{max}}^{-1}\sum_{j=1}^{m}|T(j)|f_j^2\leq \sum_{j=1}^{m}f_j^2\leq 2{\cal {M}}\sum_{k=1}^{m}(x_{k}-x_{k-1})^2+2\sum_{\ell =1}^{m}a_\ell^2.
\end{align*}
 Recall \eqref{e1}-\eqref{e2} in the form $f_{j}-f_{j-1}=b_{j-1}-c_j$ and $x_{j}-x_{j-1}=f_{j}-f_{j-1}+a_{j}-a_{j-1}=b_{j-1}-a_{j-1}+a_{j}-c_{j}$ for all $j=1,\dots,m$. This and the Cauchy-Schwarz inequality imply the first two estimates in
\begin{align*}
2|x_{j}-x_{j-1}|&=\bigg|\dashint_{T(j-1)}(c-z_{j-1})\cdot\nabla f(x)\,dx+\dashint_{T(j)}(z_{j+1}-c)\cdot\nabla f(x)\,dx\bigg|\\&\leq \max\{|c-z_{j-1}|,|c-z_{j+1}|\}\Big(|T(j-1)|^{-1/2}|f|_{1,T(j-1)}+|T(j)|^{-1/2}|f|_{1,T(j)}\Big)\\&\leq h_Pt_{\text{min}}^{-1/2}|f|_{1,T(j-1)\cup T(j)} 
\end{align*}
with the definition of  $h_P$ and $t_{\text{min}}$ in the end. The  inequality $\int_{T(j)}|x-c|^2\;dx \leq \frac{1}{2}h^2_{T(j)}|T(j)|$  \cite[Lemma~2.7]{8}  and the Cauchy-Schwarz inequality show, for $j=1,\dots,m$, that
\begin{align*}
|a_j|\leq 2^{-3/2}h_{T(j)}|T(j)|^{-1/2}|f|_{1,|T(j)|}\leq 2^{-3/2}h_Pt_{\text{min}}^{-1/2}|f|_{1,|T(j)|}.
\end{align*}
The combination of the previous three displayed estimates result in
\begin{align*}
4h_P^{-2}(t_{\text{min}}/t_{\text{max}}) \|\widehat{\Pi}_0f\|^2_{L^2(P)}\leq 2{\cal{M}}\sum_{k=1}^{m}|f|^2_{T(k-1)\cup T(k)}+\sum_{\ell =1}^{m}|f|^2_{1,T(\ell)}=(4{\cal{M}}+1)|f|^2_{1,P}.
\end{align*}
 This and \eqref{2.4} conclude the proof with the constant $C_\PF^2 = ({\cal{M}}+1/4)(t_{\text{max}}/t_{\text{min}})+C_\text{P}^{2}$.
\end{proof}
 In the nonconforming VEM,  the finite-dimensional space $V_h$  is a subset of the piecewise Sobolev space 
\[H^1(\T):=\{v \in L^2(\Omega): \forall P \in \T\quad v|_P \in H^1(P)\}\equiv\prod_{P\in\T}H^1(P).\]
 The piecewise $H^1$ seminorm (piecewise with respect to $\T$ hidden in the notation for brevity) reads
\[|v_h|_{1,\text{pw}}:=\bigg(\sum_{P\in\T}| v_h|_{1,P}^2\bigg)^{1/2}\quad\text{for any}\; v_h\in H^1(\T).\]

\subsection{Local virtual element space}
The first nonconforming virtual element space \cite{2} is 
a subspace of harmonic functions with edgewise constant Neumann boundary values on each polygon. 
The extended nonconforming virtual element space \cite{9,5} reads
\begin{align}
\widehat{V}_h(P):=\begin{rcases}\begin{dcases} v_h \in H^{1}(P):& \Delta v_h\in\p_1(P)\quad\text{and}\quad \forall E \in  \e(P)\quad {\frac{\partial v_h}{\partial\bn_P}}\Big|_{E} \in \p_0(E) \end{dcases}\end{rcases}.\label{2.1}
\end{align}
 \begin{defn}[Ritz projection]\label{def1}
 	Let $\pid_1$  be the Ritz projection from $ H^1(P)$ onto the affine functions $\p_1(P)$ in the $H^1$ seminorm defined, for $v_h\in H^1(P)$, by
 		\begin{align}
 		(\nabla\pid_1 v_h -\nabla v_h, \nabla\chi)_{L^2(P)} = 0\quad \text{for all}\; \chi \in \p_1(P)\quad\text{and}\quad
 		\int_{\partial P} \pid_1 v_h \,ds= \int_{\partial P} v_h\,ds.\label{11.2}
 		\end{align}
 \end{defn}
\begin{rem}[integral mean]\label{rem2}
	For $P\in\T$ and $f\in H^1(P)$, $\nabla\pid_1f=\Pi_0\nabla f$. (This follows from (\ref{11.2}.a) and the definition of the $L^2$ projection operator $\Pi_0$ (acting componentwise) onto the piecewise constants $\p_0(P;\mathbb{R}^2)$.)
\end{rem}

\begin{rem}[representation of $\pid_1$]\label{rem3}
	For $P\in\T$ and $f\in H^1(P)$,   the Ritz projection $\pid_1f$ reads
	\begin{align}
	(\pid_1 f)(x)=\frac{1}{|P|}\Big(\int_{\partial P}f\bn_P\,ds\Big)\cdot\Big(x-\text{mid}(\partial P)\Big)+\dashint_{\partial P}f\,ds \quad\text{for}\;x\in P.\label{pid}
	\end{align}
	(The proof of (\ref{pid}) consists in the verification of (\ref{11.2}): The equation (\ref{11.2}.a) follows from Remark~\ref{rem2} with an integration by parts. The equation (\ref{11.2}.b) follows from the definition of $\text{mid}(\partial P)$ as the barycenter of $\partial P$. \qedsymbol)
	\end{rem}

The enhanced virtual element spaces \cite{9,5} are designed with a computable $L^2$ projection $\Pi_1$ onto $\p_1(\T)$. 
The resulting local discrete space under consideration throughout this paper reads
\begin{align}
\ve :=\begin{rcases}\begin{dcases} v_h \in \widehat{V}_h(P):  v_h - \pid_1 v_h \perp \p_1(P) \quad \text{in}\; L^2(P)\end{dcases}\end{rcases}.\label{10}
\end{align}
The point in the selection of $V_h(P)$ is that the Ritz projection $\pid_1 v_h$ coincides with  the $L^2$ projection $\Pi_1 v_h$ for all $v_h\in V_h(P)$.
The degrees of freedom on  $P$ are given  by 
\begin{align}
\text{dof}_E(v)=\frac{1}{|E|}\int_E v\,ds \quad \textrm{for all}\; E\in\e(P)\;\text{and}\; v\in \ve.\label{10.2}
\end{align}

\begin{propn}\label{lem2.2}
	$(a)$ The vector space $\widehat{V}_h(P)$ from (\ref{2.1}) is of dimension $3+|\e( P)|$. $(b)$ $\ve$ from \eqref{10} is of dimension $|\e(P)|$ and  the triplet $(P,V_h(P),\text{dof}_E:E\in\e(P))$ is a finite element in the sense of Ciarlet \cite{ciarlet1978finite}.
\end{propn}
\begin{proof}
	Let $E(1),\dots,E(m)$ be an enumeration of the edges $\e(P)$ of the polygonal domain $P$ in a consecutive way as depicted in Figure \ref{tz}.a and define $W(P):=\p_1(P)\times\p_0(E{(1)})\times\dots\times\p_0(E{(m)})$. Recall $\widehat{V}_h(P)$ from (\ref{2.1}) and identify the quotient space $ \widehat{V}_h(P)/\mathbb{R}\equiv\left\{f\in \widehat{V}_h(P):\right.\\\left. \int_{\partial P}f\,ds=0\right\}$ with all functions in $\widehat{V}_h(P)$ having zero integral over the boundary $\partial P$ of  $P$. Since the space $\widehat{V}_h(P)$ consists of functions with an affine Laplacian and edgewise constant Neumann data, the map
\begin{align*}
S:\widehat{V}_h(P)/\mathbb{R}\to W(P),\quad\quad f\mapsto\left(-\Delta f,\frac{\partial f}{\partial \bn_P}\Big|_{E{(1)}},\dots,\frac{\partial f}{\partial \bn_P}\Big|_{E{(m)}}\right)
\end{align*}	
is well-defined and linear. The compatibility conditions for the existence of a solution of a Laplacian problem with Neumann data show that the image of $S$ is equal to
\begin{align*}	
\mathcal{R}(S)=\left\{(f_1,g_1,\dots,g_m)\in W(P):\int_P f_1dx+\sum_{j=1}^m g_j|E(j)|=0\right\}.
\end{align*}
(The proof of this identity assumes the compatible data $(f_1,g_1,\dots,g_m)$ from the set on the right-hand side and solves the Neumann problem with a unique solution $\widehat{u}$ in $\widehat{V}_h(P)/\mathbb{R} $ and $S\widehat{u}=(f_1,g_1,\dots,g_m)$.) It is known that the Neumann problem has a unique solution up to an additive constant and so $S$ is a bijection and the dimension $m+2$ of $\widehat{V}_h(P)/\mathbb{R}$ is that of $\mathcal{R}(S)$. In particular, dimension of $\widehat{V}_h(P)$ is $m+3$. This proves $(a)$.\\
Let $\Lambda_0,\Lambda_1,\Lambda_2: H^1(P)\to\mathbb{R}$ be linear functionals
\begin{align*}
\Lambda_0f:=\Pi_0f,\quad \Lambda_jf:={\cal M}_j((\pid_1-\Pi_1)f)
\end{align*}
 with ${\cal M}_jf:=\Pi_0((x_j-c_j)f)$ for $j=1,2$ and $f\in H^1(P)$ that determines an affine function $p_1\in\p_1(P)$ such that $(P,\p_1(P),(\Lambda_0,\Lambda_1,\Lambda_2))$ is a finite element in the sense of Ciarlet. For any edge $E(j)\in\e(P)$, define $\Lambda_{j+2}f=\dashint_{E(j)}f\,ds$ as integral mean of the traces of $f$ in $H^1(P)$ on $E(j)$. It is elementary to see that $\Lambda_0,\dots,\Lambda_{m+2}$ are linearly independent: If $f$ in $\widehat{V}_h(P)$ belongs to the kernel of all the linear functionals, then $\pid_1f=0$ from (\ref{pid}) with $\Lambda_jf=0$ for each $j=3,\dots,2+m$. Since the functionals $\Lambda_jf=0$ for $j=1,2$,    $(x_j-c_j)(\pid_1-\Pi_1)f=0$ and $\pid_1f=0$ imply $\Pi_1f=0$. An integration by parts leads to
\begin{align*}
\|\nabla f\|_{L^2(P)}^2=(-\Delta f,f)_{L^2(P)}+\Big(f,\frac{\partial f}{\partial\bn_P}\Big)_{L^2(\partial P)}=0.
\end{align*}
This and $\dashint_{\partial P} f\,ds=0$ show $f\equiv0$. Consequently, the intersection $\cap_{j=0}^{m+2}\text{Ker}(\Lambda_j)$ of all kernels Ker$(\Lambda_0),\dots,\text{Ker}(\Lambda)_{m+2}$ is trivial and so that the functionals $\Lambda_0,\dots,\Lambda_{m+2}$ are linearly independent. Since the number of the linear functionals is equal to the dimension of $\widehat{V}_h(P)$,  $(P,\widehat{V}_h(P),\{\Lambda_0,\dots,\Lambda_{m+2}\})$ is a finite element in the sense of Ciarlet and there exists a nodal basis $\psi_0,\dots,\psi_{m+2}$ of $\widehat{V}_h(P)$  with
\begin{align*}
\Lambda_j(\psi_k)=\delta_{jk}\quad\text{for all}\;j,k =0,\dots,m+2.
\end{align*}
The linearly independent functions $\psi_3,\dots,\psi_{m+2}$ belong to $V_h(P)$ and so dim$(V_h(P))\geq m$. Since $V_h(P)\subset \widehat{V}_h(P)$ and three linearly independent conditions $(1 -\pid_1) v_h \perp \p_1(P)$ in $L^2(P)$  are imposed on $\widehat{V}_h(P)$ to define $V_h(P)$, dim$(V_h(P)) \leq m$. This shows that dim$(V_h(P)) = m$ and hence,  the linear functionals $\text{dof}_E=\dashint_E\bullet\,ds$ for $E\in\e(P)$ form a  dual basis of  $V_h(P)$. This concludes the proof of $(b)$.
\end{proof}
\begin{rem}[stability of $L^2$ projection]\label{rem5}
	The  $L^2$ projection $\Pi_k$ for $k=0 ,1$ is $H^1$ and $L^2$ stable in $\ve$, in the sense that any $v_h$ in $V_h(P)$ satisfies
	\begin{align}
	\|\Pi_kv_h\|_{L^2(P)}\leq\|v_h\|_{L^2(P)}\; \text{and}\;\|\nabla(\Pi_kv_h)\|_{L^2(P)}\leq\|\nabla v_h\|_{L^2(P)}.\label{s1}
	\end{align}
	(The first inequality follows from the definition of $\Pi_k$. The orthogonality in \eqref{10} and the definition of $\Pi_1$  imply that the Ritz projection $\pid_1$ and the $L^2$ projection $\Pi_1$ coincide on the space $V_h(P)$ for $P\in\T$. This with the definition of the Ritz projection $\pid_1$ verifies the second inequality. \qedsymbol)
\end{rem}
\begin{defn}[Fractional order Sobolev space \cite{7}]
	Let $\alpha:=(\alpha_1,\alpha_2)$ denote a multi-index with $\alpha_j\in \mathbb{N}_0$ for $j=1,2$ and $|\alpha|:=\alpha_1+\alpha_2.$ For a real number $m$ with $0<m<1$, define
	\begin{align*}
	H^{1+m}(\omega):=\left\{v\in H^1(\omega):\frac{|v^{\alpha}(x)-v^{\alpha}(y)|}{|x-y|^{(1+m)}}\in L^2(\omega\times\omega)\quad\text{for all}\;|\alpha|=1\right\}
	\end{align*}
	with $v^\alpha$ as the partial derivative of $v$ of order $\alpha$. Define  the seminorm $|\cdot|_{1+m}$ and Sobolev-Slobodeckij norm $\|\cdot\|_{1+m}$  by
	\begin{align*}
	|v|_{1+m,\omega}^2=\sum_{|\alpha|=1}\int_{\omega}\int_{\omega}\frac{{|v^{\alpha}(x)-v^{\alpha}(y)|}^2}{|x-y|^{2(1+m)}}\,dx\,dy\quad\text{and}\quad
	\|v\|_{1+m,\omega}^2=\|v\|^2_{1,\omega}+|v|_{1+m,\omega}^2.
	\end{align*}
\end{defn}
\noindent
\begin{propn}[approximation by polynomials {\cite[Thm.~6.1]{10}}] \label{prop2.6}
	Under the assumption \ref{M2}, there exists a positive constant $C_{\mathrm{apx}}$ (depending on $\rho$ and on the polynomial degree $k$) such that, for every $v\in H^{m}(P)$, the $L^2$ projection $\Pi_k(P)$ on $\p_k$ for $k\in\mathbb{N}_0$ satisfies
	\begin{align}
	\|v-\Pi_kv\|_{L^2(P)}+h_P|v-\Pi_kv|_{1,P}\leq C_{\mathrm{apx}}h_P^{m}|v|_{m,P}\quad\text{for}\;1\leq m\leq k+1.\label{25.1}
	\end{align}
\end{propn}

\subsection{Global virtual element space}
  Define the global nonconforming virtual element space, for any $\T\in\mathbb{T}$, by
\begin{align}
V_h:=\left\{v_h \in H^1(\T): \forall P \in \T \quad v_h|_P \in \ve \quad\text{and}\quad \forall E\in \e\quad \int_{E}[v_h]_E\,ds=0\right\}.\label{12}
\end{align}
Let $[\cdot]_E$ denote the jump across an edge $E\in \e$:   For two neighboring polygonal domains $P^+$ and $P^-$ sharing a common edge $E\in\e(P^+)\cap\e(P^-)$, $[v_h]_E:=v_{h|P^{+}}-v_{h|P^{-}}$, where $P^+$ denote the adjoint polygonal domain with $\bn_{P^+|E}=\bn_E$ and $P^-$ denote the polygonal domain with $\bn_{P^-|E}=-\bn_E$. If $E\subset\partial\Omega$ is a boundary edge, then $[v_h]_E:=v_h|_E$.
\begin{example}\label{ex}
 If each polygonal domain $P$ is a triangle, then the finite-dimensional space $V_h$ coincides with CR-FEM space.   (Since the dimension of the vector space $V_h(P)$ is three and $\p_1(P)\subset V_h(P)$, $V_h(P)=\p_1(P)$ for  $P\in\T$.)
 \end{example}
 \begin{lemma}\label{lem2.5}
  There exists a universal constant  $ C_\F$ (that depends only on $\rho$ from \ref{M2}) such that, for all ${\cal T }\in\mathbb{T}$, any $ v_h\in V_h$ from (\ref{12}) satisfies
 \begin{align}
 \|v_h\|_{L^2(\Omega)}\leq C_{\mathrm{F}}|v_h|_{1,\pw}.\label{13.1a}
 \end{align} 
 \end{lemma}
 \begin{proof}
 	Recall from Remark~\ref{2.4c} that  $\tT$ is a shape regular sub-triangulation of $\T$ into triangles. Since $V_h\subset H^1(\tT)$ and the  Friedrichs' inequality  holds for all functions in $ H^1(\tT)$ \cite[Thm.~10.6.16]{7}, there exists a positive constant $C_{\text{F}}$ such that the (first) inequality holds in
 	\begin{align*}
 	\|v_h\|_{L^2(\Omega)}\leq C_{\text{F}}\left(\sum_{T\in\tT}\|\nabla v_h\|_{L^2(T)}^2\right)^{1/2}= C_{\text{F}}|v_h|_{1,\pw}.
 	\end{align*}
 	The (second) equality follows for  $v_h\in H^1(P)$ with $P\in\T$.
 \end{proof}
 Lemma~\ref{lem2.5} implies that the seminorm $|\cdot|_{1,\pw}$ is equivalent to the norm $\|\cdot\|_{1,\pw}:=\|\cdot\|^2_{L^2(\Omega)}+|\cdot|^2_{1,\pw}$ in $V_h$ with mesh-size independent equivalence constants.
\subsection{Interpolation}
\begin{defn}[interpolation operator]\label{2.3}
Let $(\psi_E : E\in\e)$ be the nodal basis  of $V_h$ defined by $\text{dof}_E(\psi_E)=1$ and $\text{dof}_{F}(\psi_E)=0$ for all other edges $F\in\e\setminus\{E\}$. The global interpolation operator $I_h:H^1_0(\Omega)\to V_h$  reads
\begin{align*}
I_hv:=\sum_{E\in\e}\Big(\dashint_E v\,ds\Big)\psi_E\quad\text{for}\;v\in V.
\end{align*}
\end{defn}\noindent
Since a Sobolev function $v\in V$ has  traces and  the jumps  $[v]_E$ vanish across any edge $E\in\e$, the interpolation operator $I_h$ is well-defined. Recall $\rho$ from \ref{M2}, $C_{\PF}$ from Lemma~\ref{2.4b}, and $C_{\text{apx}}$ from Proposition~\ref{prop2.6}.

\begin{thm}[interpolation error]\label{thm2.7}
	\begin{enumerate}[label=$\left(\alph*\right)$]
 \item There exists a positive constant $C_{\mathrm{Itn}}$ (depending  on $\rho$) such that 
 any $v\in H^1(P)$ and its interpolation $I_hv\in V_h(P)$  satisfy 
\begin{align*}
\|\nabla I_hv\|_{L^2(P)}\leq C_{\mathrm{Itn}}\|\nabla v\|_{L^2(P)}.
\end{align*}
\item  Any $P\in\T\in\mathbb{T}$ and $v\in H^1(P)$ satisfy 
  $|v-I_hv|_{1,P}\leq (1+C_{\mathrm{Itn}})\|(1-\Pi_0)\nabla v\|_{L^2(P)}$
  and \begin{align*}
  h_P^{-1} \|(1-\Pi_1I_h)v\|_{L^2(P)}+ |(1-\Pi_1I_h)v|_{1,P}\leq (1+C_\PF)\|(1-\Pi_0)\nabla v\|_{L^2(P)}.
  \end{align*}
  \item 	The positive constant $C_\mathrm{I}:=C_{\mathrm{apx}}(1+C_{\mathrm{Itn}})(1+C_{\PF})$, any $0<\sigma\leq 1$, and any $v\in H^{1+\sigma}(P)$ with the local interpolation  $I_hv|_P\in V_h(P)$ satisfy
  \begin{align}
  \|v-I_hv\|_{L^2(P)}+h_P|v-I_hv|_{1,P}\leq C_\mathrm{I}h^{1+\sigma}_P|v|_{1+\sigma,P}.\label{25.2}
  \end{align}
\end{enumerate}
\end{thm}
\begin{proof}[Proof of $(a)$]
	The boundedness of the interpolation operator in $V_h(P)$ is mentioned in \cite{5} with a soft proof in its appendix. The subsequent analysis aims at a clarification that $C_{\text{I}}$ depends exclusively on the parameter $\rho$ in \ref{M2}. The elementary arguments apply to more general situations in particular to 3D.
	Given $I_hv\in V_h(P)$, $q_1:=-\Delta I_hv\in\p_1(P)$ is affine and  $\int_E(v-I_hv)\,ds=0$. Since  $\frac{\partial I_hv}{\partial\bn_P}$ is edgewise constant, this shows $\int_E{\frac{\partial I_hv}{\partial\bn_P}}|_E(v-I_hv)\,ds=0$ for all $E\in\e(P)$ and so $\big\langle\frac{\partial I_hv}{\partial\bn_P},v-I_hv\big\rangle_{\partial P}=0$. An integration by parts leads to
	\begin{align*}
	(\nabla I_hv,\nabla(I_hv-v))_{L^2(P)}=(q_1,I_hv-v)_{L^2(P)}=(q_1,\pid_1I_hv-v)_{L^2(P)}
	\end{align*}
	with $q_1\in\p_1(P)$ and $\Pi_1v_h=\pid_1v_h$ for $v_h\in\ve$ in the last step. Consequently,
	\begin{align}
	\|\nabla I_hv\|_{L^2(P)}^2&=(\nabla I_hv,\nabla(I_hv-v))_{L^2(P)}+(\nabla I_hv,\nabla v)_{L^2(P)}\nonumber\\
	&=(q_1,\pid_1 I_hv-v)_{L^2(P)}+(\nabla I_hv,\nabla v)_{L^2(P)}\nonumber\\
	&\leq \|q_1\|_{L^2(P)}\|v-\pid_1I_hv\|_{L^2(P)}+\|\nabla I_hv\|_{L^2(P)}\|\nabla v\|_{L^2(P)}\label{1s}
	\end{align}
	with the Cauchy inequality in the last step. 
	  Remark~\ref{rem2} and \ref{rem3} on the Ritz projection,  and the definition of $I_h$ show
	\begin{align}
\Pi_0\nabla v=	\nabla\pid_1 v=|P|^{-1}\int_{\partial P}v\,\bn_P\,ds=|P|^{-1}\int_{\partial P}I_hv \bn_P\,ds=\Pi_0 \nabla I_hv=\nabla\pid_1 I_hv.\label{i1}
	\end{align} 
	The function $f:=v-\pid_1I_hv\in H^1(P)$ satisfies $\int_{\partial P}f\,ds=\int_{\partial P}(v-I_hv)\,ds=0$ and the Poincar\'e-Friedrichs inequality from Lemma~\ref{2.4b}.a shows
	\begin{align}
	\|v-\pid_1I_hv\|_{L^2(P)}\leq C_{\text{PF}}h_P\|\nabla(v-\pid_1I_hv)\|_{L^2(P)}=C_{\text{PF}}h_P\|(1-\Pi_0)\nabla v\|_{L^2(P)}\label{2s}
	\end{align}
	with  (\ref{i1}) in the last step.	Let $\phi_c\in S^1_0(\tT(P)):=\{w\in C^0(P):w|_{T(E)}\in\p_1(T(E))\quad\text{for all}\;E\in\e(P)\}$ denote the piecewise linear nodal basis function of the interior node $c$ with respect to the triangulation $\tT(P)=\{T(E): E\in\e(P)\}$ (cf. Figure \ref{tz}.b for an illustration of $\tT(P)$). An inverse estimate 
	\begin{align*}
	\|f_1\|_{L^2(T(E))}\leq C_1\|\phi_c^{1/2}f_1\|_{L^2(T(E))}\quad\text{for all}\;f_1\in\p_1(\tT(P))
	\end{align*}
	on the triangle $T(E):=\text{conv}(E\cup\{c\})$ holds with the universal constant $C_1$. A constructive proof computes the mass matrices for $T$ with and without the weight $\phi_c$ to verify that the universal constant $C_1$  does not depend on the shape of the triangle $T(E)$. This implies 
	\begin{align}
	C_1^{-1}\|q_1\|_{L^2(P)}^2\leq (\phi_cq_1,q_1)_{L^2(P)}=(-\Delta I_hv,\phi_cq_1)=(\nabla I_hv,\nabla(\phi_cq_1))_{L^2(P)}\label{3s}
	\end{align}
	with an integration by parts for $\phi_c q_1\in H^1_0(P)$ and $I_hv$ in the last step. The mesh-size independent constant $C_2$ in the standard inverse estimate
\begin{align*}
h_{T(E)}\|\nabla q_2\|_{L^2(T(E))}\leq C_2\|q_2\|_{L^2(T(E))}\quad\text{for all}\;q_2\in\p_2(T(E))
\end{align*}
depends merely on the angles in the triangle $T(E), E\in \e(P),$ and so exclusively on $\rho$. With $C^{-1}_{\text{sr}}h_P\leq h_{T(E)}$ from Remark~\ref{2.4c}, this shows
\begin{align*}
C_2^{-1}C_{\text{sr}}^{-1}h_P\|\nabla\phi_cq_1\|_{L^2(P)}\leq\|\phi_cq_1\|_{L^2(P)}\leq\|q_1\|_{L^2(P)}.
\end{align*} 
This and (\ref{3s}) lead to
\begin{align}
\|q_1\|_{L^2(P)}\leq C_1C_2C_{\text{sr}}h_P^{-1}\|\nabla I_hv\|_{L^2(P)}.\label{inv}
\end{align}
The combination with (\ref{1s})-(\ref{2s}) proves
\begin{align}
\|\nabla I_hv\|_{L^2(P)}^2&\leq (C_1C_2C_{\text{sr}}C_{\text{PF}}\|(1-\Pi_0)\nabla v\|_{L^2(P)}+\|\nabla v\|_{L^2(P)})\|\nabla I_hv\|_{L^2(P)}\nonumber\\
&\leq(1+C_1C_2C_{\text{sr}}C_{\text{PF}})\|\nabla v\|_{L^2(P)}\|\nabla I_hv\|_{L^2(P)}.\nonumber
\qedhere
\end{align}
	\end{proof}
\begin{proof}[Proof of $(b)$]
	
	The identity (\ref{i1}) reads $\Pi_0\nabla(1-I_h)v=0$ and the triangle inequality results in
	\begin{align}
	|v-I_hv|_{1,P}&=\|(1-\Pi_0)\nabla(1-I_h)v\|_{L^2(p)}
	\nonumber\\&\leq \|(1-\Pi_0)\nabla v\|_{L^2(P)}+\|(1-\Pi_0)\nabla I_hv\|_{L^2(P)}.\label{c1}
	\end{align}
	Since $I_h$ is the identity in $\p_1(P)$, it follows $
	(1-\Pi_0)\nabla I_hv=(1-\Pi_0)\nabla I_h(v-\pid_1v).$
	  This and the boundedness of the interpolation operator $I_h$ lead to
	\begin{align}
	\|(1-\Pi_0)\nabla I_hv\|_{L^2(P)}&\leq\|\nabla I_h(1-\pid_1)v\|_{L^2(P)}\nonumber\\&\leq C_{\mathrm{Itn}}\|\nabla(1-\pid_1)v\|_{L^2(P)}=C_{\mathrm{Itn}}\|(1-\Pi_0)\nabla v\|_{L^2(P)}\label{c2}
	\end{align}  
	with Remark~\ref{rem2} in the last step.
	The combination of (\ref{c1}) and (\ref{c2}) proves the first part of $(b)$.\\
	The identity	$|(1-\Pi_1I_h)v|_{1,P}=\|(1-\Pi_0)\nabla v\|_{L^2(P)}$ follows from  (\ref{i1}). Since $\Pi_1=\pid_1$ in $V_h$ and
	$\int_{\partial P}v\,ds=\int_{\partial P}I_hv\,ds=\int_{\partial P}\pid_1I_hv\,ds$,
	the Poincar\'e-Friedrichs inequality   \[\|(1-\Pi_1I_h)v\|_{L^2(P)}\leq C_{\text{PF}}h_P|(1-\Pi_1I_h)v|_{1,P}\] follows from Lemma~\ref{2.4b}.a. This concludes the proof of $(b)$.
	\end{proof}
\begin{proof}[Proof of $(c)$]
	This is an immediate consequence of the part $(b)$ with \eqref{25.1} and the Poincar\'e-Friedrichs inequality for $v-I_hv$ (from above) in Lemma~\ref{2.4b}.a.
\end{proof}

\section{Preliminary estimates}
This subsection formulates the discrete problem along with the properties of the discrete bilinear form such as boundedness and a G$\mathring{a}$rding-type inequality. 
\subsection{The discrete problem}
Denote the restriction of the bilinear forms $a(\cdot,\cdot),\h b(\cdot,\cdot)$ and $c(\cdot,\cdot)$ on a polygonal domain $P\in\T$ by $a^P(\cdot,\cdot),\h b^P(\cdot,\cdot)$ and $c^P(\cdot,\cdot)$.  The corresponding local discrete bilinear forms are defined for $u_h, v_h\in V_h(P)$ by
\begin{align}
a_h^P(u_h,v_h)&:=  (\bk \nabla \Pi_1 u_h, \nabla \Pi_1 v_h)_{L^2(P)}+S^P((1 - \Pi_1)u_h,(1 - \Pi_1)v_h) ,\label{14}\\
b_h^P(u_h,v_h)&:=\ (\Pi_1 u_h,\bb \cdot \nabla \Pi_1v_h)_{L^2(P)},\label{15}\\
c_h^P(u_h,v_h)&:= (\gamma\Pi_1 u_h,\Pi_1 v_h)_{L^2(P)},\label{16}\\
B_h^P(u_h,v_h)&:=a_h^P(u_h,v_h)+b_h^P(u_h,v_h)+c_h^P(u_h,v_h).\label{17}
\end{align}
 Choose the stability term $S^P(u_h,v_h)$ as a symmetric positive definite bilinear form on $\ve \times \ve$ for a positive constant $C_s$ independent of $P$ and $h_P$ satisfying
\begin{equation}
C_s^{-1}a^P(v_h,v_h) \leq S^P(v_h,v_h) \leq C_s a^P(v_h,v_h) \quad \text{for all}\h v_h\in V_h(P) \h \text{with}\h\Pi_1v_h=0.\label{13}
\end{equation}
For some positive constant  approximation $\overline{\bk}_P$ of $\bk$ over $P$ and the number $N_P:=|\e(P)|$ of the degrees of freedom (\ref{10.2}) of $V_h(P)$, a standard example of a stabilization term from \cite{1},\cite[Sec.~4.3]{sutton2017virtual} with a scaling coefficient $\overline{\bk}_P$  reads
\begin{align}
S^P(v_h,w_h):=\overline{\bk}_P\sum_{r=1}^{N_P} \text{dof}_r(v_h)\text{dof}_r(w_h) \quad\text{for all}\; v_h, w_h\in V_h.\label{5.1s}
\end{align}
Note that an approximation $\overline{\textbf{A}}_P$ is a positive real number (not a matrix) and  can be chosen as $\sqrt{a_0a_1}$ with the positive constants $a_0$ and $a_1$ from (\textbf{A2}).
 For $f\in L^2(\Omega)$ and $v_h\in V_h$, define the right-hand side functional $f_h$ on $V_h$ by
\begin{align}
(f_h,v_h)_{L^2(P)}&:=( f, \Pi_1v_h)_{L^2(P)}.\label{18}
\end{align}
The sum over all the polygonal domains  $P\in\T$ reads 
\begin{align*}
a_h(u_h,v_h)&:=\sum_{P\in\T} a_h^P(u_h,v_h),\hspace{0.5cm} b_h(u_h,v_h):=\sum_{P\in\T}b_h^P(u_h,v_h),\\
c_h(u_h,v_h)&:=\sum_{P\in\T} c_h^P(u_h,v_h),\hspace{0.5cm} s_h(u_h,v_h):=\sum_{P\in\T}S^P((1-\Pi_1)u_h,(1-\Pi_1)v_h),\\
B_h(u_h,v_h)&:=\sum_{P\in\T}B_h^P(u_h,v_h),\hspace{0.5cm} (f_h,v_h):=\sum_{P\in\T}(f_h,v_h)_P \quad\text{for all}\; u_h, v_h\in V_h.
\end{align*}
The discrete problem seeks $u_h\in V_h$ such that
\begin{align}
B_h(u_h,v_h)=(f_h,v_h)\quad\text{for all}\; v_h\in V_h.\label{19}
\end{align}
\begin{rem}[polygonal mesh with  small edges]
		The conditions \ref{M1}-\ref{M2} are well established and apply throughout the paper. The sub-triangulation $\tT$ may not be shape-regular without the edge condition $|E|\geq\rho h_P$ for an edge $E\in\T(P)$ and $P\in\T$, but satisfies the maximal angle condition and the arguments  employed in the proof of  \cite[Lemma~6.3]{beirao2017stability} can be applied to show \eqref{inv} in Theorem~\ref{thm2.7}.a. For more general
		star-shaped polygon domains with short edges, the recent anisotropic analysis  \cite{beirao2017stability, brenner2018virtual, cao2019anisotropic} indicates that the stabilization term has to be modified as well to avoid a logarithmic factor in the optimal error estimates.
\end{rem}
\subsection{Properties of the discrete bilinear form}
The following proposition provides two main properties of the discrete bilinear form $B_h$.
\begin{propn}\label{prop2.4}
	There exist  positive universal constants $M, \alpha$ and a universal nonnegative constant $\beta$ depending on the coefficients $\bk,\bb,\gamma$  such that
	\begin{enumerate}[label=$\left(\alph*\right)$]
	\item \label{a} Boundedness:  
	$|B_h(u_h,v_h)| \leq M|u_h|_{1,\pw}|v_h|_{1,\pw} \quad\text{for all}\h u_h,v_h \in V_h.$
	\item \label{b}G$\mathring{a}$rding-type inequality: 
	$\alpha|v_h|^2_{1,\pw}-\beta\|v_h\|^2_{L^2(\Omega)}\leq B_h(v_h,v_h) \quad\text{for all}\h v_h\in V_h.$
	\end{enumerate}
	\begin{proof}[Proof of \ref{a}]
		  The upper bound of the coefficients from the assumption \ref{A1}, the Cauchy-Schwarz inequality, the stability (\ref{s1}) of $\Pi_1$, and  the definition (\ref{13}) of the stabilization term imply the boundedness of $B_h$ with $M:=(1+C_s)\|\bk\|_{\infty}+C_\F\|\bb\|_{\infty}+C_\F^2\|\gamma\|_{\infty}$. The details of the proof follow as in \cite[Lemma~5.2]{4}
		   with the constant $C_\F$ from Lemma~\ref{lem2.5}. 
		\end{proof}
	\begin{proof}[Proof of \ref{b}]
		  The first step shows that $a_h(\cdot,\cdot)$ is coercive. For $v_h \in \ve$,  $\Pi_1v_h=\pid_1v_h$ and $\nabla \Pi_1v_h\perp\nabla(v_h-\pid_1v_h)$ in $L^2(P;\mathbb{R}^2)$. This orthogonality, the assumption \ref{A2}, and the definition of the stability term (\ref{13}) with the constant $C_s^{-1}\leq 1$ imply for $\alpha_0=a_0C_s^{-1}$  that
		\begin{align}
		&\alpha_0|v_h|_{1,\pw}^2\leq  a_0\|\nabla_\pw\Pi_1 v_h\|^2_{L^2(\Omega)}+a_0C_s^{-1}\|\nabla_\pw (1-\Pi_1) v_h\|^2_{L^2(\Omega)}\nonumber\\&\quad\leq \left(\bk\nabla_\pw \Pi_1v_h,\nabla_\pw \Pi_1v_h)_{L^2(\Omega)}+C_s^{-1} (\bk\nabla_\pw(1-\Pi_1) v_h,\nabla_\pw(1-\Pi_1)v_h\right)_{L^2(\Omega)}\nonumber
		 \\&\quad\leq (\bk\nabla_\pw \Pi_1v_h,\nabla_\pw \Pi_1v_h)_{L^2(\Omega)} +s_h((1-\Pi_1)v_h,(1-\Pi_1)v_h)=a_h(v_h,v_h).\label{23}
		\end{align}
		The Cauchy-Schwarz inequality, (\ref{s1}), and the Young inequality lead to
		\begin{align}
		|b_h(v_h,v_h)+c_h(v_h,v_h)|
		&\leq \|\bb\|_\infty\|\Pi_1 v_h\|_{L^2(\Omega)}\|\nabla_\pw\Pi_1 v_h\|_{L^2(\Omega)}+\|\gamma\|_\infty\|\Pi_1 v_h\|_{L^2(\Omega)}^2\nonumber\\
		&\leq \|\bb\|_\infty\|v_h\|_{L^2(\Omega)}|v_h|_{1,\pw}+\|\gamma\|_\infty\|v_h\|_{L^2(\Omega)}^2\nonumber\\
		&\leq \frac{\|\bb\|^2_\infty}{2\alpha_0}\|v_h\|_{L^2(\Omega)}^2+\frac{\alpha_0}{2}|v_h|^2_{1,\pw}+\|\gamma\|_\infty\|v_h\|_{L^2(\Omega)}^2.\label{24}
		\end{align}
		The combination of (\ref{23})-(\ref{24}) proves
		\begin{align*}
		 \frac{\alpha_0}{2}|v_h|^2_{1,\pw}-\left(\frac{\|\bb\|^2_\infty}{2\alpha_0}+\|\gamma\|_\infty\right)\|v_h\|^2_{L^2(\Omega)}\leq B_h(v_h,v_h).
		\end{align*}
		This concludes the proof of \ref{b} with $\alpha=\frac{\alpha_0}{2}$ and  $\beta=\frac{\|\bb\|^2_\infty}{2\alpha_0}+\|\gamma\|_{\infty}$.
	\end{proof}
\end{propn}
\begin{rem}[$\|\cdot\|_h\approx |\cdot|_{1,\pw}$]\label{rem6}
The discrete space $V_h$ of the nonconforming VEM is endowed with the natural norm $\|\cdot\|_h := a_h(\cdot,\cdot)^{1/2}$ induced by the scalar product $a_h$. The boundedness of $a_h$ is proven in $(a)$, while \eqref{23} shows the converse estimate in the equivalence $\|\cdot\|_h\approx |\cdot|_{1,\pw}$ in $V_h$, namely
\[\alpha_0|v_h|^2_{1,\pw}\leq a_h(v_h,v_h)\leq \|\bk\|_\infty(1+C_s)|v_h|_{1,\pw}^2\quad\text{for all}\; v_h\in V_h.\]
\end{rem}
\subsection{Consistency error}
This subsection discusses the consistency error between the continuous bilinear form $B$ and the corresponding discrete bilinear form $B_h$. Recall the definition  $B^P(\cdot,\cdot)\equiv a^P(\cdot,\cdot)+b^P(\cdot,\cdot)+c^P(\cdot,\cdot)$
and 
 $B_h^P(\cdot,\cdot)\equiv a_h^P(\cdot,\cdot)+b_h^P(\cdot,\cdot)+c_h^P(\cdot,\cdot)$ for a polygonal domain $P\in\T$ from Subsection~2.1.
\begin{lemma}[consistency]\label{l1}
		$(a)$ There exists a positive constant $C_{\text{cst}}$  (depending only on $\rho$) such that any $v\in H^1(\Omega)$ and $w_h\in V_h$ satisfy
		\begin{align} 
		B^P(\Pi_1 v,w_h)-B_h^P(\Pi_1 v,w_h)\leq C_{\mathrm{cst}}\,h_P\|v\|_{1,P}|w_h|_{1,P}\quad\text{for all}\; P\in \T.\label{25}
		\end{align}
		$(b)$ Any $f\in L^2(\Omega)$ and $f_h:=\Pi_1f$ satisfy
		\begin{align}
		\|f-f_h\|_{V_h^*}:=\sup_{0\neq v_h\in V_h}\frac{(f-f_h,v_h)}{\|v_h\|_{1,\pw}}\leq C_\PF\, \mathrm{osc}_1(f,\T).\label{26}
		\end{align}	
\end{lemma}
\begin{proof}
	Observe that  $S^P((1-\Pi_1)\Pi_1 v,(1-\Pi_1)w_h)=0$ follows from $(1-\Pi_1)\Pi_1 v=0$. The definition of $B^P$ and $B_h^P$ show
	\begin{align}
	B^P(\Pi_1 v,w_h)-B_h^P(\Pi_1 v,w_h)=: T_1+T_2+T_3.\label{27}
	\end{align}
	The term $T_1$ in (\ref{27}) is defined as the difference of the  contributions from $a^P$ and $a^P_h$. Their definitions prove the equality (at the end of the first line below) and the definition of $\Pi_1$ prove the next equality in
	\begin{align*}
	T_1&:=a^P(\Pi_1 v,w_h)-a_h^p(\Pi_1 v,w_h)
	=(\bk\nabla\Pi_1 v,\nabla (1-\Pi_1) w_h)_{L^2(P)}\\
	&=((\bk-\po\bk)(\nabla\Pi_1 v),\nabla (1-\Pi_1) w_h)_{L^2(P)}\leq h_P  |\bk|_{1,\infty}|v|_{1,P}|w_h|_{1,P}.
	\end{align*}
	 The last inequality follows from the Cauchy-Schwarz inequality, the Lipschitz continuity of $\bk$, and the stabilities $\|\nabla\Pi_1v_h\|_{L^2(P)}\leq \|\nabla v_h\|_{L^2(P)}$ and $\|\nabla(1-\Pi_1)w_h\|_{L^2(P)}\leq\|\nabla w_h\|_{L^2(P)}$ from Remark~\ref{rem5}. 
	 Similar arguments apply to $T_2$ from the  differences of $b^P$ and $b^P_h$, and $T_3$ from those of $c^P$ and $c_h^P$ in (\ref{27}). This leads to 
	\begin{align*}
	T_2&:=b^P(\Pi_1v,w_h)-b_h^P(\Pi_1 v,w_h)\nonumber\\&=((\bb-\Pi_0\bb)\Pi_1v,\nabla(1-\Pi_1)w_h)_{L^2(P)}+((\Pi_0\bb)(1-\Pi_0)(\Pi_1v),\nabla(1-\Pi_1)w_h)_{L^2(P)}\nonumber\\&\leq
	 (|\bb|_{1,\infty}+C_{\mathrm{apx}}\|\bb\|_{\infty})h_P\|v\|_{1,P}|w_h|_{1,P},\\
	T_3&:=c^P(\Pi_1v,w_h)-c_h^P(\Pi_1v,w_h)=(\gamma\Pi_1v,(1-\Pi_1)w_h)_{L^2(P)}\leq C_{\PF}\, \|\gamma\|_{\infty}h_P\|v\|_{L^2(P)}|w_h|_{1,P}.
	\end{align*}
	The inequality for the last step in $T_2$ follows from the Cauchy-Schwarz inequality, the Lipschitz continuity of $\bb$, the estimate $\|(1-\Pi_0)\Pi_1v\|_{L^2(P)}\leq \|(1-\Pi_0)v\|_{L^2(P)}\leq C_{\text{apx}}h_P|v|_{1,P}$ from \eqref{25.1}, and the above stabilities $\|\nabla\Pi_1v_h\|_{L^2(P)}\leq \|\nabla v_h\|_{L^2(P)}$ and $\|\nabla(1-\Pi_1)w_h\|_{L^2(P)}\leq\|\nabla w_h\|_{L^2(P)}$. The inequality for the last step in $T_3$ follows from the Cauchy-Schwarz inequality, $\|\Pi_1v\|_{L^2(P)}$ 
	$\leq\|v\|_{L^2(P)}$ from \eqref{s1} and the Poinca\'re-Friedrichs  inequality in Lemma~\ref{2.4b}.a for $w_h-\Pi_1w_h$ with $\int_{\partial P}(w_h-\Pi_1w_h)\,ds=0$ from $\Pi_1=\pid_1$ in $V_h$. The combination of the above estimates shows (\ref{25}). 
	The proof of (\ref{26}) adapts the arguments in the above analysis of $T_3$ and the definition of $\mathrm{osc}_1(f,\T)$ in Subsection~2.1 for the proof of
	\begin{align*}
	(f-f_h,w_h)_{L^2(P)}
	=(f-\Pi_1 f,w_h-\Pi_1 w_h)_{L^2(P)}\leq C_{\text{PF}}|w_h|_{1,P}\, \mathrm{osc}_1(f,P).
	\end{align*}
	This concludes the proof.
\end{proof}
\subsection{Nonconformity error}
Enrichment operators  play a vital role in the analysis of nonconforming finite element methods \cite{brenner2015forty}. For any $v_h\in V_h,$ the objective is to find a corresponding function $Jv_h\in H_0^1(\Omega)$. The idea is to map the VEM nonconforming space into the Crouzeix-Raviart finite element space
\begin{align*}
\text{CR}_0^1(\tT):=
\{
v\in{\cal P}_1(\tT):\h & \forall\; E\in \widehat{{\cal E}}\quad v \h\text{is continuous at mid}(E)\quad\text{and}\\\h&\forall\; E\in {\cal E}(\partial\Omega)\quad v(\text{mid}(E))=0\}
\end{align*}
with respect to the shape-regular triangulation $\tT$ from Remark~\ref{2.4c}.
Let $\psi_E$ be the edge-oriented basis functions of CR$_0^1(\tT)$ with $\psi_E(\text{mid}\h E)=1$ and $\psi_E(\text{mid}\h F)=0$ for all other edges $F\in \widehat{{\cal E}}\setminus \{E\}.$  Define   the interpolation operator $I_{\text{CR}} : V_h\to \text{CR}_0^1(\tT)$, for $v_h\in V_h$, by
\begin{align}
I_{\text{CR}}v_h=\sum_{F\in\widehat{{\cal E}}}\left(\dashint_{F}v_h\,ds\right)\psi_{F}.\label{CR}
\end{align}
 The definition of $V_h$ implies $\int_F[v_h]\,ds=0$ for  $v_h\in V_h$ and for all $F\in\e$. Since $v_h|_{P}\in H^1(P)$, it follows $\int_F[v_h]\,ds=0$ for all $F\in \widehat{{\cal E}}\setminus\e$. This shows  $\int_{F} v_{h|T^{\pm}}\,ds$ is unique for all edges $F=\partial T^+\cap\partial T^-\in\widehat{{\cal E}}$  and, consequently, $I_{\text{CR}}v_h$ is well-defined (independent of the choice of traces selected in the evaluation of $\dashint_F v_h\,ds=\dashint_F v_h|_{T^+}\,ds=\dashint_F v_h|_{T^-}\,ds$). The approximation property of $I_{\text{CR}}$ on each $T\in\tT$  reads
\begin{align}
h_T^{-1}\|v_h-I_{\text{CR}}v_h\|_{L^2(T)}+|v_h-I_{\text{CR}}v_h|_{1,T}\leq 2|v_h|_{1,T}\label{im}
\end{align}
 (cf. \cite[Thm~2.1]{cc1} or \cite[Thm~4]{cc2} for explicit constants).
\noindent
Define an enrichment operator $E_h: \text{CR}_0^1(\tT)\to H_0^1(\Omega)$ 
 by averaging the function values at each interior vertex $z$, that is,
\begin{align}
E_hv_{\text{CR}}(z)=\frac{1}{|\tT(z)|}\sum_{T\in \tT(z)}{v_{\text{CR}}}|_{T}(z)\label{3.2}
\end{align}
and zero on boundary vertices. In (\ref{3.2}) the set $\tT(z):=\{T\in \tT\h|\h z\in T\}$ of neighboring triangles  has the cardinality $|\tT(z)|\geq 3$.

 The following lemma describes the construction of a modified companion operator  $J:V_h\to H_0^1(\Omega)$, which is a right-inverse of the interpolation operator $I_h$ from Definition \ref{2.3}.
\begin{lemma}[conforming companion operator]\label{lem3.2}
	There exists a linear map $J:V_h\to H^1_0(\Omega)$ and a universal constant $C_\mathrm{J}\lesssim 1$ such that any $v_h\in V_h$ satisfies $I_hJv_h=v_h$ and
	\begin{enumerate}[label=(\alph*)]
		\item \label{l41} $\displaystyle\dashint_{E}Jv_h\,ds = \dashint_{E}v_h\,ds$ for any edge $E\in \widehat{{\cal E}},$
		\item \label{l42} $\displaystyle \nabla_\pw(v_h-Jv_h)\perp\p_0(\T;\mathbb{R}^2)$ in $L^2(\Omega;\mathbb{R}^2),$
		\item \label{l43} $\displaystyle v_h-Jv_h\perp \p_1(\T)$ in $L^2(\Omega),$
		\item \label{l44}
		$\|h_\T^{-1}(v_h-Jv_h)\|_{L^2(\Omega)}+|v_h-Jv_h|_{1,\pw}\leq C_\mathrm{J}|v_h|_{1,\pw}.$
	\end{enumerate}
\end{lemma}
\begin{proof}[Design of $J$ in Lemma~\ref{lem3.2}]\phantom{\qedhere}
	Given $v_h\in V_h$, let $v_{\text{CR}}:=I_{\text{CR}}v_h\in \text{CR}^1_0(\tT)$. There exists an operator $J':\text{CR}_0^1(\tT)\to H_0^1(\Omega)$ from \cite[Prop.~2.3]{12} such that any $v_\text{CR}\in\text{CR}_0^1(\tT)$ satisfies 
	\begin{enumerate}[label=(\alph*')]
		\item \label{l31} $\displaystyle\dashint_{E}J'v_{CR}\,ds = \dashint_{E}v_{CR}\,ds$ for any edge $E\in \widehat{{\cal E}},$
		\item \label{l32} $\displaystyle\int_{P}\nabla_{\pw}(v_{\text{CR}}-J'v_{\text{CR}})\,dx=0$ for all $P\in\T$,
		\item \label{l33}
		$\displaystyle\|h_{\tT}^{-1}(v_{\text{CR}}-J'v_{\text{CR}})\|_{L^2(\Omega)}+|v_{\text{CR}}-J'v_{\text{CR}}|_{1,\pw}\leq C_\mathrm{J'}\min_{v\in H^1_0(\Omega)}|v_{\text{CR}}-v|_{1,\pw}$
	\end{enumerate} 
with a universal constant $C_\mathrm{J'}$ from \cite{8}. Set $v:=J'I_{\text{CR}}v_h\in V:=H^1_0(\Omega)$. Recall that   $\tT(P)$ is a shape-regular triangulation of $P$ into a finite number of triangles. For each $T\in\tT(P)$, let $b_T\in W_0^{1,\infty}(T)$ denote the cubic bubble-function $27\lambda_1\lambda_2\lambda_3$ for the barycentric co-ordinates $\lambda_1, \lambda_2, \lambda_3\in\p_1(T)$ of $T$ with $\dashint_Tb_T\,dx=9/20$ and $\|\nabla b_T\|_{L^2(T)}\lesssim h_T^{-1}|T|^{1/2}\approx 1.$ Let $b_T$ be extended by zero outside $T$ and, for $P\in\T$, define 
	\begin{align}
	b_P:=\frac{20}{9}\sum_{T\in\tT(P)}b_T\in W_0^{1,\infty}(P)\subset W_0^{1,\infty}(\Omega)\label{bubble}
	\end{align}
	 with 
	$
	\dashint_{P}b_P\,dx=1$ and $\|\nabla b_P\|_{L^2(P)}\lesssim h_P^{-1}|P|^{1/2}\approx 1$.
	 Let $v_P\in\p_1(\T)$ be the Riesz representation of the linear functional $\p_1(\T)\to\mathbb{R}$ defined by $w_1\mapsto(v_h-v,w_1)_{L^2(\Omega)}$ for $w_1\in\p_1(\T)$ in the Hilbert space $\p_1(\T)$ endowed with the weighted $L^2$ scalar product $(b_P\bullet,\bullet)_{L^2(P)}$. Hence $v_P$ exists uniquely and satisfies $\Pi_1(v_h-v) = \Pi_1(b_Pv_P)$. 
	Given the bubble-functions $(b_P:P\in\T)$ from \eqref{bubble} and the above functions $(v_P:P\in\T)$ for $v_h\in V_h$, define
	\begin{align}
	Jv_h:=v+\sum_{P\in\T}v_Pb_P\in V.\label{J}
	\end{align} 
	\end{proof}
	\begin{proof}[Proof of \ref{l41}]
	Since $b_P$ vanishes at any $x\in E\in \e$, it follows for any $E\in\widehat{\e}$ that
	\begin{align}
	\dashint_E Jv_h\,ds=\dashint_E v\,ds=\dashint_E J'v_{\text{CR}}\,ds=\dashint_E v_{\text{CR}}\,ds=\dashint_E v_h\,ds,\nonumber
	\end{align}
	 where the definition of $v=J'v_{\text{CR}}$, \ref{l31}, and $v_{\text{CR}}=I_{\text{CR}}v_h$ lead to the second, third, and fourth equality. This proves \ref{l41}.
	\end{proof}
	\begin{proof}[Proof of \ref{l42}]
	An integration by parts and \ref{l41}  show, for all $v_h\in V_h$ with $Jv_h$ from \eqref{J}, that
	\begin{align*}
	\int_P\nabla Jv_h\,dx=\int_{\partial P}Jv_h\bn_P\,ds=\sum_{E\in \e(P)}\Big(\int_EJv_h\bn_E\,ds\Big)=\sum_{E\in \e(P)}\Big(\int_Ev_h\bn_E\,ds\Big)=\int_P\nabla v_h\,dx.
	\end{align*}
	Since this holds for all $P\in\T$,  it proves \ref{l42}.
\end{proof}
\begin{proof}[Proof of \ref{l43}]
This is $\Pi_1v_h=\Pi_1Jv_h$ and guaranteed by the design of $J$ in \eqref{J}.
	\end{proof}
	\begin{proof}[Proof of \ref{l44}]
	This  relies on the definition of $J$ in \eqref{J} and $J'$ with \ref{l33}. Since \ref{l41} allows for $\int_{\partial P}(v_h-Jv_h)\,ds=0$, the Poincar\'e-Friedrichs inequality from Lemma~\ref{2.4b}.a implies
	\begin{align}
	h_P^{-1}\|v_h-Jv_h\|_{L^2(P)} \leq C_{\text{PF}}|v_h-Jv_h|_{1,P}.\nonumber
	\end{align}
	 Hence it remains to prove $
	|v_h-Jv_h|_{1,\pw}\lesssim |v_h|_{1,\pw}.$ 	Triangle inequalities with $v_h, Jv_h, v=J'v_{\text{CR}}$ and  $v_{\text{CR}} =I_{\text{CR}} v_h$ show the first and second inequality in
	\begin{align}
	|v_h-Jv_h|_{1,\pw}-|v-Jv_h|_{1,\pw}&\leq |v-v_h|_{1,\pw}\nonumber\\&\hspace{-1cm}\leq|v_h-I_{\text{CR}}v_h|_{1,\pw}+|v_{\text{CR}}-J'v_{\text{CR}}|_{1,\pw}\leq (1+C_\mathrm{J'})|v_h|_{1,\pw}\label{*}
	\end{align}
	with \ref{l32} for $|v_{\text{CR}}|_{1,\pw}=\|\Pi_0\nabla_{\pw}v_h\|_{L^2(\Omega)}\leq\|\nabla_{\pw}v_h\|_{L^2(\Omega)}=|v_h|_{1,\pw}$ in the last step. 
	The equivalence of norms in the finite-dimensional space $\p_1(P)$ assures the existence of  a positive constant $C_b$, independent of $h_P$, such that any $\chi\in\p_1(P)$ satisfies the inverse inequalities
	\begin{align}
	C_b^{-1}\|\chi\|^2_{L^2(P)}\leq &(b_P,\chi^2)_{L^2(P)}\leq C_b\|\chi\|^2_{L^2(P)},\label{b1}\\
	C_b^{-1}\|\chi\|_{L^2(P)}\leq\|b_P\chi\|_{L^2(P)}&+h_P\|\nabla(b_P\chi)\|_{L^2(P)}\leq C_b\|\chi\|_{L^2(P)}.\label{b2}
	\end{align}
	These estimates are completely standard on shape-regular triangles \cite[p.~27]{15} or \cite{16}; so they hold on each $T\in\tT$ and, by definition of $b_P$, their sum is \eqref{b1}-\eqref{b2}. The analysis of the term $|v-Jv_h|_{1,\pw}$ starts with one $P\in\T$ and \eqref{J} for
	\begin{align}
	|v-Jv_h|_{1,P}=|v_Pb_P|_{1,P}\leq C_bh_P^{-1}\|v_P\|_{L^2(P)}\label{j1}
	\end{align}
	with  (\ref{b2})  in the last step. The estimate (\ref{b1}) leads to the first inequality in 
	\begin{align*}
	C_b^{-1}\|v_P\|^2_{L^2(P)}\leq (b_Pv_P,v_P)_{L^2(P)}=(v_h-v,v_P)_{L^2(P)}\leq \|v_h-v\|_{L^2(P)}\|v_P\|_{L^2(P)}. 
	\end{align*}
	The  equality  results from $\Pi_1(v_h-v)=\Pi_1(v_Pb_P)$ and $v_P\in\p_1(\T)$, while the last step is the Cauchy-Schwarz inequality. Consequently,  $\|v_P\|_{L^2(P)}\leq C_b\|v_h-v\|_{L^2(P)}$. This and (\ref{j1}) show
	\begin{align*}
	|v-Jv_h|_{1,\pw}\leq C_b^2 \|h^{-1}_{\T}(v-v_h)\|_{L^2(\Omega)}\leq C_b^2C_\PF|v-v_h|_{1,\pw}
	\end{align*}
	 with 
	 $\int_{\partial P}(v-v_h)\,ds=0$ from $(a)$ and hence the Poincar\'e-Friedrichs inequality for $v-v_h$  from Lemma~\ref{2.4b}.a in the last step. Recall $|v-v_h|_{1,\pw}\lesssim |v_h|_{1,\pw}$ from   (\ref{*}) to conclude $|v-Jv_h|_{1,\pw}\lesssim |v_h|_{1,\pw}$ from the previous displayed inequality. This concludes the proof of \ref{l44}.
\end{proof}
\begin{proof}[Proof of $I_hJ = \text{id}\;\text{ in}\; V_h$]
	Definition \ref{2.3} and Lemma~\ref{lem3.2}.a show, for all $v_h\in V_h$, that
	\begin{align*}
	I_hJv_h=\sum_{E\in\e}\Big(\dashint_E Jv_h\,ds\Big)\psi_E=\sum_{E\in\e}\Big(\dashint_E v_h\,ds\Big)\psi_E=v_h.
	\end{align*}
		This concludes the proof of Lemma~\ref{lem3.2}.
\end{proof}
Since $V_h$ is not a subset of $H^1_0(\Omega)$ in general, the substitution of discrete function $v_h$  in the weak formulation leads to a nonconformity error. 
\begin{lemma}[nonconformity error]\label{l2}
	There exist  positive universal constants $C_{\NC}, C^*_{\NC}$ (depending on the coefficients $\bk,\bb$ and the universal constants  $\rho, \sigma$) such that  all $f,g\in L^2(\Omega)$ and  all $\T\in\mathbb{T}(\delta)$ (with the assumption $h_\text{max}\leq\delta\leq 1$) satisfy $(a)$ and $(b)$.\\
	 $(a)$ The solution $u\in H^{1+\sigma}(\Omega)\cap H^1_0(\Omega)$ to $(\ref{1})$ satisfies
	\begin{align}
	\sup_{0\neq v_h\in V_h}\frac{| B_\pw(u,v_h)-(f,v_h)_{L^2(\Omega)}|}{\|v_h\|_{1,\pw}}\leq C_{\NC}h_{\text{max}}^{\sigma}\|f\|_{L^2(\Omega)}.\label{31}
	\end{align}
	$(b)$ The solution $\Phi\in H^{1+\sigma}(\Omega)\cap H^1_0(\Omega)$ to the dual problem $(\ref{5})$ satisfies
	\begin{align}
	\sup_{0\neq v_h\in V_h}\frac{| B_\pw(v_h,\Phi)-(g,v_h)_{L^2(\Omega)}|}{\|v_h\|_{1,\pw}}\leq C^*_{\NC}h_{\text{max}}^{\sigma}\|g\|_{L^2(\Omega)}.\label{3.10}
	\end{align}	
\end{lemma}
\begin{proof}[Proof of $(a)$]
	 Given $v_h\in V_h$, define $Jv_h\in V$ and  the piecewise averages $\ol{\bk}:=\po(\bk), \ol{\bb}:=\po(\bb)$, and $\ol{\gamma}:=\po(\gamma)$ of the coefficients $\bk, \bb$, and $\gamma$. The choice of test function $v:=Jv_h\in V$ in the weak formulation \eqref{9} having extra properties provides the terms with oscillations in the further analysis.  Abbreviate $\bs:=\bk\nabla u+\bb u$. The weak formulation (\ref{9}), Lemma~\ref{lem3.2}.b-c, and the Cauchy-Schwarz inequality  reveal that
	\begin{align}
	&B_\pw(u,v_h)-(f,v_h)_{L^2(\Omega)}=B_\pw(u,v_h-Jv_h)-(f,v_h-Jv_h)_{L^2(\Omega)}\nonumber\\
	&\leq \|\bs-\Pi_0\bs\|_{L^2(\Omega)}\|\nabla_\pw(1-J)v_h\|_{L^2(\Omega)}+\|h_\T(1-\Pi_1)(f-\gamma u)\|_{L^2(\Omega)}\|h_{\T}^{-1}(1-J)v_h\|_{L^2(\Omega)}.\label{3.5}
	\end{align}
	The first term on the right-hand side of (\ref{3.5})  involves the factor
\begin{align}
\|\bs-\Pi_0\bs\|_{L^2(\Omega)}&\leq \|\bk\nabla u- \Pi_0(\bk\nabla u)\|_{L^2(\Omega)}+\|\bb u-\Pi_0(\bb u)\|_{L^2(\Omega)}\nonumber\\&\leq\|(\bk-\ol{\bk})\nabla u+\ol{\bk}(1-\Pi_0)\nabla u\|_{L^2(\Omega)}+\|(\bb-\ol{\bb})u+\ol{\bb}(1-\Pi_0)u\|_{L^2(\Omega)}\nonumber\\
&\leq\Big( h_{\text{max}}(|\bk|_{1,\infty}+|\bb|_{1,\infty})+C_{\text{apx}}(h_{\text{max}}^\sigma\|\bk\|_\infty+h_{\text{max}}\|\bb\|_\infty)\Big)\;\|u\|_{1+\sigma,\Omega}\nonumber.
\end{align}
The last inequality follows from the Lipschitz continuity of the coefficients $\bk$ and $\bb$, and the estimate \eqref{25.1}. Lemma~\ref{lem3.2}.d leads to the estimates  $\|\nabla_\pw(1-J)v_h\|_{L^2(\Omega)}\leq C_J|v_h|_{1,\pw}$  and 
\begin{align*}\|h_\T(1-\Pi_1)(f-\gamma u)\|_{L^2(\Omega)}\|h_{\T}^{-1}(1-J)v_h\|_{L^2(\Omega)}\leq\mathrm{osc}_1(f-\gamma u,\T)C_J|v_h|_{1,\pw}.
\end{align*}
	The substitution of the previous estimates in (\ref{3.5}) with $h_{\mathrm{max}}\leq 1$ (from $\delta\leq 1$ by assumption) and the regularity  (\ref{6}) show
	\begin{align*}
	B_\pw(u,v_h)-(f,v_h)\leq C_{\NC}h_{\text{max}}^{\sigma}\|f\|_{L^2(\Omega)}\|v_h\|_{1,\pw}
	\end{align*}
	with $C_{\NC}:=C_J\Big(( |\bk|_{1,\infty}+|\bb|_{1,\infty}+C_{\text{apx}}(\|\bk\|_\infty+\|\bb\|_\infty)+\|\gamma\|_\infty)C_{\text{reg}}+1\Big)$. This concludes the proof of Lemma~\ref{l2}.a.
	\end{proof}
	\begin{proof}
	[Proof of $(b)$]
  The  solution $\Phi\in V$ to (\ref{5}) satisfies
	$B(v,\Phi)=(g,v)_{L^2(\Omega)}$ for all $v\in V.$
	This implies
	\[ B_\pw(v_h,\Phi)-(g,v_h)_{L^2(\Omega)}= B_\pw(v_h-Jv_h,\Phi)-(g,v_h-Jv_h)_{L^2(\Omega)}.\]
	The arguments in the proof of $(a)$ lead to the bound (\ref{3.10}) with\[ C^*_{\NC}:=C_J\Big((|\bk|_{1,\infty}+C_{\text{apx}}\|\bk\|_{\infty}+\|\bb\|_{\infty}+\|\gamma\|_{\infty})C^*_{\text{reg}}+1\Big).\]
	The remaining analogous details are omitted in  the proof of Lemma~\ref{l2}.b for brevity.
\end{proof}
\section{A priori error analysis}
This section focuses on the stability, existence, and uniqueness of the discrete solution $u_h$. The \textit{a priori} error analysis uses the discrete inf-sup condition. 
\subsection{Existence and uniqueness of the discrete solution}
\begin{thm}[stability]\label{thm2.6}
There exist  positive constants $\delta\leq 1$ and $C_{\mathrm{stab}}$ (depending on $\alpha, \beta, \sigma, \rho,$ and $C_\F$) such that, for all $\T\in\mathbb{T}(\delta)$ and for all $f\in L^2(\Omega)$, the discrete problem (\ref{19}) has a unique solution $u_h\in V_h$ and 
	\begin{align*}
	|u_h|_{1,\pw}\leq C_{\mathrm{stab}}\|f_h\|_{V_h^{*}}.
	\end{align*} 
\end{thm}
\begin{proof}
	 In the first part of the proof, suppose there exists some solution $u_h\in V_h$ to the  discrete problem (\ref{19}) for some $f\in L^2(\Omega)$. (This is certainly true for all $f\equiv 0 \equiv u_h$, but will be discussed for all those pairs at the end of the proof and shall lead to the uniqueness of discrete solutions.) Since $u_h$ satisfies a G$\mathring{a}$rding-type inequality in Proposition~\ref{prop2.4}.b, 
	\begin{align}
	\alpha|u_h|_{1,\pw}^2&\leq\beta\|u_h\|^2_{L^2(\Omega)}+B_h(u_h,u_h) = \beta\|u_h\|^2_{L^2(\Omega)}+(f_h,u_h)_{L^2(\Omega)}.\nonumber
	\end{align}
	This, (\ref{13.1a}), and the definition of the dual norm  in (\ref{26}) lead to
	\begin{align}
	\alpha|u_h|_{1,\pw}\leq\beta C_{\text{F}}\|u_h\|_{L^2(\Omega)}+\|f_h\|_{V_h^{*}}.\label{38}
	\end{align}
	Given $g:=u_h\in L^2(\Omega)$, let $\Phi\in V\cap H^{1+\sigma}(\Omega)$ solve the dual problem ${\cal L}^*\Phi=g$ and let $I_h\Phi\in V_h$ be the interpolation of $\Phi$ from Subsection~2.4. Elementary algebra shows
	\begin{align}
	\|u_h\|^2_{L^2(\Omega)}&=\Big((g,u_h)_{L^2(\Omega)}-B_{\pw}(u_h,\Phi)\Big)+B_{\pw}(u_h,\Phi-I_h\Phi)\nonumber\\&\quad+\Big(B_{\pw}(u_h,I_h\Phi)-B_h(u_h,I_h\Phi)\Big)+(f_h,I_h\Phi)_{L^2(\Omega)}.\label{40}
	\end{align}
	 Rewrite a part of the third term corresponding to diffusion  on the right-hand side of (\ref{40}) as
	\begin{align}
	&a^P(u_h,I_h\Phi)-a_h^P(u_h,I_h\Phi) =(\bk\nabla u_h,\nabla(1-\Pi_1)I_h\Phi)_{L^2(P)}  \nonumber\\
	&+(\nabla(1-\Pi_1) u_h,(\bk-\po\bk)(\nabla\Pi_1 I_h\Phi))_{L^2(P)}
	-S^P\big((1-\Pi_1)u_h,(1-\Pi_1)I_h\Phi\big).\nonumber
	\end{align}
	The Cauchy-Schwarz inequality in the semi-scalar product $S^P(\bullet,\bullet)$, and (\ref{13}) with the upper bound $\|\bk\|_\infty$ for the coefficient $\bk$ in $a^P(\bullet,\bullet)$ lead to the estimate
	\begin{align}
	C_s^{-1}&S^P\big((1-\Pi_1)u_h,(1-\Pi_1)I_h\Phi\big)
	\leq   |(1-\Pi_1)u_h|_{1,P}|(1-\Pi_1)I_h\Phi|_{1,P}\nonumber\\
	&\qquad\leq \|\bk\|_{\infty}|u_h|_{1,P}\Big(\|\nabla(I_h\Phi-\Phi)\|_{L^2(P)}+\|\nabla(1-\Pi_1I_h)\Phi\|_{L^2(P)}\Big)\nonumber\\
	&\qquad\leq \|\bk\|_{\infty}C_{\text{apx}} \Big(2+C_{\PF}+C_{\text{Itn}}\Big)h_P^\sigma|u_h|_{1,P}|\Phi|_{1+\sigma,P}\label{stab}
	\end{align}
	with Theorem~\ref{thm2.7}.b followed by  \eqref{25.1} in the final step. This and Theorem~\ref{thm2.7}  imply  that
	\begin{align}
	|a^P(u_h,I_h\Phi)-a_h^P(u_h,I_h\Phi)|&\leq h_P^\sigma|u_h|_{1,P}\|\Phi\|_{1+\sigma,P}\nonumber\\&\quad\times\Big(\|\bk\|_\infty C_{\text{apx}}(2+C_{\PF}+C_{\text{Itn}})(1+C_s)+|\bk|_{1,\infty} C_{\text{Itn}}\Big).\nonumber
	\end{align}
	The terms $b^P-b_h^P$ and $c^P-c_h^P$ are controlled by
	\begin{align}
	&|b^P(u_h,I_h\Phi)-b_h^P(u_h,I_h\Phi)|+|c^P(u_h,I_h\Phi)-c_h^P(u_h,I_h\Phi)|\nonumber\\&\leq h_P^{\sigma}\|\Phi\|_{1+\sigma,P}\big(\|\bb\|_{\infty}(C_{\text{apx}}(2+C_{\PF}+C_{\text{Itn}})\|u_h\|_{L^2(P)}+C_{\mathrm{Itn}}C_{\PF}|u_h|_{1,P})\nonumber\\&\hspace{2.5cm}+\|\gamma\|_{\infty}C_{\PF}(C_{\text{Itn}}\|u_h\|_{L^2(P)}+|u_h|_{1,P})\big).\nonumber
	\end{align}
	 The combination of the previous four displayed estimates with Lemma~\ref{lem2.5} leads to an estimate for $P$. The sum  over all polygonal domains $P \in \mathcal{T} $ reads 
	\begin{align}
	 B_\pw(u_h,I_h\Phi)-B_h(u_h,I_h\Phi)\leq C_dh_{\text{max}}^{\sigma}|u_h|_{1,\pw}\|\Phi\|_{1+\sigma,\Omega}\label{43}
	\end{align}
	 with  a universal constant $C_d$. The bound for \eqref{40} results from Lemma~\ref{l2}.b for the first term, the boundedness of $B_{\pw}$ (with a universal constant $M_b:=\|\bk\|_{\infty}+C_{\F}\|\bb\|_{\infty}+C_{\F}^2\|\gamma\|_{\infty}$) and (\ref{25.2}) for the second term, (\ref{43}) for the third term,  and Theorem~\ref{thm2.7}.a for the last term on the right-hand side of (\ref{40}). This shows
	\begin{align}
	\|u_h\|^2_{L^2(\Omega)}&\leq \Big(C^*_{\NC}+C_\text{I}M_b+C_d\Big)h_{\text{max}}^{\sigma}|u_h|_{1,\pw}\|\Phi\|_{1+\sigma,\Omega}+C_{\text{Itn}}\|f_h\|_{V_h^{*}}\|\Phi\|_{1,\Omega}.\nonumber
	\end{align}
	This and the regularity estimate (\ref{6}) lead to $C_3=C^*_{\NC}+C_\text{I}M_b+C_d$  in
	\begin{align}
	\|u_h\|_{L^2(\Omega)}\leq C_3\,C^*_{\text{reg}}h_{\text{max}}^\sigma|u_h|_{1,\pw}+C_{\text{Itn}}\|f_h\|_{V_h^{*}}.\nonumber%
	\end{align}
	The substitution of this  in (\ref{38}) proves
	\begin{align}
	\alpha|u_h|_{1,\pw}\leq \beta C_{\text{F}}C_3C^*_{\text{reg}}h_{\text{max}}^\sigma|u_h|_{1,\pw}+(\beta C_{\text{F}}C_{\text{Itn}}+1)\|f_h\|_{V_h^{*}}.\label{47}
	\end{align}
	For all $0<h_{\text{max}}\leq\delta:=(\frac{\alpha}{2\beta C_{\text{F}}C_3C^*_{\text{reg}}})^{1/\sigma}$, the constant $\overline{c}=(1-\frac{\beta}{\alpha}C_{\text{F}}C_3C^*_{\text{reg}}h_{\text{max}}^\sigma)$ is positive and $C_{\text{stab}}:=\frac{\beta C_{\text{F}}C_{\text{Itn}}+1}{\alpha-\beta C_{\F}C_3C^*_{\text{reg}}h_0^{\sigma}}$  is well-defined.
	This leads in \eqref{47} to 
	\begin{align}
	|u_h|_{1,\pw}\leq C_{\text{stab}}\|f_h\|_{V_h^{*}}.\label{49}
	\end{align}
	In the last part of the proof, suppose $f_h\equiv 0$ and let $u_h$ be any solution to the resulting homogeneous linear discrete system. The stability result (\ref{49}) proves $u_h\equiv 0$. Hence, the linear system of equations (\ref{19}) has a unique solution and   the coefficient matrix is regular. This proves that there exists a unique solution $u_h$ to (\ref{19}) for any right-hand side $f_h\in V_h^*$. The combination of this with \eqref{49} concludes the proof.
\end{proof}
An immediate consequence of  Theorem~\ref{thm2.6} is the following discrete inf-sup estimate.
\begin{thm}[discrete inf-sup]
	There exist $0<\delta\leq 1$ and $\overline{\beta}_0>0$ such that, for all $\T\in\mathbb{T}(\delta)$,
	\begin{align}
	\overline{\beta}_0\leq\inf_{0\neq u_h\in V_h}\sup_{0\neq v_h\in V_h}\frac{B_h(u_h,v_h)}{|u_h|_{1,\pw}|v_h|_{1,\pw}}.\label{50}
	\end{align}
\end{thm}
\begin{proof}
	Define the operator ${\cal L}_h:V_h\to V_h^* ,$  $v_h\mapsto B_h(v_h,\bullet)$. The stability Theorem~\ref{thm2.6} can be interpreted as follows: For any $f_h\in V_h^*$ there exists $u_h\in V_h$ such that ${\cal L}_hu_h=f_h$ and
	\begin{align*}
	\overline{\beta}_0|u_h|_{1,\pw}\leq \|f_h\|_{V_h^{*}}=\sup_{0\neq v_h\in V_h}\frac{(f_h,v_h)}{|v_h|_{1,\pw}}=\sup_{0\neq v_h\in V_h}\frac{B_h(u_h,v_h)}{|v_h|_{1,\pw}}.
	\end{align*}
	The discrete problem $B_h(u_h,\bullet)=(f_h,\bullet)$ has a unique solution in $V_h$. Therefore,  $f_h$ and $u_h$ are in one to one correspondence and  the last displayed estimate holds for any $u_h\in V_h$. The infimum over $u_h\in V_h$ therein proves (\ref{50}) with $\overline{\beta}_0=C_{\text{stab}}^{-1}$.
\end{proof}
\subsection{ \textit{A priori} error estimates}
This subsection establishes the error estimate in the energy norm $|\cdot|_{1,\pw}$ and  in the $L^2$ norm. The discrete inf-sup condition allows for an error estimate in the $H^1$ norm  and an Aubin-Nitsche duality argument leads to an error estimate in the $L^2$ norm.\par Recall $u\in H^1_0(\Omega)$ is a unique solution of  \eqref{9} and $u_h\in V_h$ is a unique solution of \eqref{19}.  Recall the definition of the bilinear form $s_h(\cdot,\cdot)$ from Section~3.1 and define  the induced seminorm  $|v_h|_\s:= s_h(v_h,v_h)^{1/2}$ for $v_h\in V_h$ as a part of the norm $\|\cdot\|_h$ from Remark~\ref{rem6}.  

\begin{thm}[error estimate]\label{h1}
	  Set $\bm{\sigma}:=\bk\nabla u+\bb u\in H(\text{div},\Omega)$. There exist  positive constants $C_4, C_5,$ and $\delta$ such that, for all $\T\in\mathbb{T}(\delta)$, the discrete problem (\ref{19}) has a unique solution $u_h\in V_h$ and
	\begin{align}
 &|u-u_h|_{1,\pw}+|u-\Pi_1u_h|_{1,\pw}+h_{\mathrm{max}}^{-\sigma}(\|u-u_h\|_{L^2(\Omega)}+\|u-\Pi_1u_h\|_{L^2(\Omega)})+|u_h|_\s+|I_hu-u_h|_\s\nonumber\\&\quad\leq C_4\Big( \|(1-\Pi_0) \bm{\sigma}\|_{L^2(\Omega)}+\|(1-\Pi_0)\nabla u\|_{L^2(\Omega)}+\mathrm{osc}_1(f-\gamma u,\T)\Big)\leq C_5h_{\mathrm{max}}^\sigma\|f\|_{L^2(\Omega)}.\label{err}
	\end{align}
\end{thm}
\begin{proof}
Step 1 (initialization). Let $I_hu \in V_h$ be the interpolation of $u$ from  Definition~\ref{2.3}. 
	The discrete inf-sup condition (\ref{50}) for $I_hu-u_h\in V_h$  leads to some $v_h\in V_h$ with $|v_h|_{1,\pw}\leq 1$  such that
	\begin{align}
	\overline{\beta}_0|I_hu-u_h|_{1,\pw}= B_h(I_hu-u_h,v_h).\nonumber
	\end{align}
	\noindent
	Step 2 (error estimate for $|u-u_h|_{1,\pw}$). Rewrite the last equation with the continuous  and the discrete problem (\ref{9}) and (\ref{19}) as
	\begin{align}
	\overline{\beta}_0|I_hu-u_h|_{1,\pw}=B_h(I_hu,v_h)-B(u,v)+(f,v)_{L^2(\Omega)}-(f_h,v_h)_{L^2(\Omega)}.\nonumber
	\end{align}
	 This equality is rewritten with the definition of $B(u,v)$ in  (\ref{8}), the definition of $B_h(I_hu,v_h)$ in Section~3.1, and with $f_h=\Pi_1f$. Recall $v:=Jv_h\in V$ from  Lemma~\ref{lem3.2} and recall $\nabla_{\pw}\Pi_1 I_hu=\Pi_0\nabla u$ from (\ref{i1}). This results in
	\begin{align}
	 \text{LHS} &:= \overline{\beta}_0|I_hu-u_h|_{1,\pw}-s_h((1-\Pi_1)I_hu,(1-\Pi_1)v_h)\nonumber\\&=(\bk\Pi_0\nabla u+\bb\Pi_1I_hu,\nabla_{\pw}\Pi_1 v_h)_{L^2(\Omega)}+(\gamma\Pi_1I_hu,\Pi_1v_h)_{L^2(\Omega)}-(\bm{\sigma},\nabla v)_{L^2(\Omega)}\nonumber\\
	 &\qquad+(f-\gamma u,v)_{L^2(\Omega)}-(f,\Pi_1v_h)_{L^2(\Omega)}.\nonumber
	 \end{align}
	  Abbreviate $w:=v-\Pi_1v_h$ and observe the orthogonalities $\nabla_{\pw} w\perp\p_0(\T;\mathbb{R}^2)$ in $L^2(\Omega;\mathbb{R}^2)$ and $w\perp\p_1(\T)$ in $L^2(\Omega)$ from Lemma~\ref{lem3.2}.b-c and the definition of $\Pi_1$ with $\Pi_1=\pid_1$ in $V_h$.  Lemma~\ref{lem3.2}.d,  the bound $|(1-\pid_1)v_h|_{1,\pw}\leq |v_h|_{1,\pw}\leq 1$, and the Poincar\'e-Friedrichs inequality for $v_h-\pid_1v_h$ from Lemma~\ref{2.4b}.a  lead to
	 \begin{align}|w|_{1,\pw}&\leq|v-v_h|_{1,\pw}+|v_h-\Pi_1v_h|_{1,\pw}\leq C_\mathrm{J}+1,\label{w1}\\\|h_{\T}^{-1}w\|_{L^2(\Omega)}&\leq \|h_{\T}^{-1}(v-v_h)\|_{L^2(\Omega)}+\|h_{\T}^{-1}(v_h-\Pi_1v_h)\|_{L^2(\Omega)}\leq C_\mathrm{J}+C_\PF.\label{w2}\end{align} 
	   Elementary algebra and the above orthogonalities  prove that 
	 \begin{align}
	  \text{LHS}
	  &=((\bk-\Pi_0\bk)(\Pi_0-1)\nabla u+\bb(\Pi_1I_hu-u),\nabla_{\pw}\Pi_1 v_h)_{L^2(\Omega)}-((1-\Pi_0)\bm{\sigma},\nabla_{\pw}w)_{L^2(\Omega)}\nonumber\\&\qquad+(\gamma(\Pi_1I_hu-u),\Pi_1v_h)_{L^2(\Omega)}+(h_\T(1-\Pi_1)(f-\gamma u),h_{\T}^{-1}w)_{L^2(\Omega)}\nonumber\\
	  &
	  \leq  \Big(|\bk|_{1,\infty}+(1+C_{\PF})(\|\bb\|_{\infty}+C_{\F}\|\gamma\|_{\infty})\Big)h_{\text{max}}\|(1-\Pi_0)\nabla u\|_{L^2(\Omega)}\nonumber\\&\quad+(C_\mathrm{J}+1)\|(1-\Pi_0) \bm{\sigma}\|_{L^2(\Omega)}+(C_\mathrm{J}+C_{\text{PF}})\mathrm{osc}_1(f-\gamma u,\T)\label{3.35}
	 \end{align}
	 with the Lipschitz continuity of $\bk$, Lemma~\ref{thm2.7}.b,  the  stabilities of $\Pi_1$ from \eqref{s1}, and \eqref{w1}-\eqref{w2} in the last step. The definition of stability term (\ref{13}) and Theorem~\ref{thm2.7}.b lead to
	 \begin{align}
	 C_s^{-1}s_h((1-\Pi_1)I_hu,(1-\Pi_1)v_h)&\leq \|\bk\|_{\infty}|(1-\Pi_1)I_hu|_{1,\pw}|(1-\Pi_1)v_h|_{1,\pw}\nonumber\\&\leq \|\bk\|_{\infty} (|I_hu-u|_{1,\pw}+|u-\Pi_1I_hu|_{1,\pw})|v_h|_{1,\pw}\nonumber\\&\leq \|\bk\|_{\infty}(2+C_{\text{Itn}}+C_{\PF}) \|(1-\Pi_0)\nabla u\|_{L^2(\Omega)}|v_h|_{1,\pw}.\label{sb}
	 \end{align}
	 The triangle inequality, the bound \eqref{25.2} for the term $|u-I_hu|_{1,\pw}$, and  (\ref{3.35})-\eqref{sb} for the term $|I_hu-u_h|_{1,\pw}$ conclude the proof of (\ref{err}) for the term $|u-u_h|_{1,{\pw}}$.\\
Step $3$ (duality argument). To prove the bound for $u-u_h$ in the $L^2$ norm with a duality technique,
	let $g:=I_hu-u_h\in L^2(\Omega)$. The solution $\Phi\in H^1_0(\Omega)\cap H^{1+\sigma}(\Omega)$ to the dual problem (\ref{5})  satisfies the elliptic regularity (\ref{6}),
	\begin{align}
	\|\Phi\|_{1+\sigma,\Omega}\leq C^*_{\text{reg}}\|I_hu-u_h\|_{L^2(\Omega)}.\label{regl2}
	\end{align}
	Step $4$ (error estimate for $\|u-u_h\|_{L^2(\Omega)}$).
	 Let $I_h\Phi\in V_h$ be the interpolation of $\Phi$ from  Definition \ref{2.3}.  Elementary algebra reveals the identity
	\begin{align}
	\|g\|^2_{L^2(\Omega)}&=((g,g)_{L^2(\Omega)}-B_{\pw}(g,\Phi))+B_{\pw}(g,\Phi-I_h\Phi)\nonumber\\&\quad+(B_{\pw}(g,I_h\Phi)-B_h(g,I_h\Phi))+B_h(g,I_h\Phi).\label{L2}
	\end{align}
	   The bound (\ref{43}) with $g$ as the first argument shows
	   \begin{align*}
	   B_\pw(g,I_h\Phi)-B_h(g,I_h\Phi)\leq C_dh_{\text{max}}^{\sigma}|g|_{1,\pw}\|\Phi\|_{1+\sigma,\Omega}.
	   \end{align*}
 This controls the third term in \eqref{L2}, Lemma~\ref{l2}.b  controls the first term, the boundedness of $B_\pw$ and the interpolation error estimate (\ref{25.2}) control the second term on the right-hand side of (\ref{L2}). This results in
	\begin{align}
	\|I_hu-u_h\|^2_{L^2(\Omega)}\leq (C^*_{\NC}+C_\mathrm{I}M_b+C_d)h_{\mathrm{max}}^{\sigma}|g|_{1,\pw}\|\Phi\|_{1+\sigma,\Omega}+B_h(g,I_h\Phi).\label{4.15}
	\end{align}
It remains to bound  $B_h(g,I_h\Phi)$. The continuous  and the discrete problem (\ref{9}) and (\ref{19}) imply
	\begin{align}
	B_h(g,I_h\Phi)=B_h(I_hu,I_h\Phi)-B(u,\Phi)+(f,\Phi)_{L^2(\Omega)}-(f_h,I_h\Phi)_{L^2(\Omega)}.\nonumber
	\end{align}
The definition of $B_h$ and $\Pi_0$  lead to
\begin{align}
&B_h(g,I_h\Phi)-s_h((1-\Pi_1)I_hu,(1-\Pi_1)I_h\Phi)\nonumber\\
&=((\bk-\Pi_0\bk)(\Pi_0-1)\nabla u+\bb(\Pi_1 I_hu-u),\nabla_{\pw}\Pi_1I_h\Phi)_{L^2(\Omega)}+(\gamma(\Pi_1I_hu-u),\Pi_1I_h\Phi)_{L^2(\Omega)}\nonumber\\&\qquad-((1-\Pi_0)\bm{\sigma},\nabla_{\pw}(1-\Pi_1I_h)\Phi)_{L^2(\Omega)}+(f-\gamma u,\Phi-\Pi_1I_h\Phi)_{L^2(\Omega)}.\label{4.15n}
\end{align}	
The bound for the stability term as in (\ref{sb}) is
\begin{align}
s_h((1-\Pi_1)I_hu,(1-\Pi_1)I_h\Phi)&\leq C_s\|\bk\|_{\infty}|(1-\Pi_1)I_hu|_{1,\pw}|(1-\Pi_1)I_h\Phi|_{1,\pw}\nonumber\\ &\hspace{-2cm}\leq C_s\|\bk\|_{\infty}(2+C_{\mathrm{Itn}}+C_{\PF})^2 C_{\text{apx}}h_{\text{max}}^{\sigma}\|(1-\Pi_0)\nabla u\|_{L^2(\Omega)}|\Phi|_{1+\sigma,\Omega}.\label{4.17}
\end{align}
Step $5$ (oscillation). The last term in (\ref{4.15n}) is of optimal order $O(h_{\text{max}}^{1+\sigma})$, but the following arguments allow to write it as an oscillation. Recall the bubble-function $b_{\T}|_P:=b_P\in H^1_0(P)$ from (\ref{bubble}) extended by zero outside $P$. Given $\Psi:=\Phi-\Pi_1I_h\Phi$, let $\Psi_1\in\p_1(\T)$ be the Riesz representation of the linear functional $\p_1(\T)\to\R$ defined by $w_1\mapsto(\Psi,w_1)_{L^2(\Omega)}$ in the Hilbert space $\p_1(\T)$ endowed with the weighted scalar product $(b_\T\bullet,\bullet)_{L^2(\Omega)}$. That means  $\Pi_1(b_\T\Psi_1)=\Pi_1\Psi$.  The identity $(f-\gamma u,b_{\T}\Psi_1)_{L^2(\Omega)}=(\bm{\sigma},\nabla(b_\T\Psi_1))_{L^2(\Omega)}$ follows from (\ref{9}) with the test function $b_{\T}\Psi_1\in H^1_0(\Omega)$. The $L^2$ orthogonalities $\Psi-b_{\T}\Psi_1\perp\p_1(\T)$ in $L^2(\Omega)$ and $\nabla( b_{\T}\Psi_1)\perp\p_0(\T;\mathbb{R}^2)$ in $L^2(\Omega;\mathbb{R}^2)$ allow the rewriting of the latter identity as
\begin{align}
&(f-\gamma u,\Psi)_{L^2(\Omega)}=(h_\T(1-\Pi_1)(f-\gamma u),h_{\T}^{-1}(\Psi-b_{\T}\Psi_1))_{L^2(\Omega)}+((1-\Pi_0)\bm{\sigma},\nabla(b_\T\Psi_1))_{L^2(\Omega)}\nonumber\\&\qquad\leq\mathrm{osc}_1(f-\gamma u,\T)\|h_{\T}^{-1}(\Psi-b_{\T}\Psi_1)\|_{L^2(\Omega)}+\|(1-\Pi_0)\bs\|_{L^2(\Omega)}|b_\T\Psi|_{1,\pw}.\label{4.16a}
\end{align}	
It remains to control the terms  $\|h_{\T}^{-1}(\Psi-b_{\T}\Psi_1)\|_{L^2(\Omega)}$ and $|b_\T\Psi|_{1,\pw}$. Since the definition of $I_h$ and the definition of $\pid_1$ with $\Pi_1=\pid_1$ in $V_h$ imply $\int_{\partial P}\Psi\,ds=\int_{\partial P}(\Phi-\Pi_1I_h\Phi)\,ds=0$, this allows the Poincar\'e-Friedrichs inequality for $\Psi$ from Lemma~\ref{2.4b}.a on each $P\in\T$. This shows \begin{align}
\|h_\T^{-1}\Psi\|_{L^2(\Omega)}\leq C_\PF|\Psi|_{1,\pw}\leq C_\PF C_{\text{apx}} h_{\text{max}}^{\sigma} |\Phi|_{1+\sigma,\Omega}
\end{align}
 with Theorem~\ref{thm2.7}.b and (\ref{25.1}) in the last inequality. Since $b_P\Psi_1\in H^1_0(P)$ for $P\in\T$, the Poincar\'e-Friedrichs inequality from Lemma~\ref{2.4b}.a leads to
\begin{align}
\|h_P^{-1}(b_P\Psi_1)\|_{L^2(P)}\leq C_\PF|b_P\Psi_1|_{1,P}.\label{4.17a}
\end{align}
 The first estimate in \eqref{b1}, the identity $\Pi_1(b_\T\Psi_1)=\Pi_1\Psi$, and  the Cauchy-Schwarz inequality imply 
\begin{align*}
C_b^{-1}\|h_{P}^{-1}\Psi_1\|_{L^2(P)}^2\leq \|h_{P}^{-1}b_{P}^{1/2}\Psi_1\|_{L^2(P)}^2=(h_{P}^{-1}\Psi_1,h_{P}^{-1}\Psi)_{L^2(P)}\leq \|h_{P}^{-1}\Psi_1\|_{L^2(P)}\|h_{P}^{-1}\Psi\|_{L^2(P)}.
\end{align*}
  This proves $\|h_{P}^{-1}\Psi_1\|_{L^2(P)}\leq C_b \|h_{P}^{-1}\Psi\|_{L^2(P)}$. The second estimate in \eqref{b2} followed by the first estimate in \eqref{b1} leads to the first inequality and the arguments as above lead to the second inequality in
\begin{align}
C_{b}^{-3/2} |b_P\Psi_1|_{1,P}\leq\|h_{P}^{-1}b_{P}^{1/2}\Psi_1\|_{L^2(P)}\leq\|h_{P}^{-1}\Psi_1\|_{L^2(P)}^{1/2}\|h_{P}^{-1}\Psi\|_{L^2(P)}^{1/2}&\leq C_{b}^{1/2}\|h_{P}^{-1}\Psi\|_{L^2(P)}\nonumber
\end{align}
with $\|h_{P}^{-1}\Psi_1\|_{L^2(P)}^{1/2}\leq C_b^{1/2} \|h_{P}^{-1}\Psi\|_{L^2(P)}^{1/2}$  from above in the last step.    The combination of the previous displayed estimate and \eqref{4.16a}-\eqref{4.17a} results with $C_6 := C_\PF C_{\text{apx}}(1+C_b^2(1+C_\PF))$ in
\begin{align}
(f-\gamma u,\Psi)_{L^2(\Omega)}\leq C_6 (\mathrm{osc}_1(f-\gamma u,\T)+\|(1-\Pi_0)\bs\|_{L^2(\Omega)})h_{\text{max}}^{\sigma}|\Phi|_{1+\sigma,\Omega}.\label{4.19}
\end{align}
 Step $6$ (continued proof of estimate for $\|u-u_h\|_{L^2(\Omega)}$). The estimate in Step 2 for $|g|_{1,\pw}$,  (\ref{4.15})-(\ref{4.17}), and (\ref{4.19}) with the regularity (\ref{regl2})  show
\begin{align}
\|I_hu-u_h\|_{L^2(\Omega)}\lesssim h_{\mathrm{max}}^{\sigma}\Big(\|(1-\Pi_0)\nabla u\|_{L^2(\Omega)}+\|(1-\Pi_0)\bm{\sigma}\|_{L^2(\Omega)}+\mathrm{osc}_1(f-\gamma u,\T)\Big).\label{4.18}
\end{align}
Rewrite the difference $u-u_h= (u-I_hu)+(I_hu-u_h)$, and apply the triangle inequality with (\ref{25.2}) for the first term
\[\|u-I_hu\|_{L^2(\Omega)}\leq C_\mathrm{I}h_{\text{max}}^{1+\sigma}|u|_{1+\sigma,\Omega}.\]
 This and (\ref{4.18}) for the second term $I_hu-u_h$ conclude the proof of the estimate  for the term $h_{\text{max}}^{-\sigma}\|u-u_h\|_{L^2(\Omega)}$ in  \eqref{err} .\\
 Step $7$ (stabilisation error $|u_h|_\s$ and $|I_hu-u_h|_\s$).  The triangle inequality and the upper bound of the stability term \eqref{13} lead to
 \begin{align*}
 |u_h|_\s\leq |I_hu-u_h|_\s+|I_hu|_\s\leq C_s^{1/2}\|\bk\|_\infty^{1/2}(|I_hu-u_h|_{1,\pw}+|(1-\Pi_1)I_hu|_{1,\pw})
 \end{align*}
 with $|(1-\Pi_1)(I_hu-u_h)|_{1,\pw}\leq |I_hu-u_h|_{1,\pw}$ in the last inequality. The arguments as in \eqref{sb} prove that $|(1-\Pi_1)I_hu|_{1,\pw}\leq (2+C_{\text{Itn}}+C_\PF)\|(1-\Pi_0)\nabla u\|_{L^2(\Omega)}$. This and the arguments in Step~2 for the estimate of $|I_hu-u_h|_{1,\pw}$ show the upper bound in \eqref{err} for the terms $|u_h|_\s$ and $|I_hu-u_h|_\s$.\\
 Step $8$ (error estimate for $u-\Pi_1u_h$). The VEM solution $u_h$ is defined by the computed degrees of freedom given in \eqref{10.2}, but the evaluation of the function itself requires expansive additional calculations. The later are avoided if $u_h$ is replaced by the Ritz projection $\Pi_1u_h$ in the numerical experiments. The triangle inequality leads to
 \begin{align}
 |u-\Pi_1u_h|_{1,\pw}\leq |u-u_h|_{1,\pw}+|u_h-\Pi_1u_h|_{1,\pw}.\label{4.25}
 \end{align}
 A lower bound of the stability term (\ref{13}) and the assumption \ref{A2} imply
 \begin{align}
 |u_h-\Pi_1u_h|_{1,P}\leq a_0^{-1/2} C_s^{1/2}S^P((1-\Pi_1)u_h,(1-\Pi_1)u_h)^{1/2}.\label{5.11}
 \end{align}
  This shows that the second term in \eqref{4.25} is bounded by $|u_h|_\s$.  Hence Step 2 and Step 7 prove the estimate for $|u-\Pi_1u_h|_{1,\pw}$. Since $\int_{\partial P}(u_h-\Pi_1u_h)\,ds=0$ from the definition of $\pid_1$ and $\Pi_1=\pid_1$ in $V_h$, the combination of Poincar\'e-Friedrichs inequality for $u_h-\Pi_1u_h$ from Lemma~\ref{2.4b}.a  and (\ref{5.11}) result in
 \begin{align}
 C_{\PF}^{-1} a_0^{1/2}C_s^{-1/2}\|u_h-\Pi_1u_h\|_{L^2(P)}\leq h_PS^P((1-\Pi_1)u_h,(1-\Pi_1)u_h)^{1/2}.\label{5.19}
 \end{align}
 The analogous arguments for $\|u-\Pi_1u_h\|_{L^2(\Omega)}$, \eqref{5.19}, and the  estimate for $|u_h|_\s$ prove the bound \eqref{err} for the term $h_{\text{max}}^{-\sigma}\|u-\Pi_1u_h\|_{L^2(\Omega)}$. This concludes the proof of Theorem~\ref{h1}.
\end{proof} 

\section{\textit{A posteriori} error analysis} 
This section presents the reliability and efficiency of a residual-type \textit{a posteriori} error estimator. 
\subsection{Residual-based explicit \textit{a posteriori} error control}
Recall $u_h\in V_h$ is the solution to the problem \eqref{19}, and the definition of jump $[\cdot]_E$ along an edge $E\in\e$ from Section~2.
For any polygonal domain $P\in\T$, set
\begin{center}
\begin{tabular}{lr}
$\ds\eta_P^2:=h_P^2\|f-\gamma\Pi_1u_h\|_{L^2(P)}^2$&(volume residual),\\
$\ds\zeta_P^2:=S^P((1-\Pi_1)u_h,(1-\Pi_1)u_h)$&(stabilization),\\
$\ds\Lambda_P^2:=\|(1-\Pi_0)(\bk\nabla\Pi_1u_h+\bb\Pi_1u_h)\|_{L^2(P)}^2$&(inconsistency),\\
$\ds\Xi_{P}^2:=\sum_{E\in \e(P)}|E|^{-1}\|[\Pi_1u_h]_E\|_{L^2(E)}^2$&(nonconformity).
\end{tabular}
\end{center}
These local quantities $\bullet|_P$ form a family ($\bullet|_P:P\in\T$) over the index set $\T$ and their Euclid vector norm $\bullet|_\T$ enters the upper error bound: $\eta_{\T}:=(\sum_{P\in\T}\eta_P^2)^{1/2}$,  $\zeta_\T:=(\sum_{P\in\T}\zeta_P^2)^{1/2}$, $\Lambda_\T:=(\sum_{P\in\T}\Lambda_P^2)^{1/2}$, and  $\Xi_\T:=(\sum_{P\in\T}\Xi_P^2)^{1/2}$. The following theorem provides an upper bound to the error $u-u_h$ in the $H^1$ and the $L^2$ norm. Recall the elliptic regularity \eqref{6} with the index  $0<\sigma\leq 1$, and recall the assumption $h_\text{max}\leq 1$ from Subsection~2.1.
\begin{thm}[reliability]\label{5.2}
	There exist  positive constants $C_{\text{rel}1}$ and $C_{\text{rel}2}$  (both depending on $\rho$) such that
	\begin{align}
	C_{\mathrm{rel}1}^{-2}|u-u_h|_{1,\pw}^2\leq  \eta_\T^2+\zeta_\T^2+\Lambda_\T^2+\Xi_{\T}^2\label{5.22}
	\end{align}
	and
	\begin{align}
	\|u-u_h\|_{L^2(\Omega)}^2\leq C_{\mathrm{rel}2}^{2} \sum_{P\in\T}\Big(h_P^{2\sigma}(\eta_P^2+\zeta_P^2+\Lambda_P^2+\Xi_{P}^2)\Big).\label{5.22a}
	\end{align}
\end{thm}
The proof of this theorem in Subsection~5.3 relies on a  conforming companion operator elaborated in the next subsection. The upper bound in Theorem~\ref{5.2}  is efficient in the following local sense, where $\omega_E:=\textrm{int}(\cup \T(E))$ denotes the patch of an edge $E$ and consists of the one or the two neighbouring polygons in the set $\T(E):=\{P'\in\T:E\subset\partial P'\}$ that share $E$. Recall $\bs =\bk\nabla u+\bb u$ from Subsection~4.2 and the data-oscillation $\mathrm{osc}_1(f,P):= \|h_P(1-\Pi_1)f\|_{L^2(P)}$ from Subsection~2.1.
\begin{thm}[local efficiency up to oscillation]\label{efficiency}
	The quantities $\eta_P, \zeta_P, \Lambda_P,$ and $\Xi_P$ from Theorem~\ref{5.2} satisfy
	\begin{align}
	\zeta^2_P&\lesssim |u-u_h|^2_{1,P}+|u-\Pi_1u_h|^2_{1,P}\label{lb}\\
	\eta_P^2&\lesssim \|u-u_h\|^2_{1,P}+|u-\Pi_1u_h|^2_{1,P}+\|(1-\Pi_0)\bs\|_{L^2(P)}^2+\mathrm{osc}_1^2(f-\gamma u,P),\label{lb1}\\
	\Lambda_P^2&\lesssim \|u-u_h\|^2_{1,P}+|u-\Pi_1u_h|^2_{1,P}+\|(1-\Pi_0)\bs\|_{L^2(P)}^2,\label{lb3}\\
	\Xi_P^2&\lesssim \sum_{E\in\e(P)}\sum_{P'\in\omega_E}( \|u-u_h\|^2_{1,P'}+|u-\Pi_1u_h|^2_{1,P'}).\label{lb2}
	\end{align}
\end{thm}
The proof of Theorem~\ref{efficiency} follows in Subsection~5.4. The reliability and efficiency estimates in Theorem~\ref{5.2} and \ref{efficiency} lead to an equivalence up to the approximation term 
\[\text{apx} := \|\bs-\Pi_0\bs\|_{L^2(\Omega)}+\mathrm{osc}_1(f-\gamma u,\T).\]  Recall the definition of $|u_h|_\s$ from Subsection~4.2. In this paper, the norm $|\cdot|_{1,\pw}$ in the nonconforming space $V_h$ has been utilised for simplicity and one alternative is the norm $\|\cdot\|_h$ from Remark~\ref{rem6} induced by $a_h$. Then it appears natural to have the total error with the stabilisation term as 
\[\text{total error} := |u-u_h|_{1,\pw}+|u-\Pi_1u_h|_{1,\pw}+h_{\text{max}}^{-\sigma}\|u-u_h\|_{L^2(\Omega)}+h_{\text{max}}^{-\sigma}\|u-\Pi_1u_h\|_{L^2(\Omega)}+|u_h|_\s.\]
The point is that Theorem~\ref{h1} assures that total error $+$ apx converges with the expected optimal convergence rate. 
\begin{cor}[equivalence]  The $\mathrm{estimator} := \eta_\T+\zeta_\T+\Lambda_\T+\Xi_\T \approx  \mathrm{total\; error} + \mathrm{apx}$.
\end{cor}
\vspace{-0.6cm}
\begin{proof}\phantom\qedhere
	Theorem~\ref{efficiency} motivates apx and  shows
	\[\mathrm{estimator}\lesssim \|u-u_h\|_{1,\pw}+\|\bs-\Pi_0\bs\|_{L^2(\Omega)}+\mathrm{osc}_1(f-\gamma u,\T)+|u_h|_\s\leq \mathrm{total\; error} + \mathrm{apx}.\]  
	This proves the first inequality $\lesssim$ in the assertion.
	Theorem~\ref{5.2}, the estimates in Subsection~5.3.3.1, and the definition of $|u_h|_s$  show $\text{total error} \lesssim \text{estimator}$.	The first of the terms in apx is
	\[\|\bs-\Pi_0\bs\|_{L^2(\Omega)}\leq \|\bs-\Pi_0\bs_h\|_{L^2(\Omega)}\leq \|\bs-\bs_h\|_{L^2(\Omega)}+\|(1-\Pi_0)\bs_h\|_{L^2(\Omega)}.\]	
	The definition of $\bs$ and $\bs_h$ plus the triangle and the Cauchy-Schwarz inequality show 
	\begin{align*}
	\|\bs-\bs_h\|_{L^2(\Omega)}\leq \|\bk\|_\infty|u-\Pi_1u_h|_{1,\pw}+\|\bb\|_{\infty}\|u-\Pi_1u_h\|_{L^2(\Omega)}\lesssim \|u-\Pi_1u_h\|_{1,\pw}.
	\end{align*}
	The upper bound is $\lesssim$ estimator as mentioned above. Since the  term $\|(1-\Pi_0)\bs_h\|_{L^2(\Omega)}= \Lambda_\T$ is a part of the estimator, $\|(1-\Pi_0)\bs\|_{L^2(\Omega)}\lesssim \mathrm{estimator}$.  The other term in apx is
	\begin{align*}
	\mathrm{osc}_1(f-\gamma u,\T) &\leq \mathrm{osc}_1(f-\gamma\Pi_1u_h,\T)+\|h_\T\gamma(u-\Pi_1u_h)\|_{L^2(\Omega)}\\&\leq \eta_\T+\|\gamma\|_\infty h_{\text{max}}\|u-\Pi_1u_h\|_{L^2(\Omega)}\lesssim \text{estimator}. \qed
	\end{align*}
\end{proof}

Section~5 establishes the \textit{a posteriori} error analysis of the nonconforming VEM. Related results are known for the conforming VEM and the nonconforming FEM.
\begin{rem}[comparison with nonconforming FEM]
	Theorem~\ref{5.2} generalizes a result for the nonconforming FEM in \cite[Thm.~3.4]{11} from triangulations into triangles to those in polygons (recall  Example~\ref{ex}). The only difference is the extra stabilization term that can be dropped in the nonconforming FEM.
\end{rem}
\begin{rem}[comparison with conforming VEM]
	The volume residual, the inconsistency term, and the stabilization also arise in the \textit{a posteriori} error estimator   for the conforming VEM in \cite[Thm.~13]{6}. But it also  includes an additional term with normal jumps compared to the estimator \eqref{5.22}. The extra nonconformity term  in this paper  is caused by  the nonconformity $V_h\not\subset V$ in general.
\end{rem}

\subsection{Enrichment and conforming companion operator}
 The link from the nonconforming approximation $u_h\in V_h$ to a global Sobolev function in $H^1_0(\Omega)$ can be designed  with the help of the underlying refinement $\tT$ of the triangulation $\T$ (from Section~2). The interpolation $I_{\text{CR}}:V+V_h\to \textrm{CR}^1_0(\tT)$ in the Crouzeix-Raviart finite element space $\textrm{CR}^1_0(\tT)$ from Subsection~3.4 allows for a right-inverse $J'$. A  companion operator $J'\circ I_{\text{CR}}:V_h\to H^1_0(\Omega)$ acts as displayed
  \tikzstyle{line} = [draw, -latex']
 \begin{figure}[H]
 	\begin{center}
 		\begin{tikzpicture}[node distance = 1.5cm, auto]
 		\node(step1){$V_h$};
 		\node[right of =step1,node distance=4cm](step2){$\text{CR}^1_0(\tT)$};
 		\node[right of = step2,node distance=4cm](step3){$H^1_0(\Omega)$};
 		\path[line](step2)--node[above,text centered]{$J'$}(step3);
 		\path[line](step1)--node[above,text centered]{$I_{\text{CR}}$}(step2);
 		\end{tikzpicture}
 	\end{center}
 \end{figure}
 \vspace{-1cm} 
Define an  enrichment operator $E_\pw:\p_1(\tT)\to S^1_0(\tT)$ by averaging  nodal values: For any  vertex $z$ in the refined triangulation $\tT$, let $\tT(z)=\{T\in\tT:z\in T\}$ denote the set of $|\tT(z)|\geq 1$ many triangles that share the vertex $z$, and define 
\begin{align*}
E_\pw v_1(z)=\frac{1}{|\tT(z)|}\sum_{T\in \tT(z)}{v_1}|_{T}(z)
\end{align*}
for an interior  vertex $z$ (and zero for a boundary vertex $z$ according to the homogeneous boundary conditions). This defines $E_\pw v_1$ at any vertex of a triangle $T$ in $\tT$, and  linear interpolation  then defines $E_\pw v_1$ in $T\in\tT$, so that $E_\pw v_1\in S^1_0(\tT)$. Huang \textit{et al.} \cite{HUANG2021113229} design an enrichment operator by an extension of \cite{dG} to polygonal domains, while we deduce it from a sub-triangulation.  The following lemma provides an approximation property of the operator $E_\pw$.
\begin{lemma}\label{lem5.3}
	There exists a positive constant $C_{En}$ that depends only on the shape regularity of $\tT$ such that any $v_1\in\p_1(\T)$ satisfies
	\begin{align}
	\|h_{\T}^{-1}(1-E_\pw) v_1\|_{L^2(\Omega)}+	|(1-E_\pw) v_1|_{1,\pw}\leq C_\mathrm{En}\left(\sum_{E\in{\e}}|E|^{-1}\|[v_1]_E\|_{L^2(E)}^2\right)^{1/2}.\label{P2}
	\end{align}
\end{lemma}
\begin{proof}
	There exists a positive constant $C_{En}$ independent of $h$ and $v_1$  \cite[p.~2378]{dG} such that 
	\begin{align*}
	\|h_{\tT}^{-1}(1-E_\pw) v_1\|_{L^2(\Omega)}+\left(\sum_{T\in\tT}\|\nabla(1-E_\pw) v_1\|_{L^2(T)}^2\right)^{1/2}\leq C_\mathrm{En}\left(\sum_{E\in\widehat{\e}}|E|^{-1}\|[v_1]_E\|_{L^2(E)}^2\right)^{1/2}.
	\end{align*}
	Note that any edge $E\in \e$ is unrefined in the sub-triangulation $\tT$.  Since $v_{1|P}\in H^1(P)$ is continuous in each polygonal domain $P\in\T$ and $h_T\leq h_P$ for all $T\in \tT(P)$, the above inequality reduces to (\ref{P2}). This concludes  the proof.
\end{proof}
 Recall the $L^2$ projection $\Pi_1$ onto the piecewise affine functions  $\p_1(\T)$ from Section~2.  
An enrichment operator $E_\pw\circ\Pi_1:V_h\to H^1_0(\Omega)$ acts as displayed
\tikzstyle{line} = [draw, -latex']
\begin{figure}[H]
	\begin{center}
		\begin{tikzpicture}[node distance = 1.5cm, auto]
		\node(step1){$V_h$};
		\node[right of =step1,node distance=4cm](step2){$\p_1(\T)\hookrightarrow\p_1(\tT)$};
		\node[right of = step2,node distance=4cm](step3){$H^1_0(\Omega)$};
		\path[line](step2)--node[above,text centered]{$E_\pw$}(step3);
		\path[line](step1)--node[above,text centered]{$\Pi_1$}(step2);
		\end{tikzpicture}
	\end{center}
\end{figure}
\vspace{-1cm}
 
\subsection{Proof of Theorem~\ref{5.2}}
 \subsubsection{Reliable $H^1$ error control}
 Define $E_1u_h:=E_\pw\Pi_1u_h\in H^1_0(\Omega)$ so that $u-E_1u_h \in H^1_0(\Omega)$. The inf-sup condition (\ref{9.1}) leads to some $v\in H^1_0(\Omega)$ with $\|v\|_{1,\Omega}\leq 1$ and
	\begin{align}
	\beta_0\|u-E_1u_h\|_{1,\Omega}= B(u-E_1u_h,v)=((f,v)_{L^2(\Omega)}-B_{\pw}(\Pi_1u_h,v))+B_{\pw}(\Pi_1u_h-E_1u_h,v)\label{4.4}
	\end{align}
	with $B(u,v)=(f,v)$ from \eqref{9} and the piecewise version $B_\pw$   of $B$ in the last step.  The definition of $B_h$ from Subsection~3.1 and the discrete problem (\ref{19})  with $v_h= I_hv$ imply
	\begin{align}
	B_{\pw}(\Pi_1u_h,\Pi_1I_hv)+s_h((1-\Pi_1)u_h,(1-\Pi_1)I_hv)=B_h(u_h,I_hv)=(f,\Pi_1I_hv)_{L^2(\Omega)}.\label{rem}
	\end{align}
Abbreviate $w:=v-\Pi_1I_hv$ and $\bm{\sigma}_h:=\bk\nabla_{\pw}\Pi_1u_h+\bb\Pi_1u_h$.  This and \eqref{rem} simplify 
\begin{align} 
&(f,v)_{L^2(\Omega)}-B_{\pw}(\Pi_1u_h,v)= (f,w)_{L^2(\Omega)}-B_{\pw}(\Pi_1u_h,w)+s_h((1-\Pi_1)u_h,(1-\Pi_1)I_hv)\nonumber\\&=(f-\gamma\Pi_1u_h,w)_{L^2(\Omega)}-((1-\Pi_0)\bm{\sigma}_h,\nabla_{\pw}w)_{L^2(\Omega)}+s_h((1-\Pi_1)u_h,(1-\Pi_1)I_hv)\label{5.8}
\end{align}	
with  $\int_P\nabla w\,dx=0$ for any $P\in\T$  from  (\ref{i1}) in the last step. Recall the notation  $\eta_P, \Lambda_P$, and $\zeta_P$ from Subsection~5.1. The Cauchy-Schwarz inequality and Theorem~\ref{thm2.7}.b followed by $\|(1-\Pi_0)\nabla v\|_{L^2(\Omega)}\leq |v|_{1,\Omega}\leq 1$ in the second step show
\begin{align}
(f-\gamma\Pi_1u_h,w)_{L^2(P)}&\leq \eta_Ph_P^{-1}\|w\|_{L^2(P)}\leq (1+C_\PF)\eta_P,\label{5.10a}\\
((1-\Pi_0)\bm{\sigma}_h,\nabla w)_{L^2(P)}&\leq\Lambda_P|w|_{1,P} \leq (1+C_\PF)\Lambda_P.\label{5.10b}
\end{align}
The upper bound $\|\bk\|_\infty$ of the coefficient $\bk$, (\ref{13}), and the Cauchy-Schwarz inequality for the stabilization term lead to the first inequality in
\begin{align}
C_s^{-1/2}S^P((1-\Pi_1)u_h,(1-\Pi_1)I_hv)&\leq  \|\bk\|_\infty^{1/2}S^P((1-\Pi_1)u_h,(1-\Pi_1)u_h)^{1/2} |(1-\Pi_1)I_hv|_{1,P} \nonumber\\&\leq \|\bk\|_\infty^{1/2}(2+C_{\PF}+C_{\text{Itn}})\zeta_P.\label{5.14}
\end{align}
The second inequality in (\ref{5.14}) follows as in (\ref{stab}) and with  $\|(1-\Pi_0)\nabla v\|_{L^2(P)}\leq 1$. Recall  the  boundedness constant $M_b$ of $B_\pw$ from Subsection~4.1 and deduce from (\ref{P2}) and the  definition of $\Xi_\T$ from Subsection~5.1 that \begin{align} B_{\pw}(\Pi_1u_h-E_1u_h,v)\leq M_b|\Pi_1u_h-E_1u_h|_{1,\pw}\leq M_bC_\mathrm{En}\Xi_\T. \label{5.10c}\end{align}   The substitution  of \eqref{5.8}-\eqref{5.10c} in  \eqref{4.4}  reveals that
\begin{align}
\|u-E_1u_h\|_{1,\Omega}
&\leq C_7(\eta_\T+\Lambda_\T+\zeta_\T+\Xi_\T)\label{5.12}
\end{align}
with $\beta_0C_7=1+C_{\PF}+C_s^{1/2}\|\bk\|_\infty^{1/2}(2+C_{\PF}+C_{\text{Itn}})+M_bC_\mathrm{En}.$
  
The combination of (\ref{5.11}), \eqref{5.12} and (\ref{P2}) leads in  the triangle inequality
 \begin{align*}
|u-u_h|_{1,\pw}\leq|u-E_1u_h|_{1,\Omega}+|E_1u_h-\Pi_1u_h|_{1,\pw}+|\Pi_1u_h-u_h|_{1,\pw}
\end{align*}  to (\ref{5.22}) with   $C_{\text{rel}1}/2=C_7+C_\mathrm{En}+a_0^{-1/2}C_s^{1/2}$. \qed
\subsubsection{Reliable $L^2$ error control}
 Recall $I_{\text{CR}}$ from \eqref{CR} and $J'$ from the proof of Lemma~\ref{lem3.2}, and define $E_2u_h:=J'I_{\text{CR}}u_h\in H^1_0(\Omega)$ from Subsection~5.2. Let $\Psi\in H^1_0(\Omega)\cap H^{1+\sigma}(\Omega)$ solve  the dual problem $B(v,\Psi)=(u-E_2u_h,v)$ for all $v\in V$ and recall (from \eqref{6}) the regularity estimate
\begin{align} 
\|\Psi\|_{1+\sigma,\Omega}\leq C^*_{\text{reg}}\|u-E_2u_h\|_{L^2(\Omega)}.\label{reg}
\end{align}
  The substitution  of $v:=u-E_2u_h\in V$ in the dual problem  shows
\begin{align*}
\|u-E_2u_h\|^2_{L^2(\Omega)}=B(u-E_2u_h,\Psi).
\end{align*}
 The algebra in (\ref{4.4})-(\ref{5.8}) above leads with $v=\Psi$  to the identity
\begin{align}
&\|u-E_2u_h\|^2_{L^2(\Omega)}-s_h((1-\Pi_1)u_h,(1-\Pi_1)I_h\Psi)=(f-\gamma\Pi_1u_h,\Psi-\Pi_1I_h\Psi)_{L^2(\Omega)}\nonumber\\
&\quad\quad
-((1-\Pi_0)\bs_h,\nabla_\pw(\Psi-\Pi_1I_h\Psi))_{L^2(\Omega)}
+B_\pw(\Pi_1u_h-E_2u_h,\Psi).\label{5.15}
\end{align} 
The definition of $I_{\text{CR}}$ and $J'$ proves  the first and second equality in \[\int_{E}u_h\,ds=\int_{E}I_{\text{CR}}u_h\,ds=\int_{E}E_2u_h\,ds\quad\text{for all}\;E\in\e.\] This and an integration by parts imply $\int_P\nabla(u_h-E_2u_h)\,dx=0$ for all $P\in\T$. Hence Definition~\ref{def1} of  Ritz projection $\pid_1=\Pi_1$ in $V_h$ shows $\int_{P}\nabla(\Pi_1u_h-E_2u_h)\,ds=0$ for all $P\in\T$. This $L^2$ orthogonality $\nabla_\pw(\Pi_1u_h-E_2u_h)\perp \p_0(\T;\mathbb{R}^2)$ and the definition of $B_\pw$ in the last term of \eqref{5.15} result with elementary algebra  in
\begin{align}
&B_\pw(\Pi_1u_h-E_2u_h,\Psi)=((\bk-\Pi_0\bk)\nabla_\pw(\Pi_1u_h-E_2u_h),\nabla\Psi)_{L^2(\Omega)}\nonumber\\
&\;+(\nabla_\pw(\Pi_1u_h-E_2u_h),(\Pi_0\bk)(1-\Pi_0)\nabla\Psi)_{L^2(\Omega)}+(\Pi_1u_h-E_2u_h,\bb\cdot\nabla\Psi+\gamma\Psi)_{L^2(\Omega)}.\label{5.15a}
\end{align}
 The triangle inequality and \ref{l33} from  the proof of Lemma~\ref{lem3.2} imply the first inequality in  
\begin{align}
|\Pi_1u_h-E_2u_h|_{1,\pw}&\leq |\Pi_1u_h-I_{\text{CR}}u_h|_{1,\pw}+C_\mathrm{J'}\min_{v\in V}|I_{\text{CR}}u_h-v|_{1,\pw}\nonumber\\&\leq |\Pi_1u_h-I_{\text{CR}}u_h|_{1,\pw}+C_\mathrm{J'}|I_{\text{CR}}u_h-E_1u_h|_{1,\pw}\nonumber\\&\leq|\Pi_1u_h-I_{\text{CR}}u_h|_{1,\pw}+C_\mathrm{J'}(|I_{\text{CR}}u_h-\Pi_1u_h|_{1,\pw}+|\Pi_1u_h-E_1u_h|_{1,\pw})\nonumber\\&\leq (1+C_\mathrm{J'})|u_h-\Pi_1u_h|_{1,\pw}+C_\mathrm{J'}|\Pi_1u_h-E_1u_h|_{1,\pw}.\label{5.15b}
\end{align}
 The second estimate in \eqref{5.15b} follows from $E_1u_h\in V$,  the third is a triangle inequality, and eventually $|\Pi_1u_h-I_{\text{CR}}u_h|_{1,\pw}\leq|u_h-\Pi_1u_h|_{1,\pw}$ results from the orthogonality $\nabla_\pw(u_h-I_{\text{CR}})\perp \p_0(\tT;\mathbb{R}^2)$ and $\Pi_1u_h\in\p_1(\T)$. The Cauchy-Schwarz inequality, the Lipschitz continuity of $\bk$, and the approximation estimate  $\|(1-\Pi_0)\nabla\Psi\|_{L^2(P)}\leq C_{\text{apx}}h_P^\sigma|\Psi|_{1+\sigma,P}$    in \eqref{5.15a} lead to the first inequality   in
\begin{align}
B_\pw(\Pi_1u_h-E_2u_h,\Psi)&\leq\sum_{P\in\T}\Big((h_P|\bk|_{1,\infty}+\|\bk\|_\infty C_{\text{apx}}h_P^\sigma)|\Pi_1u_h-E_2u_h|_{1,P}\nonumber\\&\qquad+\|\Pi_1u_h-E_2u_h\|_{L^2(P)}(\|\bb\|_\infty+\|\gamma\|_\infty)\Big)\|\Psi\|_{1+\sigma,P}\nonumber\\
&\hspace{-3cm}\leq\sum_{P\in\T}\Big(h_P|\bk|_{1,\infty}+\|\bk\|_\infty C_{\text{apx}}h_P^\sigma+C_\PF(\|\bb\|_{\infty}+\|\gamma\|_{\infty})h_P\Big)|\Pi_1u_h-E_2u_h|_{1,P}\|\Psi\|_{1+\sigma,P}\nonumber\\
&\hspace{-3cm}\leq C_8\sum_{P\in \T}h_P^\sigma ((1+C_\mathrm{J'})|u_h-\Pi_1u_h|_{1,P}+C_\mathrm{J'}|\Pi_1u_h-E_1u_h|_{1,P})\|\Psi\|_{1+\sigma,P}.\label{5.18}
\end{align}
The second inequality in (\ref{5.18}) follows from the Poincar\'e-Friedrichs inequality in Lemma~\ref{2.4b}.a for $\Pi_1u_h-E_2u_h$ with $\int_{\partial P}(\Pi_1u_h-E_2u_h)\,ds=0$ (from above); the  constant $C_8 := |\bk|_{1,\infty}+C_{\text{apx}}\|\bk\|_{\infty}+C_\PF(\|\bb\|_{\infty}+\|\gamma\|_{\infty})$ results  from \eqref{5.15b} and $h_P\leq h_P^{\sigma}$ (recall $h_\text{max}\leq 1$). Lemma~\ref{lem5.3} with $v_1=\Pi_1u_h$  and \eqref{5.11}  in (\ref{5.18}) show
\begin{align}
B_\pw(\Pi_1u_h-E_2u_h,\Psi)&\leq C_8\sum_{P\in\T}h_P^\sigma((1+C_\mathrm{J'})a_0^{-1/2}C_s^{1/2}\zeta_P+C_\mathrm{J'}C_\mathrm{En}\Xi_P)\|\Psi\|_{1+\sigma,P}.\label{5.19a}
\end{align}
Rewrite (\ref{5.10a})-(\ref{5.14}) with $w=\Psi-\Pi_1I_h\Psi$ and  $h_P^{-1}\|w\|_{L^2(P)}+|w|_{1,P}\leq(1+C_\PF)\|(1-\Pi_0)\nabla\Psi\|_{L^2(P)}\leq C_{\text{apx}}(1+C_\PF)h_P^\sigma|\Psi|_{1+\sigma,P}$ from (\ref{25.1}). This and (\ref{5.19a}) lead in (\ref{5.15}) to 
\begin{align*}
\|u-E_2u_h\|_{L^2(\Omega)}^2 \leq C_9\sum_{P\in\T}h_P^\sigma(\eta_P+\zeta_P+\Lambda_P+\Xi_P)\|\Psi\|_{1+\sigma,P} 
\end{align*} 
for $C_9:=C_{\text{apx}}(1+C_\PF+C_s^{1/2}\|\bk\|_\infty^{1/2}(2+C_{\PF}+C_{\text{Itn}}))+C_8((1+C_\mathrm{J'})a_0^{-1/2}C_s^{1/2}+C_\mathrm{J'}C_\mathrm{En}).$  This and the regularity \eqref{reg} result in
  \begin{align}
  \|u-E_2u_h\|_{L^2(\Omega)}\leq C_9C^*_{\text{reg}}\sum_{P\in\T}h_P^\sigma(\eta_P+\zeta_P+\Lambda_P+\Xi_P).\label{5.20a}
  \end{align}
 The arguments in the proof of (\ref{5.18})-(\ref{5.19a}) also lead to
\begin{align}
\|E_2u_h-\Pi_1u_h\|_{L^2(\Omega)}\leq C_\PF((1+C_\mathrm{J'})a_0^{-1/2}C_s^{1/2}+C_\mathrm{J'}C_\mathrm{En})\sum_{P\in\T}h_P(\zeta_P+\Xi_P).\label{5.20}
\end{align}
 The combination of \eqref{5.19}, \eqref{5.20a}-(\ref{5.20})  and the triangle inequality 
 \[ \|u-u_h\|_{L^2(\Omega)} \leq \|u-E_2u_h\|_{L^2(\Omega)}+\|E_2u_h-\Pi_1u_h\|_{L^2(\Omega)}+\|\Pi_1u_h-u_h\|_{L^2(\Omega)}\]
 lead to  (\ref{5.22a}) with 
$ C_{rel2}/2=C_9C^*_{\text{reg}}+C_\PF\big((2+C_\mathrm{J'})a_0^{-1/2}C_s^{1/2}+C_\mathrm{J'}C_\mathrm{En}\big).$ This concludes the proof of the $L^2$ error estimate in Theorem~\ref{5.2}. \qed
\subsubsection{Comments}
\paragraph{Estimator for  $u-\Pi_1u_h$}
	  The triangle inequality with  (\ref{5.22}) and (\ref{5.11}) provide an  upper bound for $H^1$ error 
	\begin{align*}
	\frac{1}{2}|u-\Pi_1u_h|_{1,\pw}^2 \leq |u-u_h|_{1,\pw}^2+|(1-\Pi_1)u_h|_{1,\pw}^2 \leq 2C_{\text{rel}1}^2(\eta_\T^2+\zeta_\T^2+\Lambda_\T^2+\Xi_\T^2).
	\end{align*}
The same arguments for an upper bound of the $L^2$ error in Theorem~\ref{5.2} show that 
	\begin{align*}
	\frac{1}{2}\|u-\Pi_1u_h\|_{L^2(\Omega)}^2 &\leq \|u-u_h\|_{L^2(\Omega)}^2+\|(1-\Pi_1)u_h\|_{L^2(\Omega)}^2 \\&\leq C_{\text{rel}2}^2\sum_{P\in\T}h_P^{2\sigma}(\eta_P^2+2\zeta_P^2+\Lambda_P^2+\Xi_{P}^2).
	\end{align*}
	The numerical experiments do not display $C_{\text{rel}1}$ and $C_{\text{rel}2}$, and directly compare the error $H1e:=|u-\Pi_1u_h|_{1,\pw}$ in the piecewise $H^1$ norm and the error $L2e:=\|u-\Pi_1u_h\|_{L^2(\Omega)}$ in the $L^2$ norm with the upper bound $H1\mu$ and $L2\mu$ (see, e.g., Figure~\ref{fig5.2}).
\paragraph{Motivation and discussion of apx}
We first argue that those extra terms have to be expected and utilize the abbreviations $\bs := \bk\nabla u+\bb u$ and $g := f-\gamma u$ for the exact solution $u\in H^1_0(\Omega)$ to \eqref{9}, which reads
\begin{align}
(\bs,\nabla v)_{L^2(\Omega)} = (g,v)_{L^2(\Omega)}\quad\text{for all}\; v\in H^1_0(\Omega).\label{a1}
\end{align}
Recall the definition of $s_h(\cdot,\cdot)$ from Subsection~3.1. The discrete problem \eqref{19} with the discrete solution $u_h\in V_h$ assumes the form 
\begin{align}
(\bs_h,\nabla\Pi_1 v_h)_{L^2(\Omega)} + s_h((1-\Pi_1)u_h,(1-\Pi_1)v_h)= (g_h,\Pi_1v_h)_{L^2(\Omega)}\quad\text{for all}\; v_h\in V_h\label{a2}
\end{align}
for $\bs_h:=\bk\nabla\Pi_1u_h+\bb\Pi_1u_h$, and $g_h:=f-\gamma \Pi_1u_h$. Notice that $\bs_h$ and $g_h $ may be replaced in \eqref{a2}  by $\Pi_0\bs_h$ and $\Pi_1g_h$ because the test functions $\nabla\Pi_1v_h$ and $\Pi_1v_h$ belong to $\p_0(\T;\R^2)$ and $\p_1(\T)$ respectively. In other words, the discrete problems \eqref{19} and \eqref{a2} do not see a difference of $\bs_h$ and $g_h$ compared to $\Pi_0\bs_h$ and $\Pi_1g_h$ and so the errors $\bs_h-\Pi_0\bs_h$ and $g_h-\Pi_1g_h$ may arise in \textit{a posteriori} error control. This motivates the \textit{a posteriori} error term $\|\bs_h-\Pi_0\bs_h\|_{L^2(\Omega)}=\Lambda_\T$ as well as the approximation terms $\bs-\Pi_0\bs$ and $g-\Pi_1g$ on the continuous level. The natural norm for the dual variable $\bs$ is $L^2$ and that of $g$ is $H^{-1}$ and hence their norms form the approximation term apx as defined in Subsection~5.1. 
\begin{example}[$\bb=0$]
	The term $(1-\Pi_0)\bs$ may not be visible in case of no advection $\bb=0$ at least if $\bk$ is piecewise constant. Suppose $\bk\in\p_0(\T;\R^{2\times 2})$ and estimate
	\[\|(1-\Pi_0)(\bk\nabla u)\|_{L^2(\Omega)}\leq \|\bk\|_\infty\|(1-\Pi_0)\nabla u\|_{L^2(\Omega)}\lesssim |u-\Pi_1u_h|_{1,\pw}. \] 
	If $\bk$ is not constant, there are oscillation terms that can be treated properly in adaptive mesh-refining algorithms, e.g., in \cite{cascon2008quasi}.
\end{example}
\begin{example}[$\gamma$ piecewise constant]
	While the data approximation term $\mathrm{osc}_1(f,\T)$ \cite{binev2004adaptive} is widely accepted as a part of the total error in the approximation of nonlinear problems, the term $\mathrm{osc}_1(\gamma u,\T)=\|\gamma h_\T(u-\Pi_1 u)\|_{L^2(\Omega)}\lesssim h_{\text{max}}^{1+\sigma}\|f\|_{L^2(\Omega)}$ is of higher order and may even be absorbed in the overall error analysis for a piecewise constant coefficient $\gamma\in\p_0(\T)$. In the general case $\gamma\in L^\infty(\Omega)\backslash \p_0(\T)$, however, $\mathrm{osc}_1(u,\T)$  leads in particular to terms with $\|\gamma-\Pi_0\gamma\|_{L^\infty(\Omega)}$.
\end{example}
\paragraph{Higher-order nonconforming VEM}
The analysis applied in Theorem~\ref{5.2} can be extended to the nonconforming VEM space of higher order $k\in\mathbb{N}$ (see \cite[Sec.~4]{5} for the definition of discrete space). Since the projection operators $\nabla\Pi_k^{\nabla}$ and $\Pi_{k-1}\nabla$ are not the same for general $k$, and the first operator does not lead to optimal order of convergence  for $k\geq 3$, the discrete formulation uses $\Pi_{k-1}\nabla$ (cf. \cite[Rem.~4.3]{4} for more details). The definition and approximation properties of the averaging operator $E_\pw$ extend to the operator $E^k:\p_k(\tT)\to H^1_0(\Omega)$ (see \cite[p.~2378]{dG} for a proof). The identity (\ref{rem}) does not hold in general, but  algebraic calculations lead to  
\begin{align*}
\eta_P^2&:=h_P^2\|f-\gamma\Pi_ku_h\|_{L^2(P)}^2,\qquad\Lambda_P^2:=\|(1-\Pi_{k-1})(\bk\Pi_{k-1}\nabla u_h+\bb\Pi_k u_h)\|_{L^2(P)}^2
\\
\zeta_P^2&:=S^P((1-\Pi_k)u_h,(1-\Pi_k)u_h),\qquad\Xi_P^2:=\sum_{E\in\e(P)}|E|^{-1}\|[\Pi_ku_h]_E\|^2_{L^2(E)}.
\end{align*}
The analysis developed for the upper bound of $L^2$ norm also extends to the general case. The model problem is chosen in 2D for the simplicity of the presentation. The results of this work can be  extended to the three-dimensional case with appropriate modifications. The present analysis holds for any higher regularity index $\sigma>0$ and avoids any trace inequality for  higher derivatives. This is possible by a medius analysis in the form  of companion operators \cite{carstensen2018prove}.
\paragraph{Inhomogeneous boundary data}\label{nonhombd}
	 The error estimator for general Dirichlet condition $u|_{\partial\Omega}=g\in H^{1/2}(\partial \Omega)$  can be obtained with some modifications of \cite{17} in Theorem~\ref{5.2}. The only difference is in the modified jump contributions of the boundary edges in the nonconformity term
	\begin{align*}
	\Xi_\T^2=\sum_{E\in\e(\Omega)}|E|^{-1}\|[\Pi_1u_h]\|_{L^2(E)}^2+\sum_{E\in\e(\partial\Omega)}|E|^{-1}\|g-\Pi_1u_h\|_{L^2(E)}^2.
	\end{align*}
	
\subsection{Proof of Theorem~\ref{efficiency}}
	   Recall the notation $\bm{\sigma} =\bk\nabla u+\bb u$ and $\bm{\sigma}_h=\bk\nabla\Pi_1u_h+\bb\Pi_1u_h$ from Subsection~5.3.
	   \begin{proof}[Proof of \eqref{lb}]
	   	The upper bound \eqref{13} for the stabilisation term and the triangle inequality show
	   	\begin{align*}
	   	\zeta_P^2\leq C_s|(1-\Pi_1)u_h|^2_{1,P}\leq 2C_s(|u-u_h|^2_{1,P}+|u-\Pi_1u_h|^2_{1,P}).
	   	\end{align*}
	   	This concludes the proof of \eqref{lb}.
	   	\end{proof}
	 \begin{proof}[Proof of (\ref{lb3})]
	 	The definition of $\Lambda_P, \Pi_0$,  and the triangle inequality  lead to
	 	\begin{align}
	 	\Lambda_P =\|\bm{\sigma}_h-\Pi_0\bs_h\|_{L^2(P)}&\leq\|\bm{\sigma}_h-\Pi_0\bm{\sigma}\|_{L^2(P)}\nonumber\\&\leq \|\bk\nabla(\Pi_1u_h-u)+\bb(\Pi_1u_h-u)\|_{L^2(P)}+\|(1-\Pi_0)\bm{\sigma}\|_{L^2(P)}.\label{5.23}
	 	\end{align}
	 	 The upper bound $\|\bk\|_\infty$ and $\|\bb\|_\infty$ for the coefficients and the triangle inequality lead   to
	 	\begin{align}
	 	&\Lambda_P-\|(1-\Pi_0)\bm{\sigma}\|_{L^2(P)}\leq (\|\bk\|_\infty+\|\bb\|_\infty)\|\Pi_1u_h-u\|_{1,P}\nonumber\\&\qquad\leq(\|\bk\|_\infty+\|\bb\|_\infty)(\|u_h-\Pi_1u_h\|_{1,P}+\|u-u_h\|_{1,P})\leq C_{10}(\zeta_P+\|u-u_h\|_{1,P})\label{5.25a}
	 	\end{align}
	 	with $\|u_h-\Pi_1u_h\|_{1,P}\leq (1+h_PC_\PF)a_0^{-1/2}C_s^{1/2}\zeta_P$ from (\ref{5.11})-(\ref{5.19}) and with  $C_{10}:=(\|\bk\|_\infty+\|\bb\|_\infty)((1+h_PC_\PF)a_0^{-1/2}C_s^{1/2}+1)$. 
	 	This followed by \eqref{lb}  concludes the proof of (\ref{lb3}).
	 	\end{proof}  
 	 Recall the bubble-function $b_{\T}|_P=b_P$ supported on a polygonal domain $P\in\T$ from (\ref{bubble}) as the sum of interior bubble-functions supported on each triangle $T\in \tT(P)$.
	     \begin{proof}[Proof of (\ref{lb1})]
	     Rewrite the term  
	\begin{align}
	 f-\gamma\Pi_1u_h = \Pi_1(f-\gamma\Pi_1u_h)+(1-\Pi_1)(f-\gamma\Pi_1u_h)=:R+\theta,\label{R}
	\end{align} 
	  and denote $R_P:=R|_P$ and $\theta_P:=\theta|_P$. The definition of $B_\pw(u-\Pi_1u_h,v)$ and the weak formulation $B(u,v)=(f,v)$  from \eqref{9} for any $v\in V$ imply 
	\begin{align}
	B_{\pw}(u-\Pi_1u_h,v)+(\bs_h,\nabla v)_{L^2(\Omega)}&=(f-\gamma\Pi_1u_h,v)_{L^2(\Omega)}=(R+\theta,v)_{L^2(\Omega)}.\label{5.16}
	\end{align}
	  Since $b_PR_P$ belongs to $H^1_0(\Omega)$ (extended  by zero outside $P$),  $v :=b_PR_P\in V$ is admissible in \eqref{5.16}. An integration by parts proves that $(\Pi_0\bs_h,\nabla (b_PR_P))_{L^2(P)}=0$. Therefore, \eqref{5.16}  shows
	\begin{align*}
	(R_P,b_PR_P)_{L^2(P)}=B^P(u-\Pi_1u_h,b_PR_P)-(\theta_P,b_PR_P)_{L^2(P)}+((1-\Pi_0)\bs_h,\nabla (b_PR_P))_{L^2(P)}.
	\end{align*}
	The substitution of $\chi=R_P=\Pi_1(f-\gamma\Pi_1u_h)|_P\in\p_1(P)$ in (\ref{b1}) and the previous identity with the boundedness of $B$ and the Cauchy-Schwarz inequality lead to the first two estimates in
	\begin{align}
	&C_b^{-1}\|R_P\|_{L^2(P)}^2\leq(R_P,b_PR_P)_{L^2(P)}\nonumber\\&\quad\leq \Big(M_b|u-\Pi_1u_h|_{1,P}+\|(1-\Pi_0)\bs_h\|_{L^2(P)}\Big)|b_PR_P|_{1,P}+\|\theta_P\|_{L^2(P)}\|b_PR_P\|_{L^2(P)}\nonumber\\&\quad\leq C_b\Big(M_b|u-\Pi_1u_h|_{1,P}+\Lambda_P+h_P\|\theta_P\|_{L^2(P)}\Big)h_P^{-1}\|R_P\|_{L^2(P)}.\nonumber
	\end{align}
 The last inequality follows from the definition of $\Lambda_P$, and (\ref{b2}) with $\chi = R_P$. This proves that $C_b^{-2}h_P\|R_P\|_{L^2(P)}\leq M_b|u-\Pi_1u_h|_{1,P}+\Lambda_P+h_P\|\theta_P\|_{L^2(P)}.$
  Recall  $\eta_P$ from Subsection~5.1 and  $\eta_P = h_P\|f-\gamma\Pi_1u_h\|_{L^2(P)}\leq h_P\|R_P\|_{L^2(P)}+h_P\|\theta_P\|_{L^2(P)}$ from the split in \eqref{R} and the triangle inequality.    This and the previous estimate of  $h_P\|R_P\|_{L^2(P)}$ show the first estimate in
 \begin{align*}
 \eta_P&\leq C_b^2(M_b|u-\Pi_1u_h|_{1,P}+\Lambda_P)+ (C_b^2+1)h_P\|\theta_P\|_{L^2(P)}\\&\leq (C_b^2+1)\Big(M_b|u-\Pi_1u_h|_{1,P}+\Lambda_P+ h_P\|(f-\gamma\Pi_1u_h)-\Pi_1(f-\gamma u)\|_{L^2(P)}\Big)\\&\leq (C_b^2+1)\Big((M_b+h_P\|\gamma\|_\infty)\|u-\Pi_1u_h\|_{1,P}+\Lambda_P+\mathrm{osc}_1(f-\gamma u,P)\Big).
 \end{align*}
 The second step results from  the definition of $\theta_P=(1-\Pi_1)(f-\gamma\Pi_1u_h)|_P$ in \eqref{R} followed by the $L^2$ orthogonality of $\Pi_1$,  and the last step results from an elementary algebra with the triangle inequality and  $\mathrm{osc}_1(f-\gamma u,P)=h_P\|(1-\Pi_1)(f-\gamma u)\|_{L^2(P)}$ from Subsection~5.1. The triangle inequality for the term $u-\Pi_1u_h$ and the estimate of $\|u_h-\Pi_1u_h\|_{1,P}$ as in \eqref{5.25a}  lead to 
\begin{align*}
C_{11}^{-1}\eta_P\leq \|u-u_h\|_{1,P}+\zeta_P+\Lambda_P+\mathrm{osc}_1(f-\gamma u,P)
\end{align*}
with  $C_{11} := (C_b^{2}+1)(M_b+h_P\|\gamma\|_\infty)((1+h_PC_\PF)a_0^{-1/2}C_s^{1/2})+1)$.  The combination of \eqref{lb} and (\ref{lb3}) in the last displayed estimate concludes the proof of (\ref{lb1}).
\end{proof}
\begin{proof}[Proof of (\ref{lb2})]
Recall  for $u\in H^1_0(\Omega)$ and $u_h\in V_h$ that $\dashint_E u\,ds$ and $\dashint_E u_h\,ds$ are well defined for all edges $E\in\e$, and so the constant $\alpha_E:=\dashint_{E}(u-u_h)\,ds$ is uniquely defined as well.  Since the jump of $u-\alpha_E$ across any edge $E\in\e$ vanishes, $[\Pi_1u_h]_E = [\Pi_1u_h-u+\alpha_E]_E$. Recall $\omega_E=\text{int}(P^+\cup P^-)$ for $E\in\e(\Omega)$ and $\omega_E=\text{int}(P)$ for $E\in\e(\partial\Omega)$  from Subsection~5.1. The trace inequality  $\|v\|^2_{L^2(E)}\leq C_T(|E|^{-1}\|v\|^2_{L^2(\omega_E)}+|E|\;\|\nabla v\|^2_{L^2(\omega_E)})$  (cf. \cite[p.~554]{14})  leads to
\begin{align}
&|E|^{-1/2}\|[\Pi_1u_h]_E\|_{L^2(E)}\leq C_T \left(|E|^{-1}\|\Pi_1u_h-u+\alpha_E\|_{L^2(\omega_E)}+\|\nabla_\pw(\Pi_1u_h-u)\|_{L^2(\omega_E)}\right).\nonumber
\end{align}
This and the triangle inequality show the first estimate in
\begin{align}
|E|^{-1/2}\|[\Pi_1u_h]_E\|_{L^2(E)}&\leq C_T\Big( |E|^{-1}(\|u_h-\Pi_1u_h\|_{L^2(\omega_E)}+\|u_h-u+\alpha_E\|_{L^2(\omega_E)})\nonumber\\&\qquad+\|\nabla_\pw(u_h-\Pi_1u_h)\|_{L^2(\omega_E)}+\|\nabla_\pw(u-u_h)\|_{L^2(\omega_E)}\Big).\label{5.27}
\end{align}
The estimates (\ref{5.11})-\eqref{5.19} control the term $\|u_h-\Pi_1u_h\|_{1,P}$ as in \eqref{5.25a}, and  the  Poincar\'e-Friedrichs inequality from Lemma~\ref{2.4b}.b for $u_h-u+\alpha_E$ with $\int_{E}(u_h-u+\alpha_E)\,ds = 0$ (by the definition of $\alpha_E$) implies that $\|u_h-u+\alpha_E\|_{L^2(P)}\leq C_\PF h_P|u_h-u|_{1,P}$. 
This with the mesh assumption $h_P\leq \rho^{-1}|E|$ and \eqref{5.27}   result in 
\begin{align*}
|E|^{-1/2}\|[\Pi_1u_h]_E\|_{L^2(E)}\leq C_T((C_\PF\rho^{-1}+1)a_0^{-1/2}C_s^{1/2}+C_\PF+1)\sum_{P'\in\omega_E}(\Lambda_{P'}+|u-u_h|_{1,P'}).
\end{align*}
Since this holds for any edge $E\in \e(P)$, the sum over all these edges and the bound \eqref{lb} in the above estimate conclude the proof of \eqref{lb2}.
\end{proof}
\begin{rem}[convergence rates of $L^2$ error control for $0<\sigma\leq 1$]
The efficiency estimates \eqref{lb1}-\eqref{lb2}  with a multiplication of $h_P^{2\sigma}$ show that the local quantity $h_P^{2\sigma}(\eta_P^2+\Lambda_P^2+\Xi_P^2)$ converges to zero with the expected convergence rate.	
\end{rem}
\begin{rem}[efficiency up to stabilisation and oscillation for $L^2$ error control when $\sigma=1$]
	For convex domains and $\sigma=1$, there is even a local efficiency result that is briefly described in the sequel: The arguments in the above proof of (\ref{lb1})-(\ref{lb3}) lead to
	\begin{align*}
	h_P^{2}\eta_P^2&\lesssim \|u-u_h\|^2_{L^2(P)}+h_P^{2}(\zeta_P^2+\mathrm{osc}_1^2(f-\gamma u,P)+\|(1-\Pi_0)\bs\|_{L^2(P)}^2),\\
	h_P^{2}\Lambda_P^2&\lesssim \|u-u_h\|^2_{L^2(P)}+h_P^{2}(\zeta_P^2+\|\bk-\Pi_0\bk\|_{L^\infty(P)}^2\|f\|^2_{L^2(\Omega)}+\|(1-\Pi_0)\bb u\|_{L^2(P)}^2).
	\end{align*}
	   The observation $[\Pi_1u_h]_E=[\Pi_1u_h-u]_E$ for the term $\Xi_P$, the trace inequality, and the triangle inequality   show, for any $E\in\e$, that
	\begin{align*}
	|E|^{1/2}\|[\Pi_1u_h]_E\|_{L^2(E)}
	 &\leq C_T\left(\|u_h-\Pi_1u_h\|_{L^2(\omega_E)}+\|u-u_h\|_{L^2(\omega_E)}\right.\nonumber\\&\quad\left.+|E|(\|\nabla\Pi_1(u-u_h)\|_{L^2(\omega_E)}+\|\nabla(u-\Pi_1u)\|_{L^2(\omega_E)})\right).
	\end{align*}
	The bound (\ref{5.19}) for the first term and the inverse estimate $\|\nabla\chi\|_{L^2(P)}\leq C_{\text{inv}}h_P^{-1}\|\chi\|_{L^2(P)}$ for $\chi\in\p_k(P)$ for the third term   result in
	\begin{align*}
	|E|^{1/2}\|[\Pi_1u_h]_E\|_{L^2(E)}\lesssim \|u-u_h\|_{L^2(\omega_E)}+|E|\sum_{P'\in\omega_E}\Big(\|\nabla(1-\Pi_1)u\|_{L^2(P')}+\Lambda_{P'}\Big).
	\end{align*}
	The mesh assumption \ref{M2} implies that $h_P^2\Xi_P^2 \leq \rho^{-1}\sum_{E\in\e(P)}|E|\;\|[\Pi_1u_h]_E\|_{L^2(E)}^2$. This and the above displayed inequality prove the efficiency estimate for $h_P^2\Xi_P^2$.
\end{rem}

\section{Numerical experiments}
This section manifests the performance of the \textit{a posteriori} error estimator and an associated adaptive mesh-refining algorithm with D$\ddot{o}$rfler marking \cite{16}. The numerical results investigate  three computational benchmarks for  the indefinite problem (\ref{1}).
\subsection{Adaptive algorithm}
\textbf{Input}:  initial partition ${\cal T}_0$ of $\Omega$. \\
For $\ell = 0,1,2,\dots$   do
\begin{enumerate}
	\item \textbf{SOLVE}. Compute the discrete solution $u_h$ to (\ref{19}) with respect to $\T_\ell$ for $\ell=0,1,2\dots$ (cf. \cite{3} for more details on the implementation).
	\item \textbf{ESTIMATE}. Compute all the four terms $\eta_\ell:=\eta_{\T_\ell}, \zeta_\ell:=\zeta_{\T_\ell}, \Lambda_\ell:=\Lambda_{\T_\ell}$ and $\Xi_\ell:=\Xi_{\T_\ell}$, which add up to the upper bound \eqref{5.22}.
	\item \textbf{MARK}.  Mark the polygons $P$ in a subset ${\cal M}_\ell \subset$ ${\cal T}_\ell$ with minimal cardinality and
	\begin{align*}
	{H1\mu}_{\ell}^2:=H1\mu^2({\cal T}_\ell):=\eta_\ell^2+\zeta_\ell^2+\Lambda_\ell^2+\Xi_\ell^2\leq 0.5\sum_{P\in{\cal M}_\ell}(\eta_P^2+\zeta_P^2+\Lambda_P^2+\Xi_P^2).
	\end{align*}
	\item \textbf{REFINE} - Refine the marked polygon domains by connecting the mid-point of the edges to the centroid of respective polygon domains and update ${\cal T}_\ell$. (cf. Figure \ref{fig6.1} for an illustration of the refinement strategy.)
\end{enumerate} 
	\begin{figure}[H]
		\centering
		\includegraphics[width=0.2\linewidth]{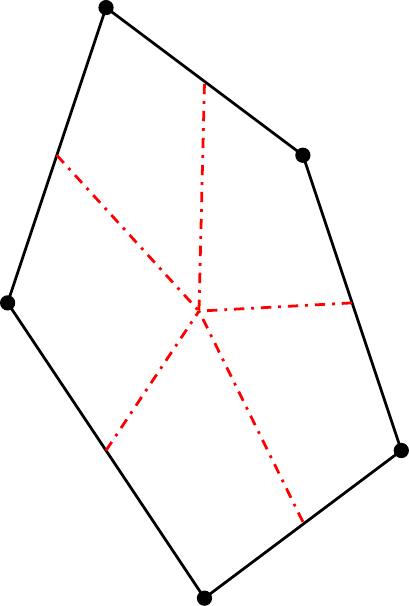}
		\caption{Refinement of a polygon into quadrilaterals}
		\label{fig6.1}
	\end{figure}%
end do\\
\textbf{Output}: The sequences $\T_\ell$, and the bounds  $\eta_\ell,
 \zeta_\ell, \Lambda_\ell, \Xi_\ell$, and $H1\mu_\ell$ for $\ell=0,1,2,\dots$.\\ \par
The adaptive algorithm is displayed for mesh adaption in the energy error $H^1$.  Replace  estimator $H1\mu_\ell$ in the algorithm by $L2\mu_\ell$ (the upper bound in \eqref{5.22a})  for local mesh-refinement in the $L^2$ error. Both uniform and adaptive mesh-refinement run to compare the empirical convergence rates and provide numerical evidence for the superiority of adaptive mesh-refinement. Note that uniform refinement means all the polygonal domains are refined.   In all examples below,  $\overline{\bk}_P=1$ in (\ref{5.1s}). The numerical realizations are based on a MATLAB implementation explained in \cite{product}  with a Gauss-like cubature formula over polygons. The cubature formula is exact for all bivariate polynomials of degree at most $2n-1$, so the choice $n\geq (k+1)/2$ leads to integrate a polynomial of degree $k$ exactly. The quadrature errors in the computation of examples presented below appear negligible for the input parameter $n=5$.

\subsection{Square domain (smooth solution)}
	This subsection discusses the problem (\ref{1}) with the coefficients $\bk=I, \bb=(x,y)$ and $\gamma=x^2+y^3$ on a square domain $\Omega=(0,1)^2$, and  the exact solution 
	\begin{align*}
	u=16x(1-x)y(1-y)\arctan(25x-100y+50)
	\end{align*}
	with $f={\cal L}u$. Since $\gamma-\frac{1}{2}\dv(\bb)=x^2+y^3-1$ is not always positive on $\Omega$, this is an indefinite problem. Initially, the error and the estimators are large because of an internal layer around the line $25x-100y+50=0$ with large first derivative of $u$  resolved after few refinements as displayed in Figure \ref{fig5.1}. 
	\begin{figure}[H]
		\centering
		\begin{subfigure}{.33\textwidth}
			\centering
			\includegraphics[width=0.8\linewidth]{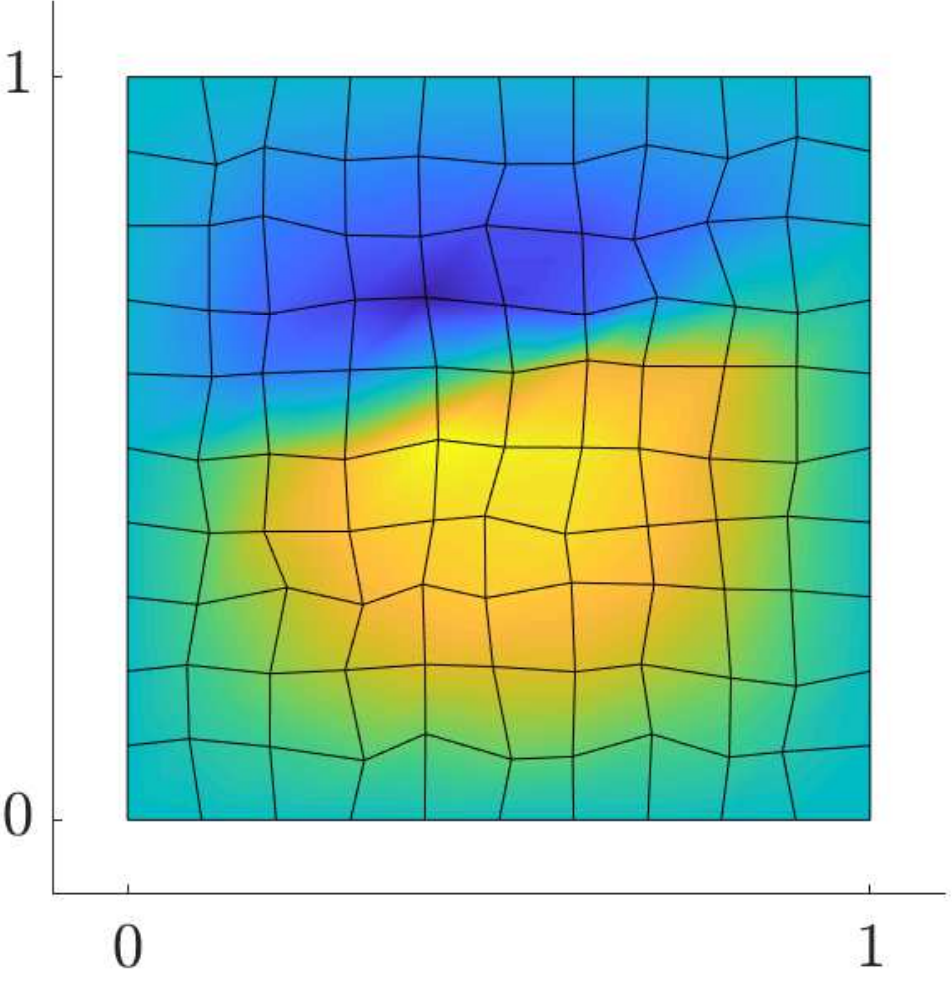}         
		\end{subfigure}%
		\begin{subfigure}{.33\textwidth}
			\centering
			\includegraphics[width=0.8\linewidth]{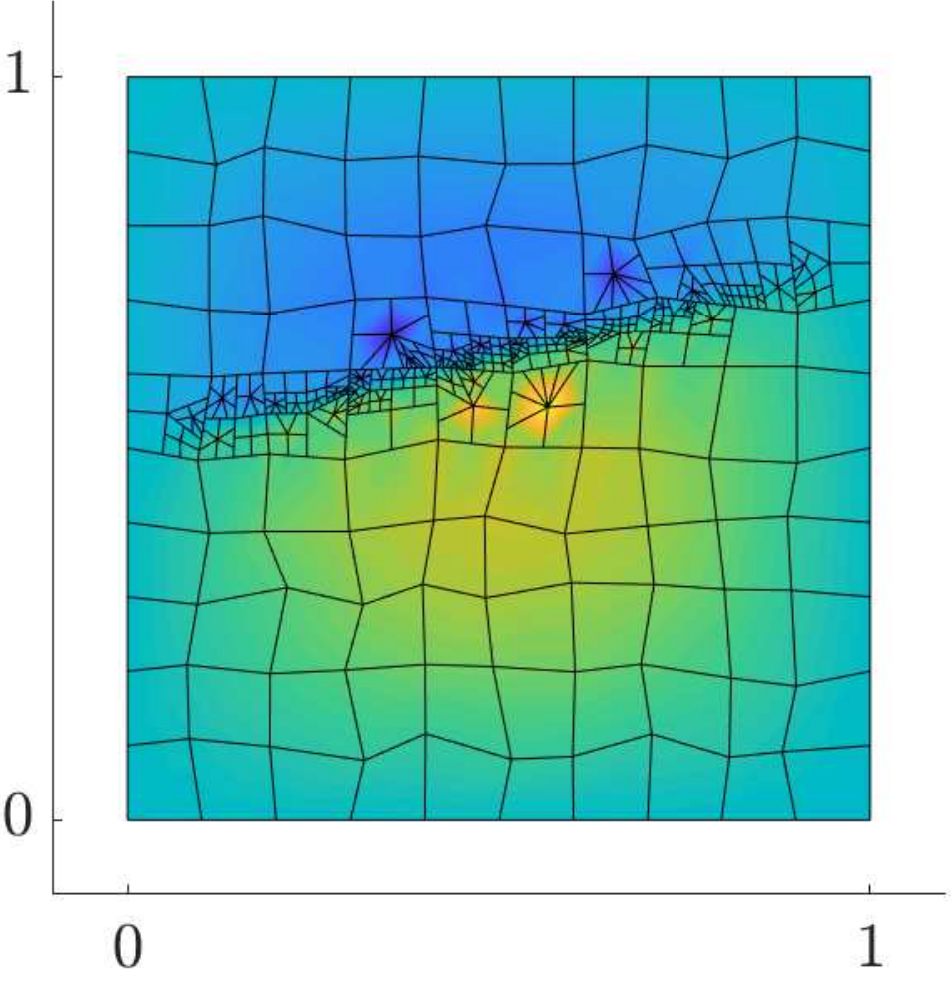}
		\end{subfigure}%
		\begin{subfigure}{.33\textwidth}
			\centering
			\includegraphics[width=0.8\linewidth]{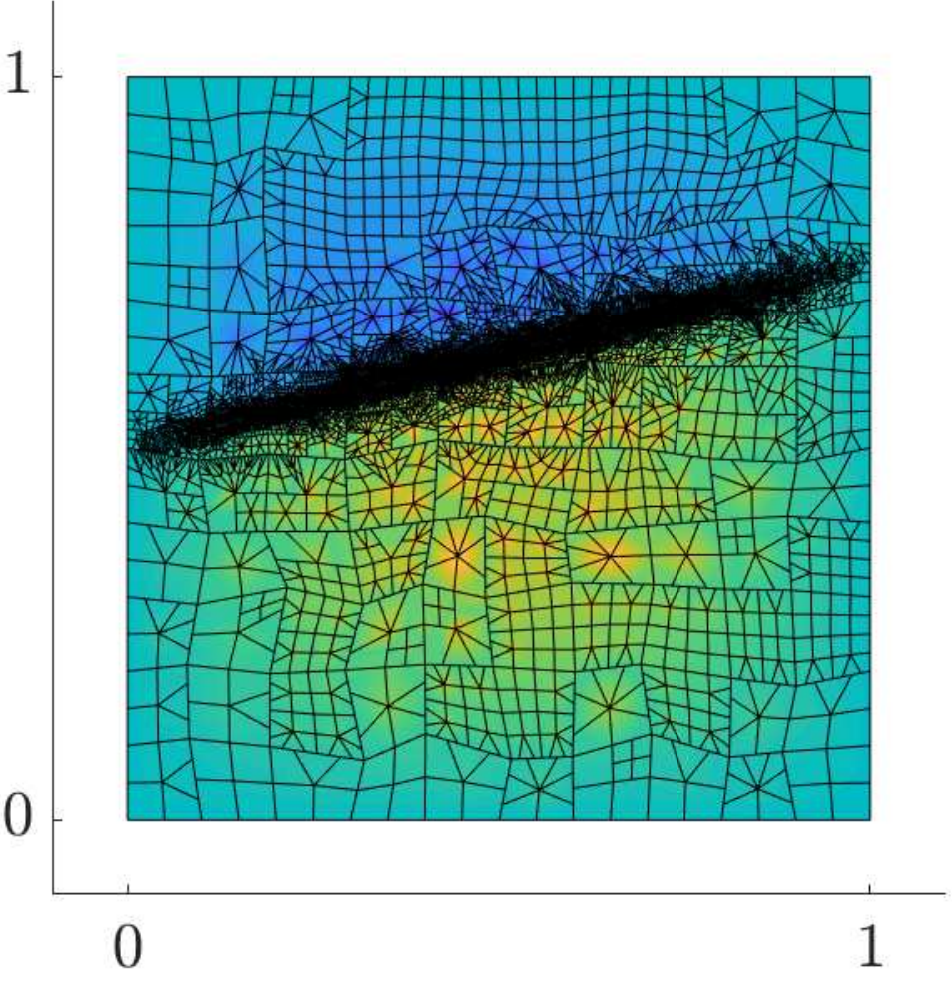}
		\end{subfigure}
		\caption{Output $\T_1, \T_8, \T_{15}$ of the adaptive algorithm}
		\label{fig5.1}
	\end{figure}
	\begin{figure}[H]
		\centering
		\begin{subfigure}{.5\textwidth}
			\centering
			\includegraphics[width=0.8\linewidth]{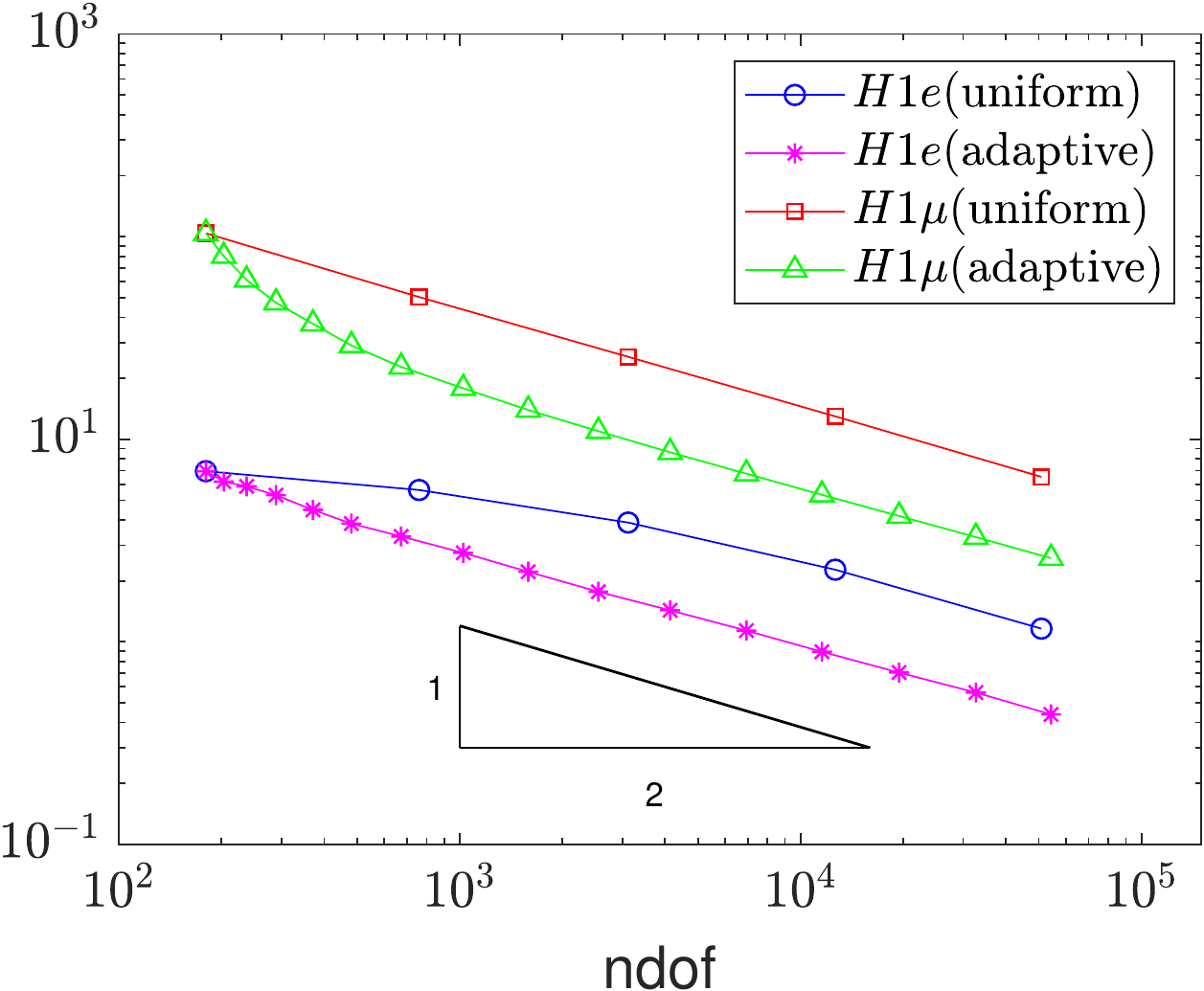}     
			\caption{}
			\vspace{0.5cm}
		\end{subfigure}%
		\begin{subfigure}{.5\textwidth}
			\centering
			\includegraphics[width=0.8\linewidth]{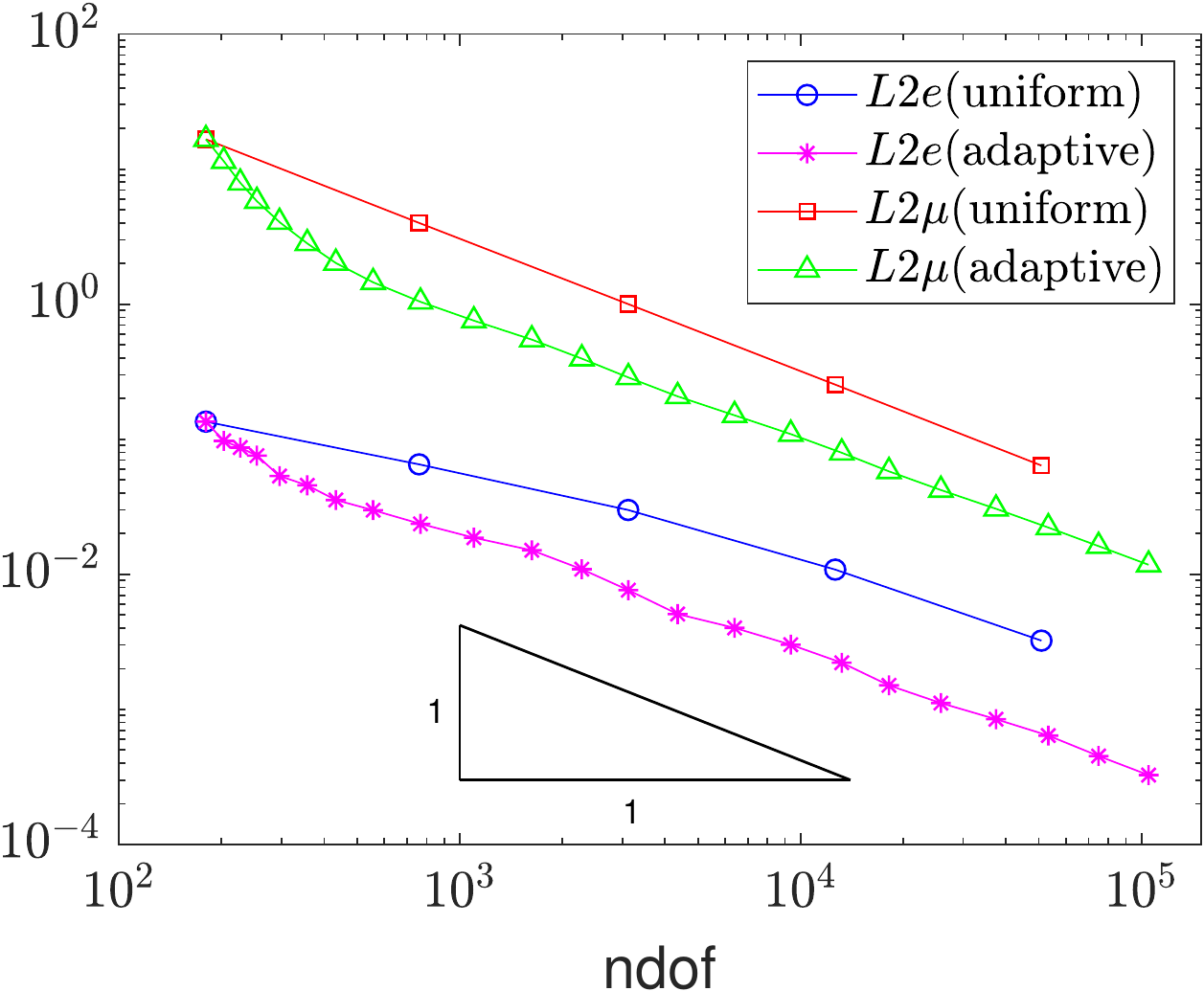}
			\caption{}
			\vspace{0.5cm}
		\end{subfigure}
		\caption{Convergence history plot of estimator $\mu$ and error $e:=u-\Pi_1u_h$  in the (a) piecewise $H^1$ norm  (b) $L^2$ norm vs number ndof of degrees of freedom  for both uniform and adaptive refinement}
		\label{fig5.2}
		\end{figure}
	
\subsection{L-shaped domain (non-smooth solution)}
 This subsection shows an advantage of using  adaptive mesh-refinement over uniform meshing for  the problem (\ref{1})  with the coefficients as $
	\bk=I, \bb=(x,y)\h\text{and}\h \gamma=-4$
	on a L-shaped domain $\Omega=(-1,1)^2\backslash [0 , 1)\times (-1 , 0]$ and the  exact solution 
	\begin{align*}
	u=r^{2/3}\sin\left(\frac{2\theta}{3}\right)
	\end{align*}
	with $f:={\cal L}u$. Since the exact solution is not zero along the  boundary $\partial\Omega$, the error estimators are modified according to Subsection~\ref{nonhombd}. Since $\gamma-\frac{1}{2}\dv(\bb)=-5<0$, the problem is non-coercive. Observe that with increase in number of iterations, refinement is more at the singularity as highlighted in Figure \ref{fig5.3}. Since the exact solution $u$ is in $H^{(5/3)-\epsilon}(\Omega)$ for all $\epsilon>0$, from \textit{a priori} error estimates the expected order of convergence in $H^1$ norm is $1/3$ and in $L^2$ norm is at least $2/3$ with respect to number of degrees of freedom for uniform refinement. Figure \ref{fig5.4} shows that uniform refinement gives the sub-optimal convergence rate, whereas adaptive refinement lead to optimal convergence rates ($1/2$ for $H^1$ norm and $5/6$ in $L^2$ norm).
	\begin{figure}[H]
		\centering
		\begin{subfigure}{.33\textwidth}
			\centering
			\includegraphics[width=0.8\linewidth]{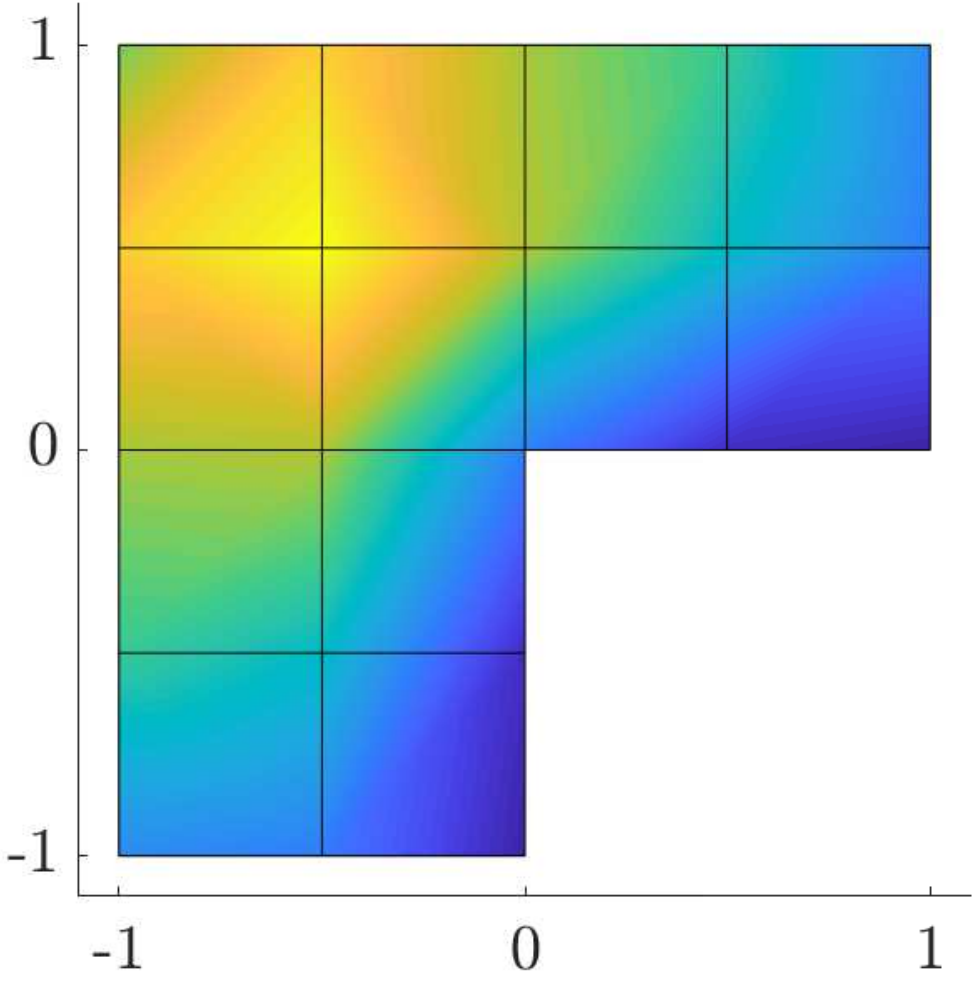}
		\end{subfigure}%
		\begin{subfigure}{.33\textwidth}
			\centering
			\includegraphics[width=0.8\linewidth]{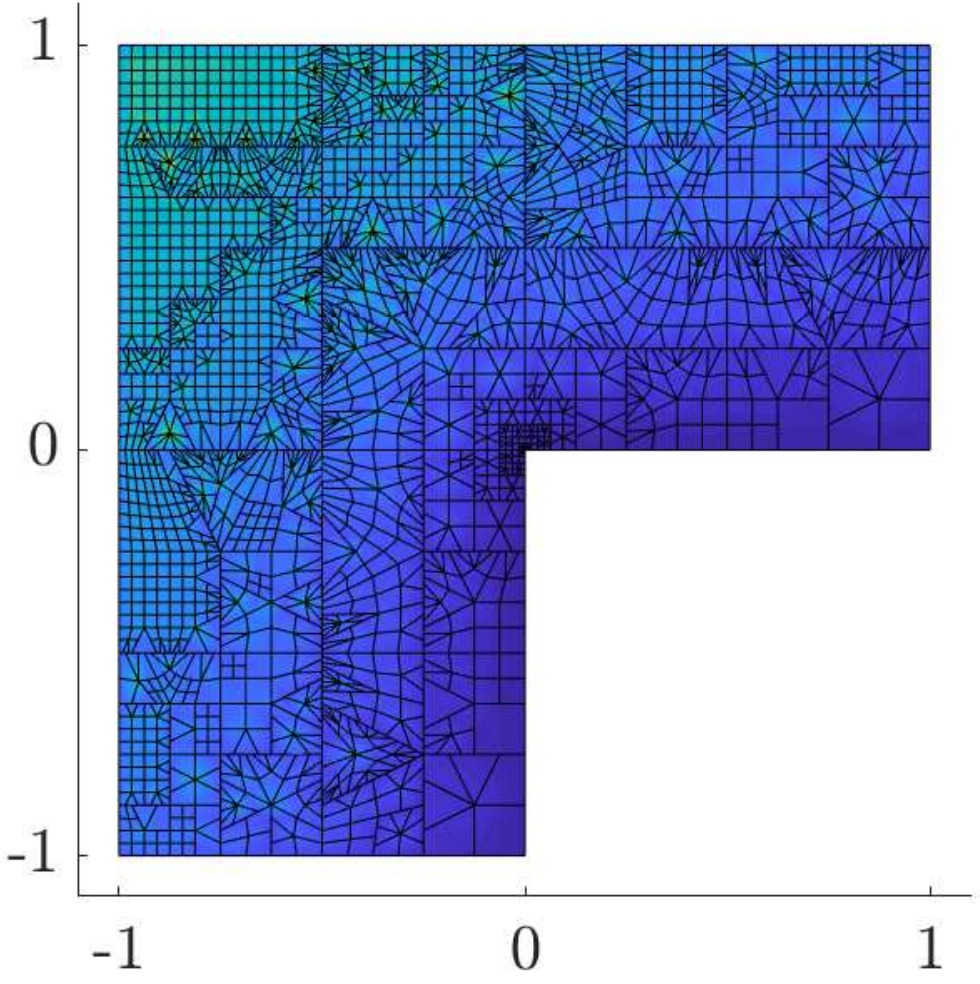}
		\end{subfigure}%
		\begin{subfigure}{.33\textwidth}
			\centering
			\includegraphics[width=0.8\linewidth]{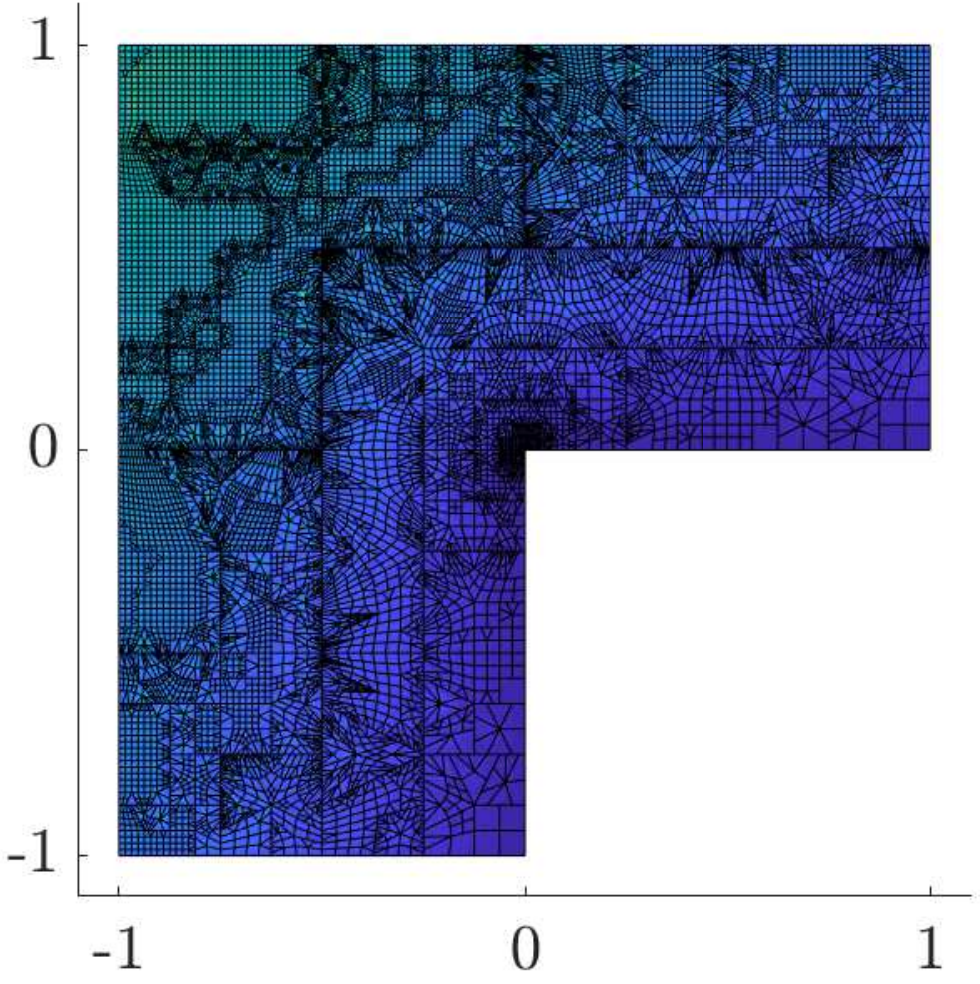}
		\end{subfigure}
		\caption{Output $\T_1, \T_{10}, \T_{15}$ of the adaptive refinement}
		\label{fig5.3}
	\end{figure}
	\begin{figure}[H]
		\centering
		\begin{subfigure}{.5\textwidth}
			\centering
			\includegraphics[width=0.8\linewidth]{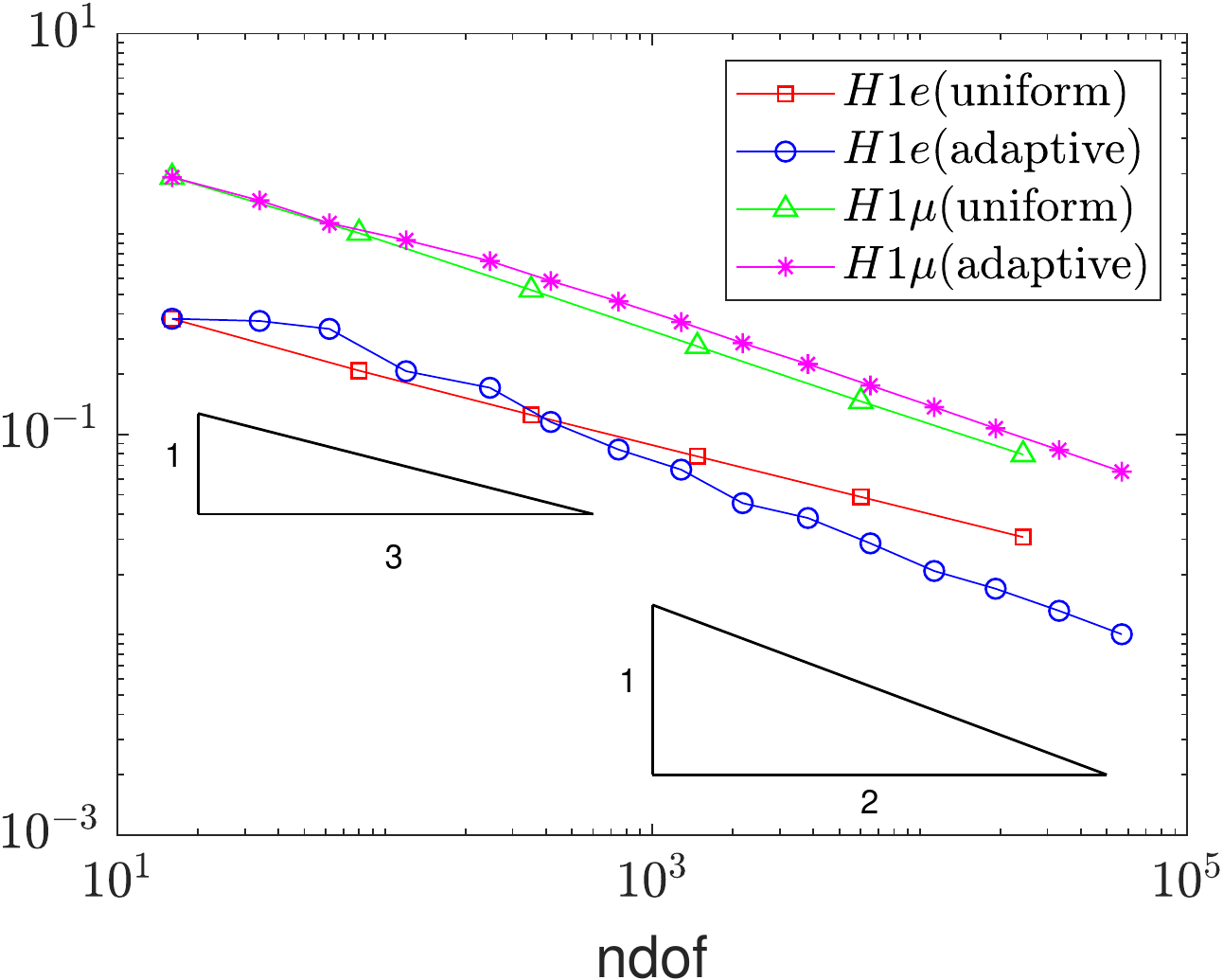}
			\caption{}
			\vspace{0.5cm}
		\end{subfigure}%
		\begin{subfigure}{.5\textwidth}
			\centering
			\includegraphics[width=0.8\linewidth]{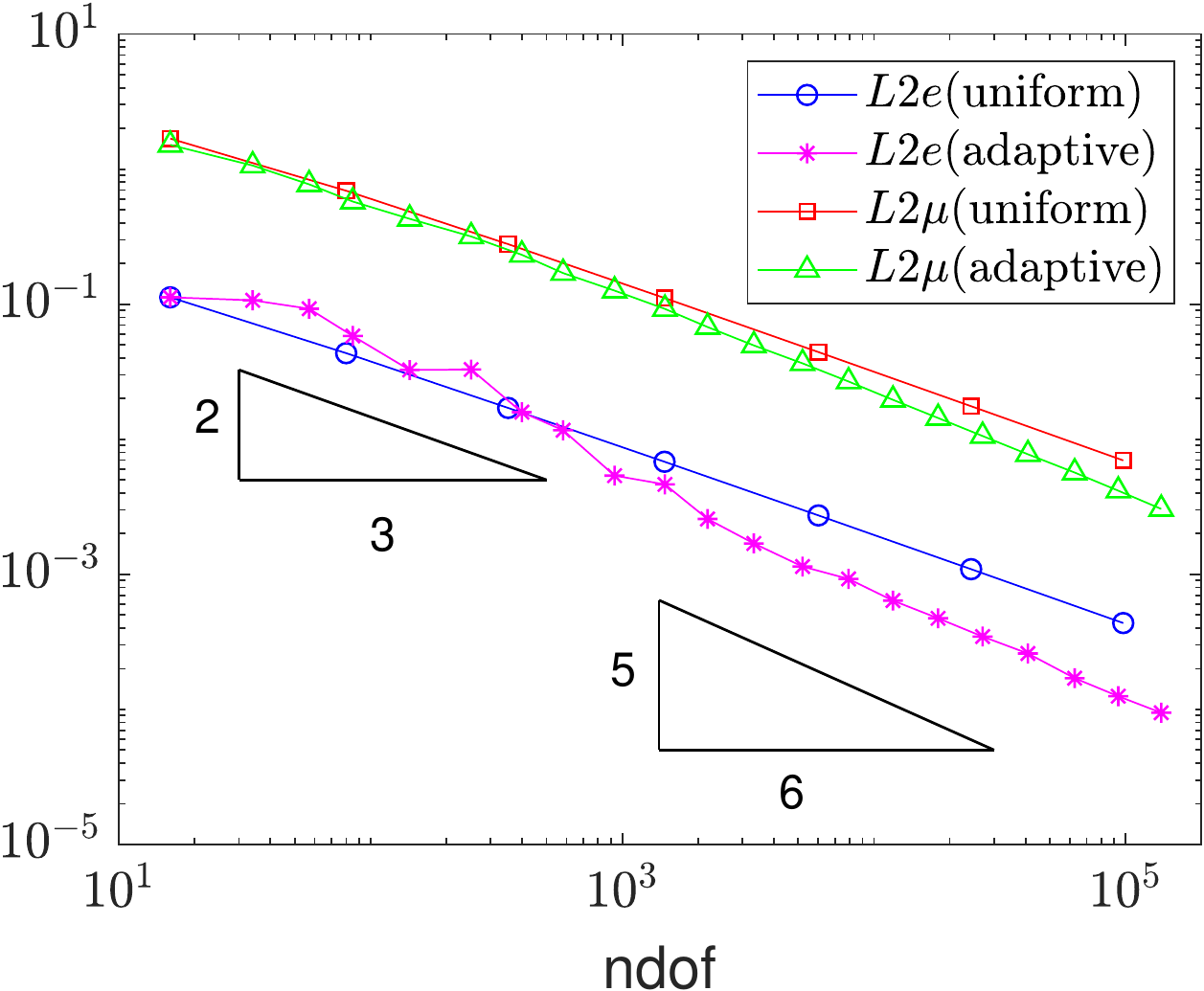}
			\caption{}
			\vspace{0.5cm}
		\end{subfigure}
		\caption{Convergence history plot of estimator $\mu$ and error $e:=u-\Pi_1u_h$  in the (a) piecewise $H^1$ norm  (b) $L^2$ norm vs number ndof of degrees of freedom  for both uniform and adaptive refinement}
		\label{fig5.4}
	\end{figure}

\subsection{Helmholtz equation }
	This subsection considers the exact solution $u=1+\tanh(-9(x^2+y^2-0.25))$ to the problem
	\begin{align*}
	-\Delta u-9 u=f\quad\quad \text{in}\quad\Omega=(-1,1)^2.
	\end{align*}
	 There is an internal layer around the circle centered at $(0,0)$ and of radius $0.25$ where the second derivatives of $u$ are large because of steep increase in the solution resulting in the large error at the beginning, and this gets resolved with refinement as displayed in Figure \ref{fig5.5}.  
	\begin{figure}[H]
		\centering
		\begin{subfigure}{.33\textwidth}
			\centering
			\includegraphics[width=0.8\linewidth]{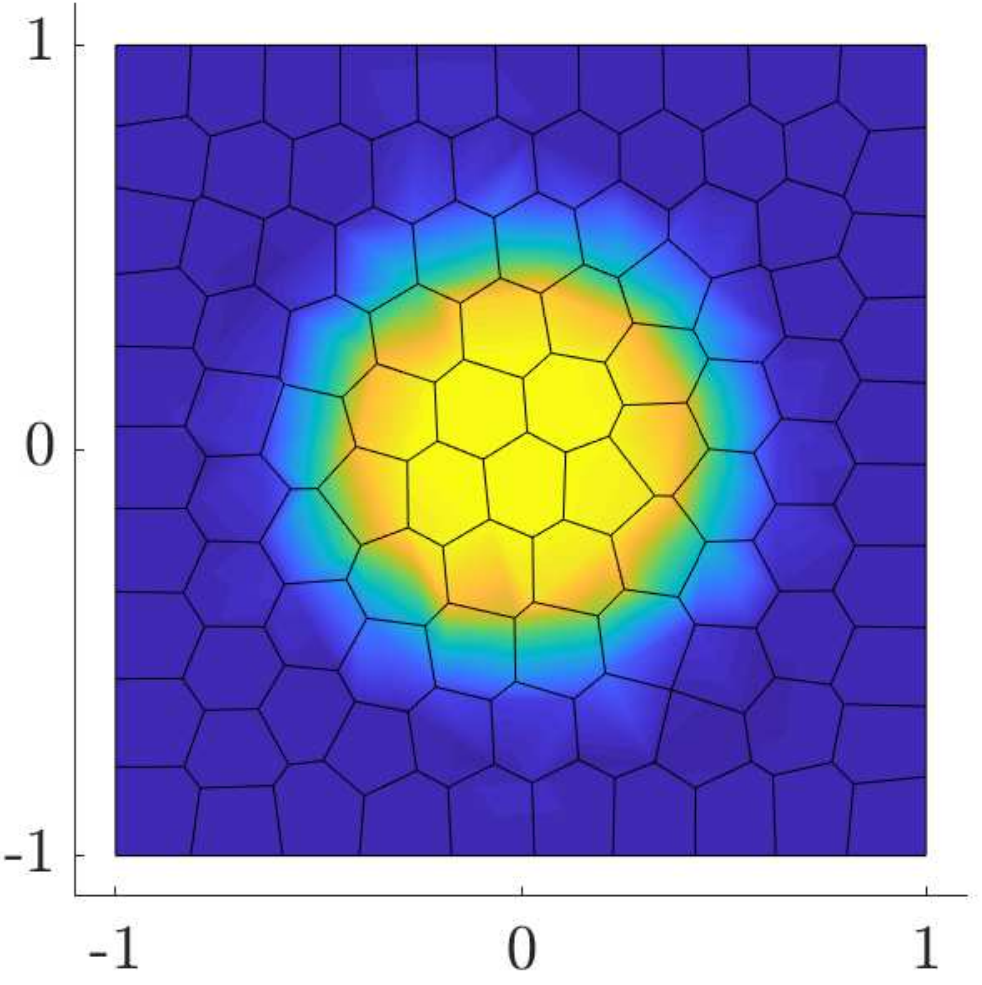}
		\end{subfigure}%
		\begin{subfigure}{.33\textwidth}
			\centering
			\includegraphics[width=0.8\linewidth]{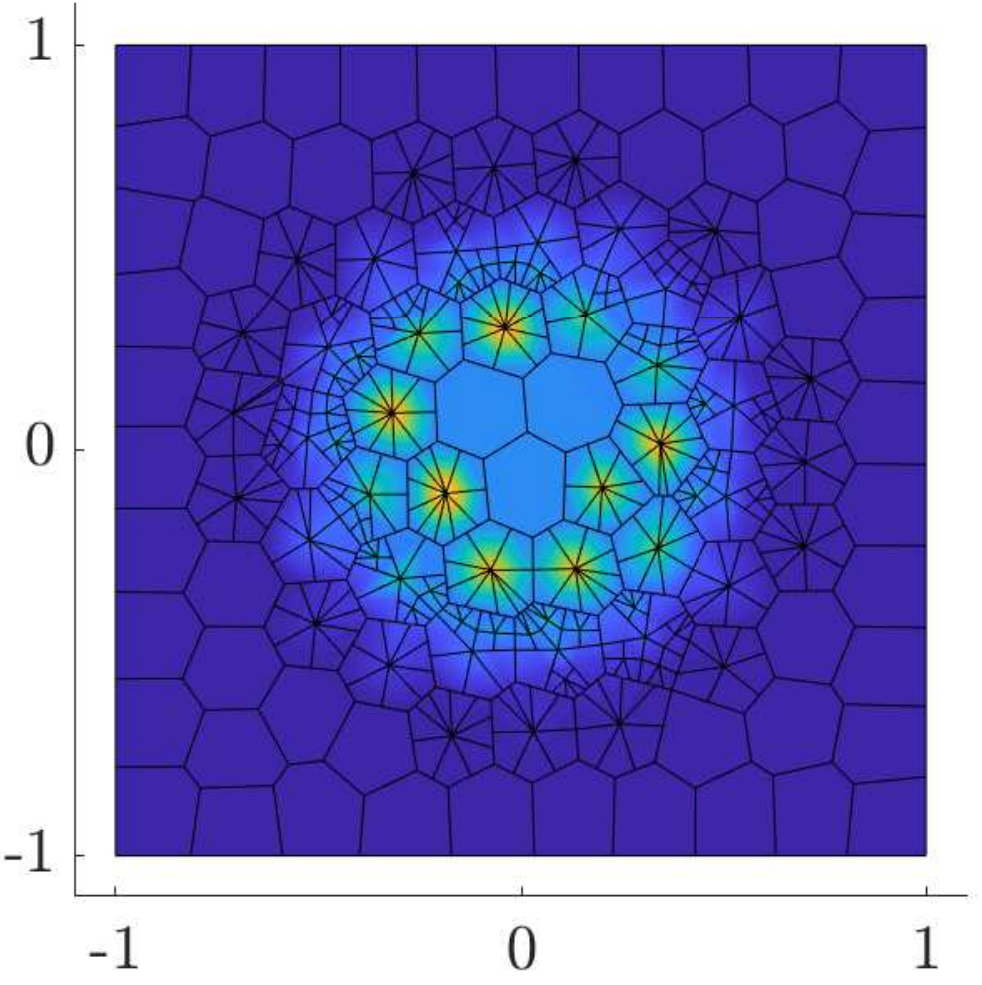}
		\end{subfigure}%
		\begin{subfigure}{.33\textwidth}
			\centering
			\includegraphics[width=0.8\linewidth]{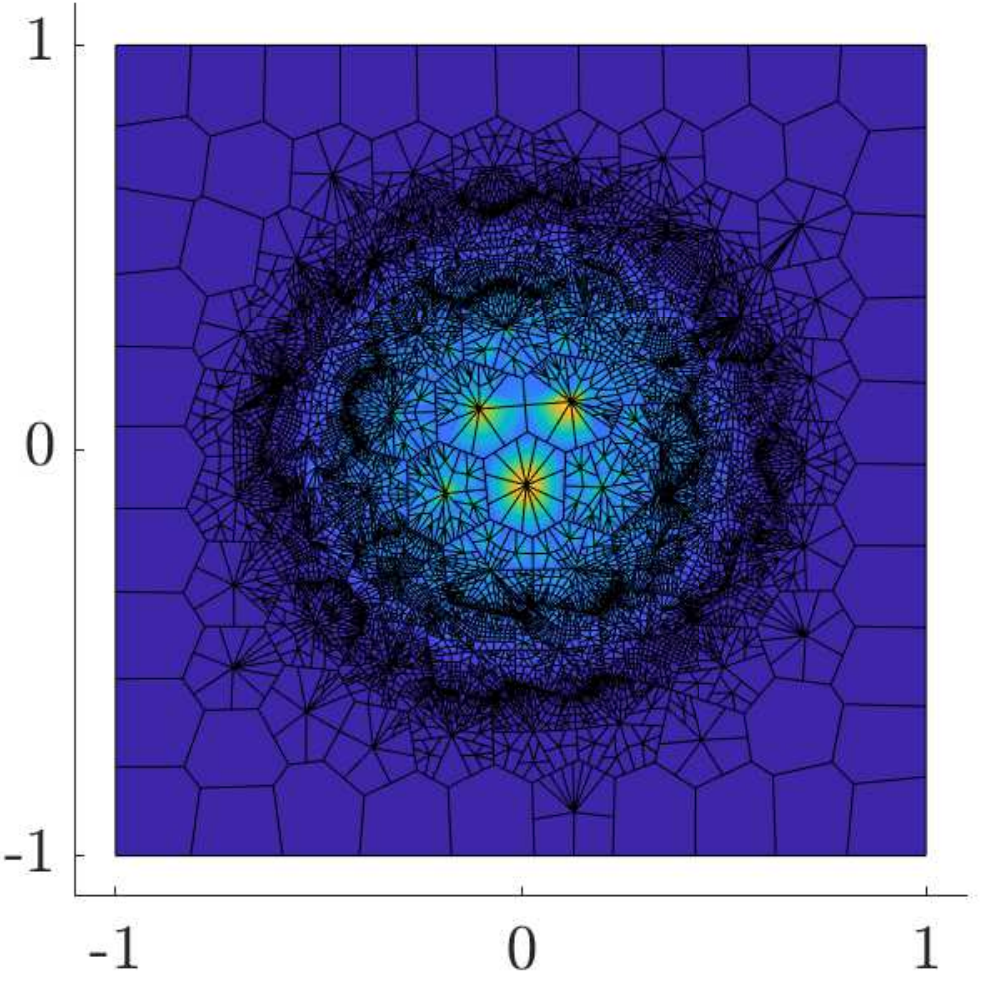}
		\end{subfigure}
		\caption{Output $\T_1, \T_5, \T_{11}$ of the adaptive refinement}
		\label{fig5.5}
	\end{figure}
	\begin{figure}[H]
		\centering
		\begin{subfigure}{.5\textwidth}
			\centering
			\includegraphics[width=0.8\linewidth]{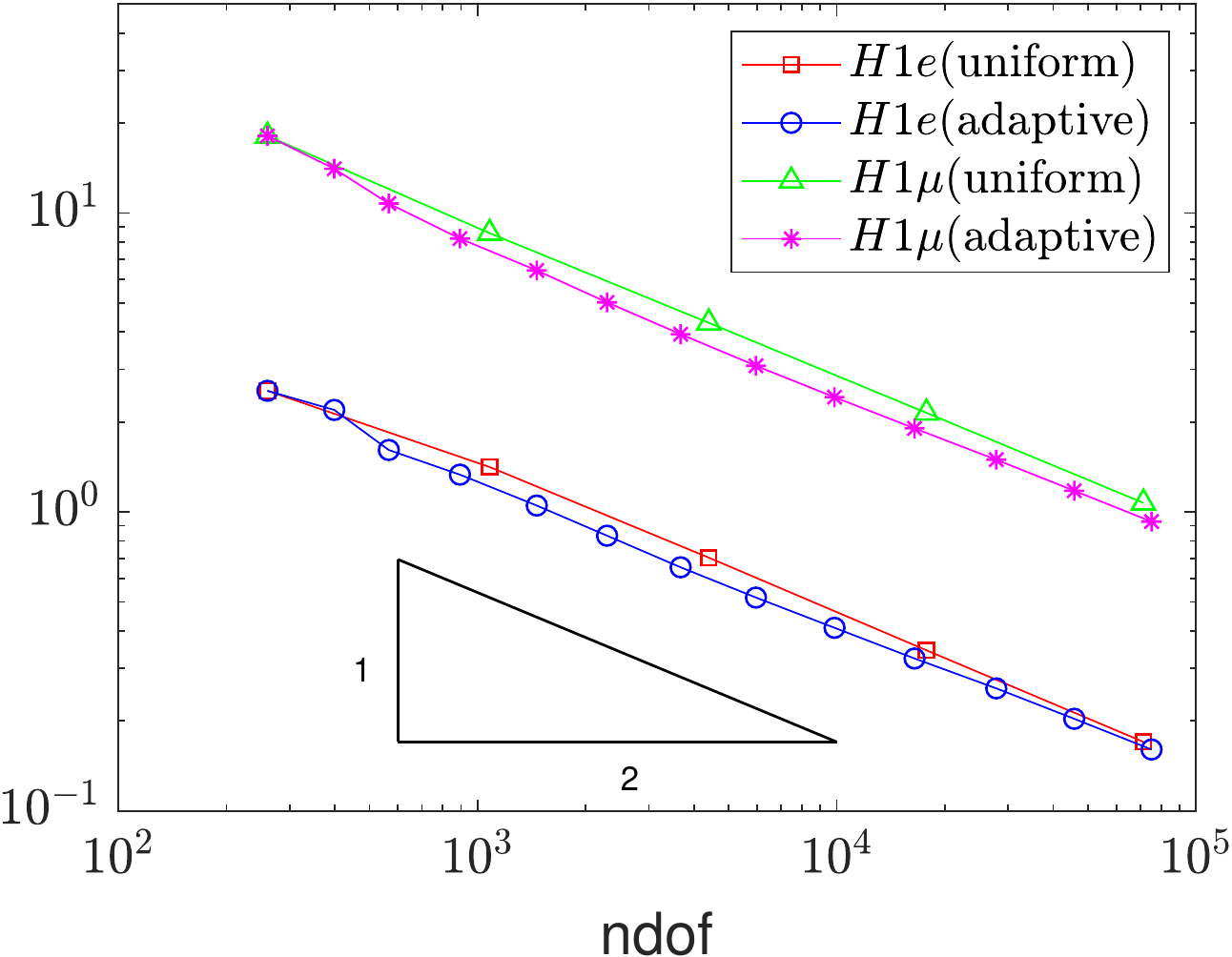}
			\caption{}
			\vspace{0.5cm}
		\end{subfigure}%
		\begin{subfigure}{.5\textwidth}
			\centering
			\includegraphics[width=0.8\linewidth]{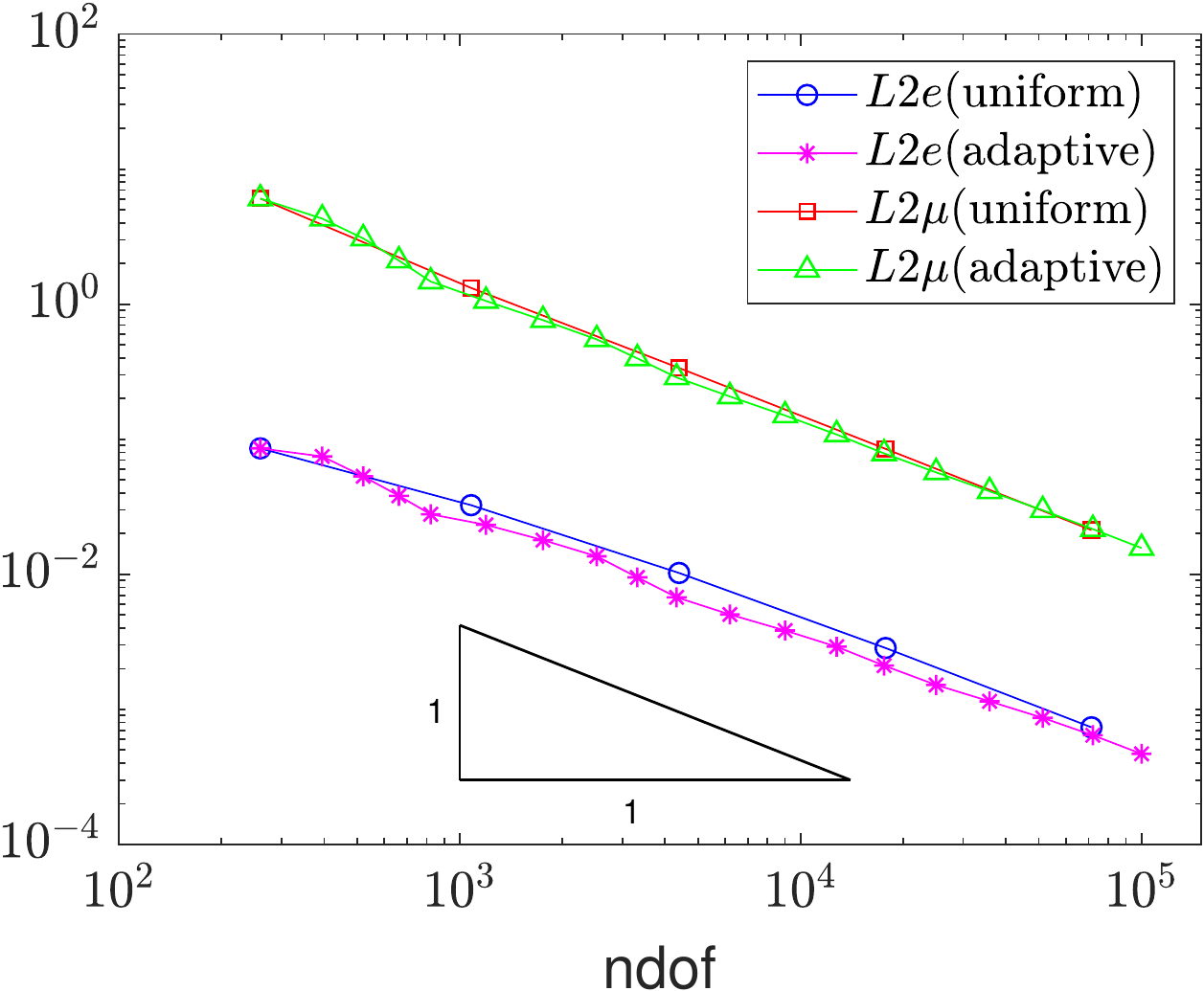}
			\caption{}
			\vspace{0.5cm}
		\end{subfigure}
		\caption{Convergence history plot of estimator $\mu$ and error $e:=u-\Pi_1u_h$  in the (a) piecewise $H^1$ norm  (b) $L^2$ norm vs number ndof of degrees of freedom  for both uniform and adaptive refinement}
		\label{fig5.6}
	\end{figure}
 
\subsection{Conclusion}
The three computational benchmarks provide empirical evidence for the sharpness of the mathematical a \textit{priori} and a \textit{posteriori} error analysis in this paper and illustrate the superiority of adaptive over uniform mesh-refining. The empirical convergence rates in all examples for the  $H^1$ and $L^2$ errors coincide with the predicted convergence rates in Theorem~\ref{h1}, in particular,  for the non-convex domain and reduced elliptic regularity. 
The a \textit{posteriori} error bounds from  Theorem~\ref{5.2} confirm these 
convergence rates as well. The ratio of the error estimator $\mu_\ell$ by the $H^1$ error $e_\ell$, sometimes called efficiency index, remains bounded up to  a typical value 6; we regard this as a typical overestimation factor for the residual-based a~posteriori error estimate. Recall that the constant $C_{\text{reg}}$ has not been displayed so the error estimator  $\mu_\ell$ does not provide a guaranteed error bound. Figure~ \ref{figl21} and \ref{figl22} display the four different contributions volume residual $(\sum_P\eta_P^2)^{1/2}$, stabilization term $(\sum_P\zeta_P^2)^{1/2}$, inconsistency term $(\sum_P\Lambda_P^2)^{1/2}$ and the nonconformity term $(\sum_P\Xi_P^2)^{1/2}$ that add up to the error estimator  $\mu_\ell$. We clearly see that all four terms converge with the overall rates that proves that none of them is a higher-order term and makes it doubtful that some of those terms can be neglected. The  volume residual clearly dominates the a \textit{posteriori} error estimates, while  the stabilisation term remains significantly smaller for the natural stabilisation (with undisplayed parameter one). The proposed adaptive mesh-refining algorithm leads to superior convergence properties and recovers the optimal convergence rates. This holds for the first example with optimal convergence rates in the large pre-asymptotic computational range as well as in the second with suboptimal convergence rates under uniform mesh-refining according to the typical corner singularity and optimal convergence rates for the adaptive mesh-refining. The third example with the Helmholtz equation and a moderate wave number shows certain moderate local mesh-refining in Figure 6.6 but no large improvement over the optimal convergence rates for uniform mesh-refining. The adaptive refinement generates hanging nodes because of the way refinement strategy is defined, but this is not troublesome in VEM setting as hanging node can be treated as a just another vertex in the decompostion of domain. However, an increasing number of hanging nodes with further mesh refinements   may violate the mesh assumption \ref{M2}, but numerically the method seems robust without putting any restriction on the number of hanging nodes. 
The future work on the theoretical investigation of the performance of adaptive mesh-refining algorithm is clearly motivated by the successful numerical experiments.  The aforementioned 
empirical observation  that the stabilisation terms do not dominate the 
a~posteriori error estimates raises  the hope for a possible  convergence analysis of the adaptive mesh-refining strategy with the axioms of adaptivity
\cite{CFPP} towards a proof of optimal convergence rates: The numerical results in this section support this conjecture at least
for the lowest-order VEM  in 2D for  indefinite non-symmetric second-order elliptic PDEs. 
\begin{figure}[H]
	\centering
	\begin{subfigure}{.33\textwidth}
		\centering
		\includegraphics[width=0.8\linewidth]{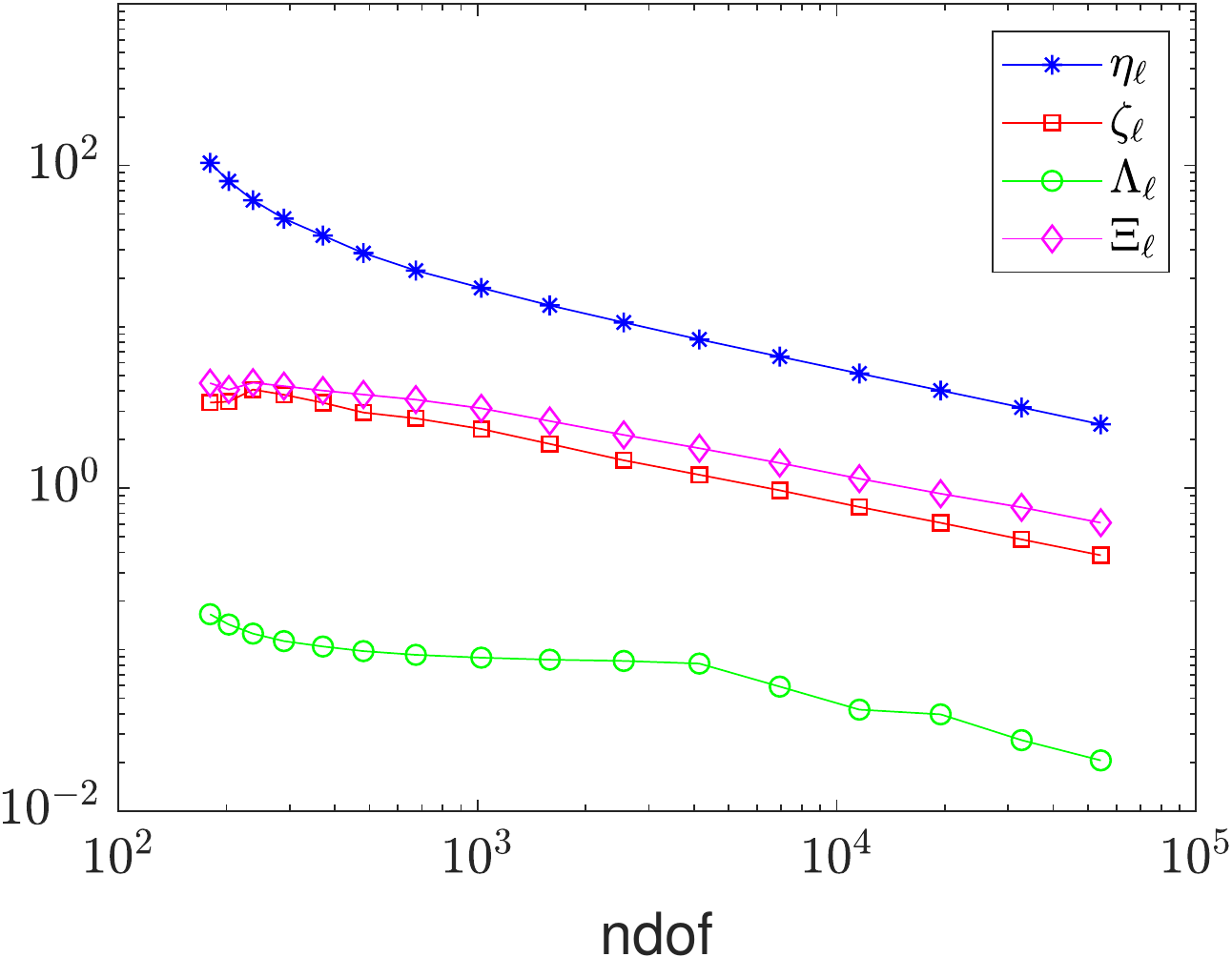}
	\end{subfigure}%
	\begin{subfigure}{.33\textwidth}
		\centering
		\includegraphics[width=0.8\linewidth]{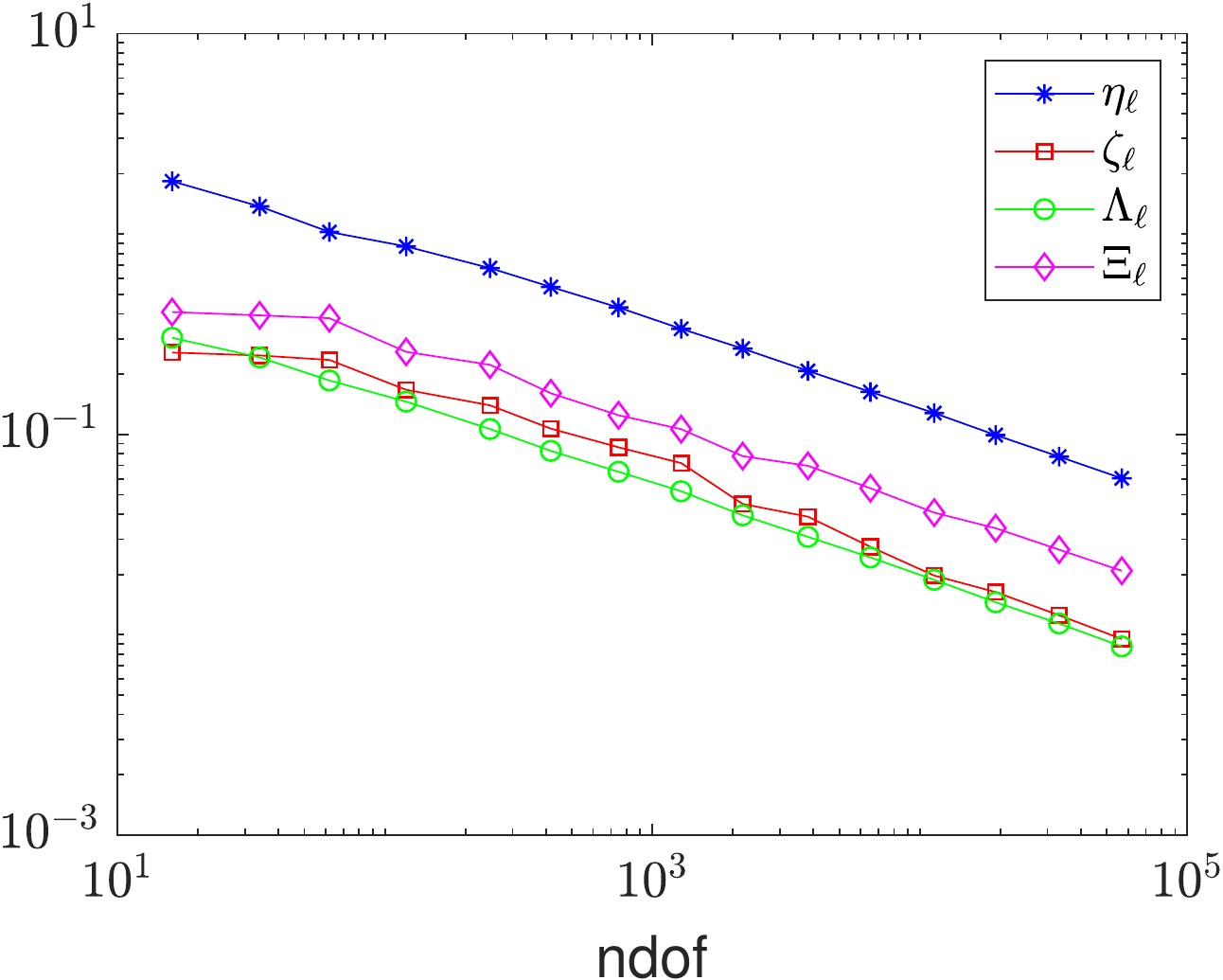}
	\end{subfigure}%
	\begin{subfigure}{.33\textwidth}
		\centering
		\includegraphics[width=0.8\linewidth]{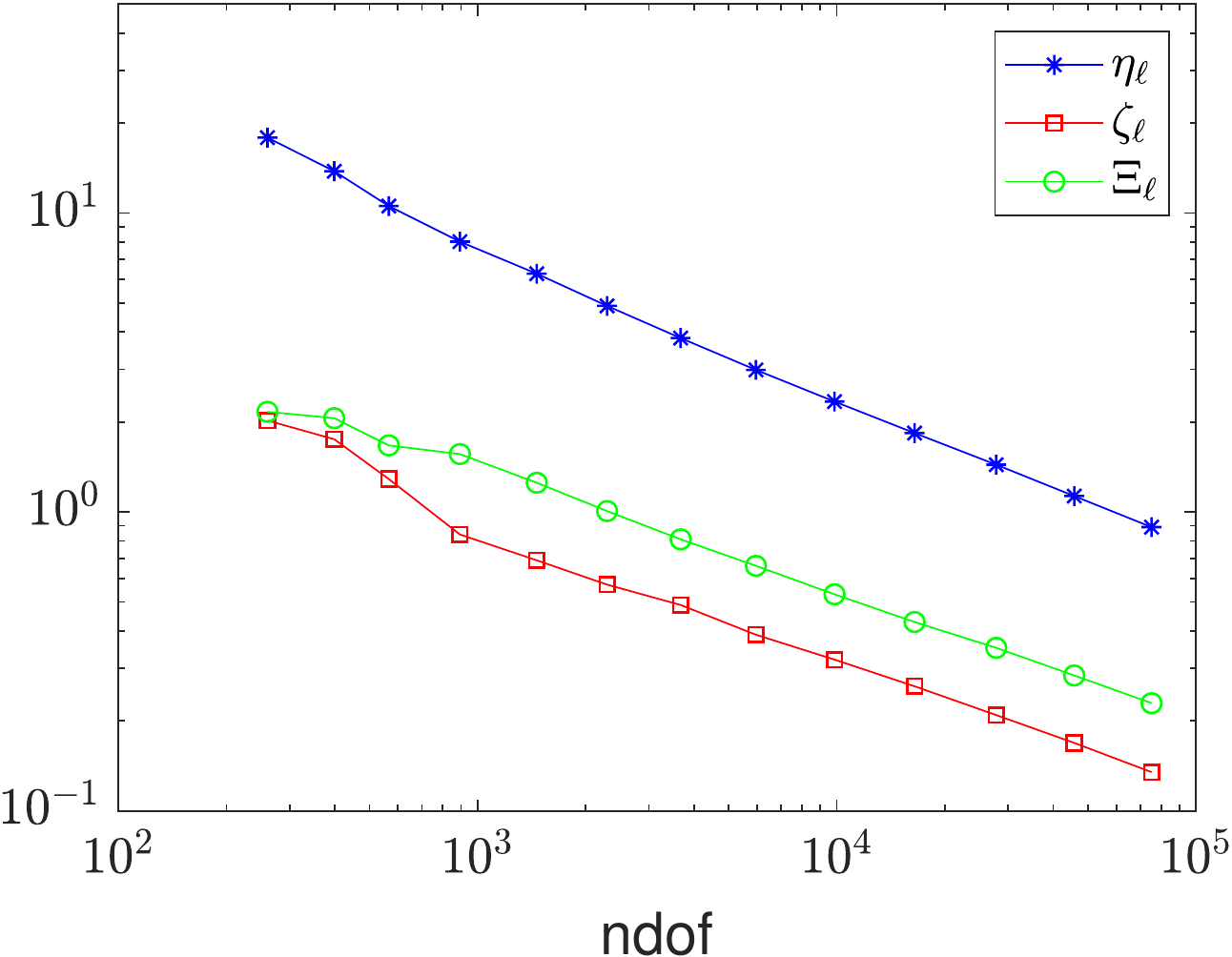}
	\end{subfigure}
	\caption{Estimator components corresponding to the error $H1e=|u-\Pi_1u_h|_{1,\pw}$ of the adaptive refinement presented in Subsection 6.2-6.4}
	\label{figl21}
\end{figure}
\begin{figure}[H]
	\centering
	\begin{subfigure}{.33\textwidth}
		\centering
		\includegraphics[width=0.8\linewidth]{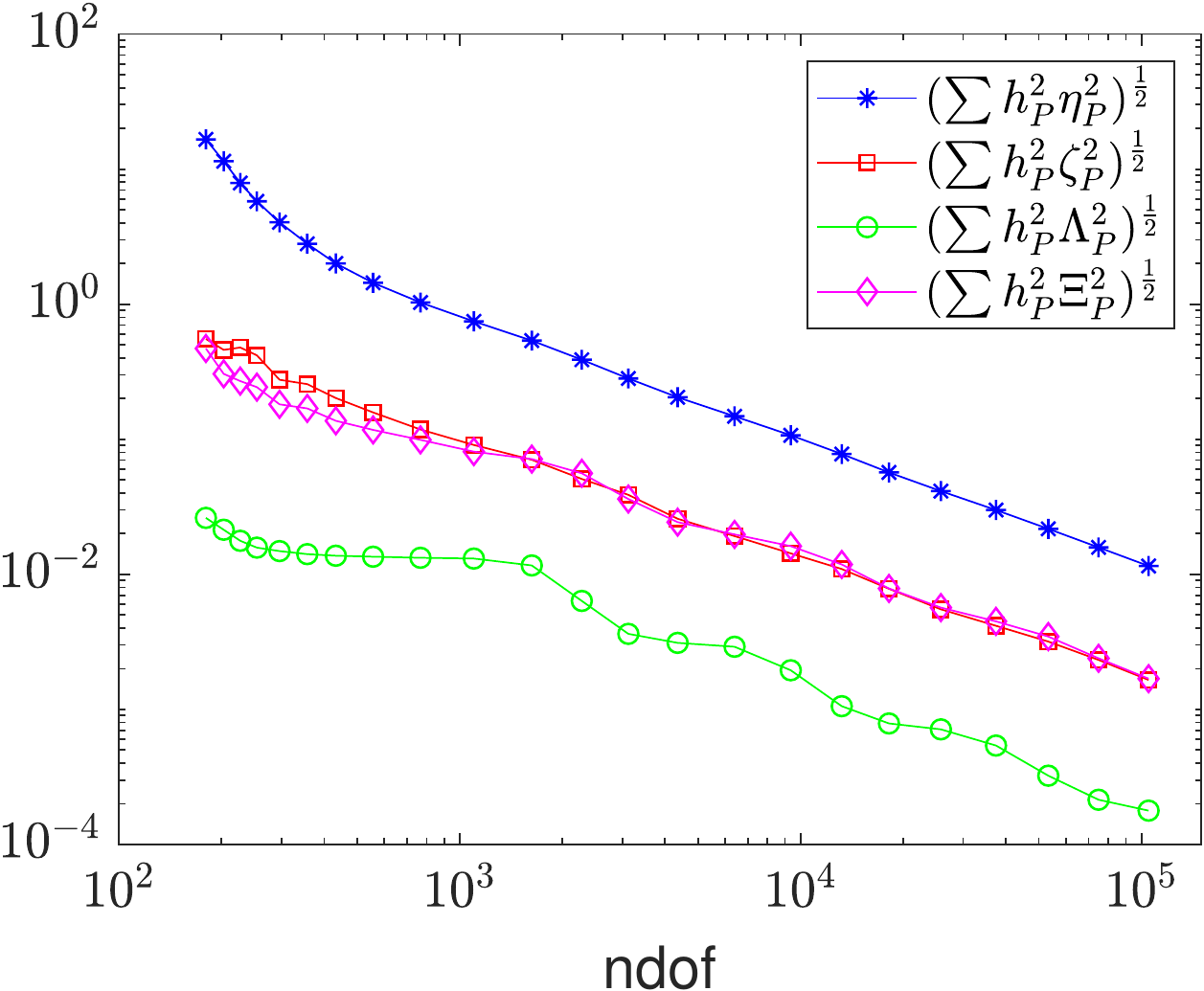}
	\end{subfigure}%
	\begin{subfigure}{.33\textwidth}
		\centering
		\includegraphics[width=0.8\linewidth]{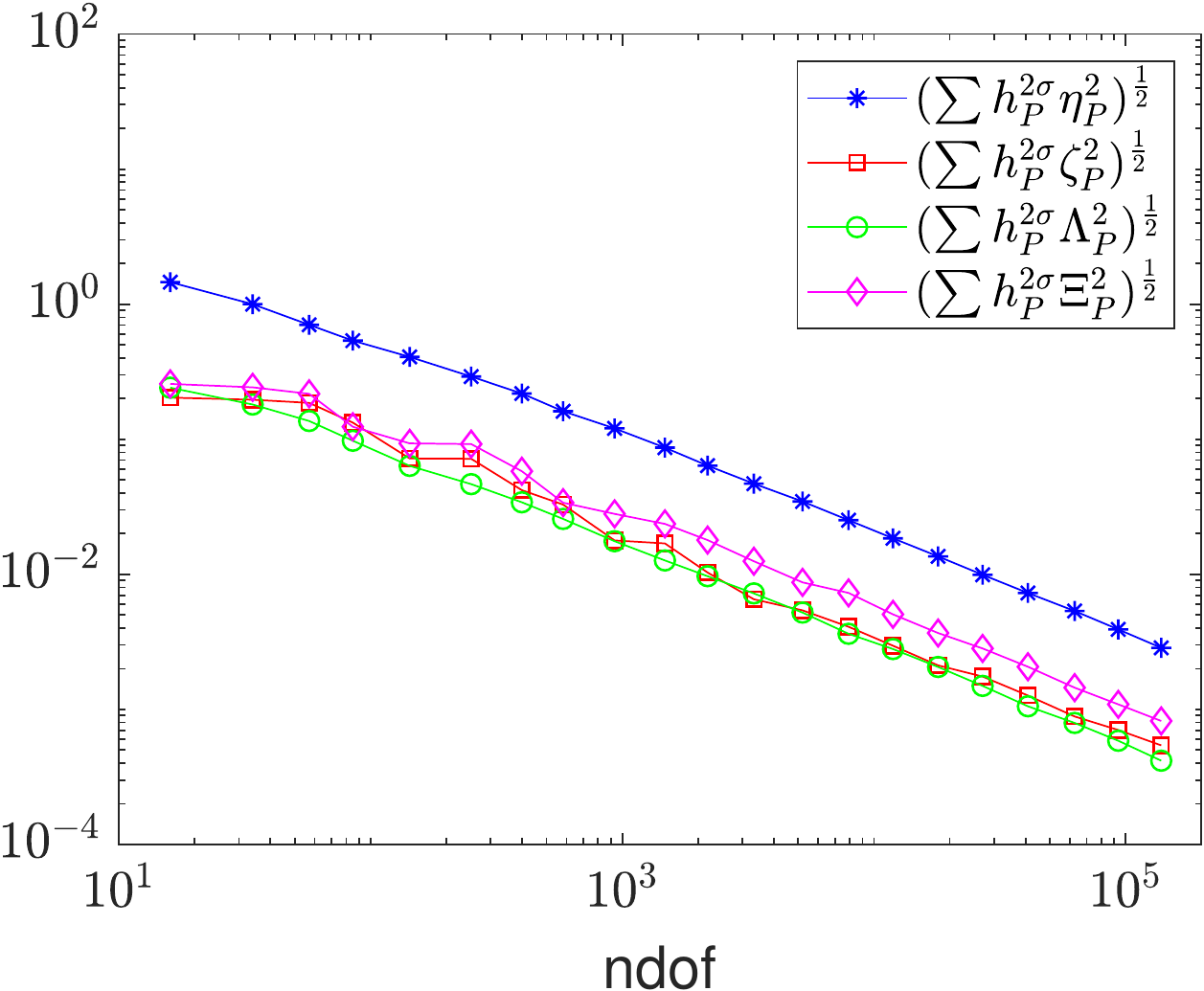}
	\end{subfigure}%
	\begin{subfigure}{.33\textwidth}
		\centering
		\includegraphics[width=0.8\linewidth]{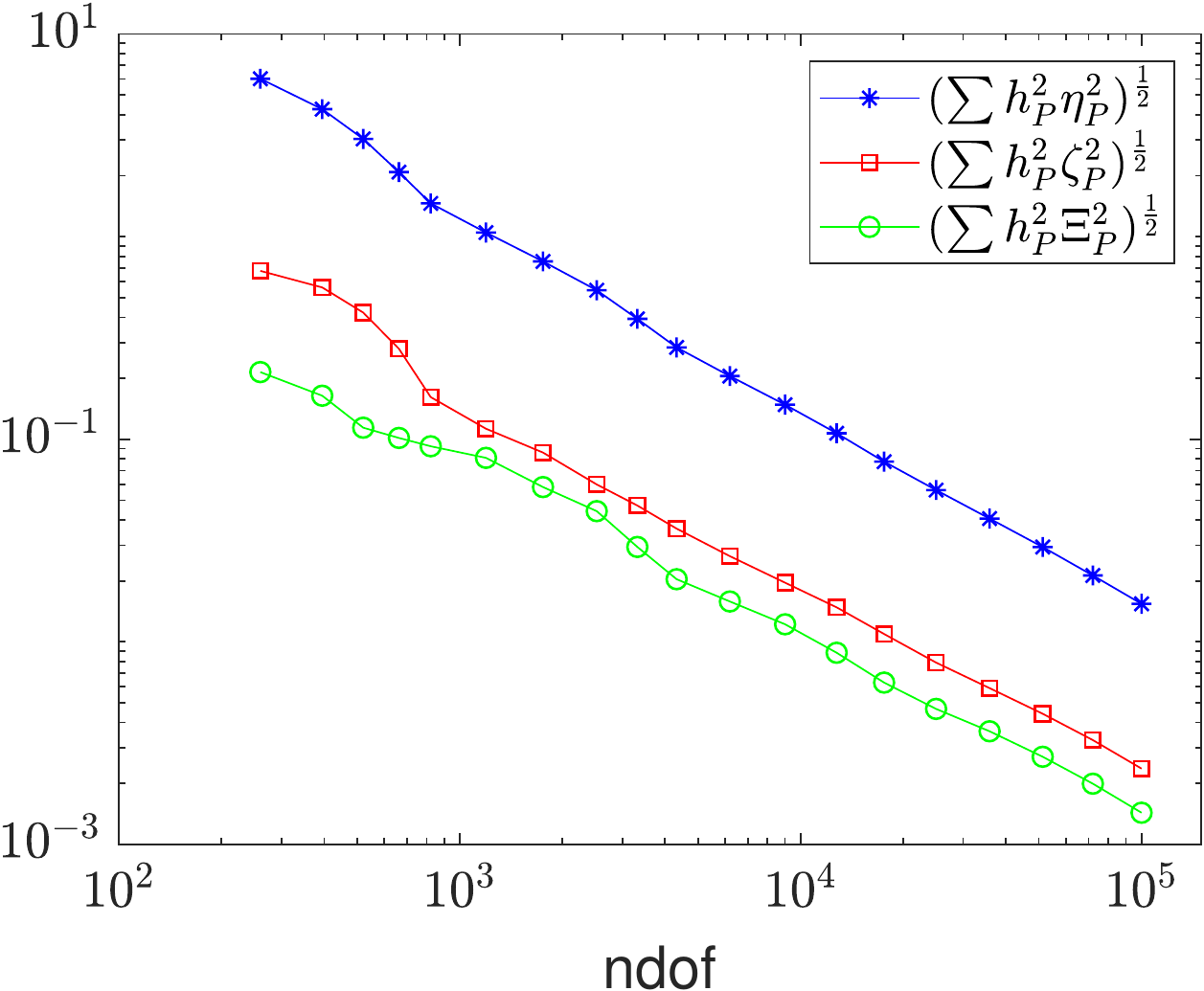}
	\end{subfigure}
	\caption{Estimator components corresponding to the error $L2e=\|u-\Pi_1u_h\|_{L^2(\Omega)}$ of the adaptive refinement presented in Subsection 6.2-6.4}
	\label{figl22}
\end{figure}
\textbf{Acknowledgements}  The authors sincerely thank one anonymous referee for  suggestions that led to Remark~5. The authors thankfully acknowlege the support from the MHRD SPARC project (ID 235) titled "The Mathematics and Computation of Plates" and the third author also thanks the hospitality of the Humboldt-Universit\"{a}t zu Berlin for the corresponding periods 1st July 2019-31st July 2019. The second author acknowledges the financial support of the University Grants Commission (UGC), Government of India.
\bibliographystyle{plainurl}
\bibliography{references}

\begin{thebibliography}{10}

\bibitem{9}
B.~Ahmad, A.~Alsaedi, F.~Brezzi, L.~D. Marini, and A.~Russo.
\newblock Equivalent projectors for virtual element methods.
\newblock {\em Comput. Math. Appl.}, 66(3):376--391, 2013.

\bibitem{15}
M.~Ainsworth and J.~T. Oden.
\newblock {\em A posteriori error estimation in finite element analysis},
  volume~37.
\newblock John Wiley \& Sons, 2011.

\bibitem{2}
B.~Ayuso~de Dios, K.~Lipnikov, and G.~Manzini.
\newblock The nonconforming virtual element method.
\newblock {\em ESAIM: M2AN}, 50(3):879--904, 2016.

\bibitem{1}
L.~Beir{\~a}o~da Veiga, F.~Brezzi, A.~Cangiani, G.~Manzini, L.~D. Marini, and
  A.~Russo.
\newblock Basic principles of virtual element methods.
\newblock {\em Math. Models Methods Appl. Sci.}, 23(01):199--214, 2013.

\bibitem{3}
L.~Beir{\~a}o~da Veiga, F.~Brezzi, L.~D. Marini, and A.~Russo.
\newblock The hitchhiker's guide to the virtual element method.
\newblock {\em Math. Models Methods Appl. Sci.}, 24(08):1541--1573, 2014.

\bibitem{4}
L.~Beir{\~a}o~da Veiga, F.~Brezzi, L.D. Marini, and A.~Russo.
\newblock Virtual element method for general second-order elliptic problems on
  polygonal meshes.
\newblock {\em Math. Models Methods Appl. Sci.}, 26(04):729--750, 2016.

\bibitem{da2014mimetic}
L.~Beir{\~a}o~da Veiga, K.~Lipnikov, and G.~Manzini.
\newblock {\em The mimetic finite difference method for elliptic problems},
  volume~11.
\newblock Springer, 2014.

\bibitem{beirao2017stability}
L.~Beir{\~a}o~da Veiga, C.~Lovadina, and A.~Russo.
\newblock Stability analysis for the virtual element method.
\newblock {\em Math. Models Methods Appl. Sci.}, 27(13):2557--2594, 2017.

\bibitem{da2015residual}
L.~Beir{\~a}o~da Veiga and G.~Manzini.
\newblock Residual a posteriori error estimation for the virtual element method
  for elliptic problems.
\newblock {\em ESAIM: M2AN}, 49(2):577--599, 2015.

\bibitem{binev2004adaptive}
P.~Binev, W.~Dahmen, and R.~DeVore.
\newblock Adaptive finite element methods with convergence rates.
\newblock {\em Numer. Math.}, 97(2):219--268, 2004.

\bibitem{braess2007finite}
D.~Braess.
\newblock {\em Finite elements: Theory, fast solvers, and applications in solid
  mechanics}.
\newblock Cambridge University Press, 2007.

\bibitem{brenner2015forty}
S.~Brenner.
\newblock Forty years of the {C}rouzeix-{R}aviart element.
\newblock {\em Numer. Methods Partial Differ. Equ.}, 31(2):367--396, 2015.

\bibitem{14}
S.~Brenner, Q.~Guan, and L.-Y. Sung.
\newblock Some estimates for virtual element methods.
\newblock {\em Comput. Methods Appl. Math.}, 17(4):553--574, 2017.

\bibitem{7}
S.~Brenner and R.~Scott.
\newblock {\em The mathematical theory of finite element methods}, volume~15.
\newblock Springer Science \& Business Media, New York, 2007.

\bibitem{brenner2018virtual}
S.~Brenner and L.~Sung.
\newblock Virtual element methods on meshes with small edges or faces.
\newblock {\em Math. Models Methods Appl. Sci.}, 28(07):1291--1336, 2018.

\bibitem{6}
A.~Cangiani, E.~H. Georgoulis, T.~Pryer, and O.~J. Sutton.
\newblock A posteriori error estimates for the virtual element method.
\newblock {\em Numer. Math.}, 137(4):857--893, 2017.

\bibitem{5}
A.~Cangiani, G.~Manzini, and O.~J. Sutton.
\newblock Conforming and nonconforming virtual element methods for elliptic
  problems.
\newblock {\em IMA J. Numer. Anal.}, 37(3):1317--1354, 2016.

\bibitem{cao2019anisotropic}
S.~Cao and L.~Chen.
\newblock Anisotropic error estimates of the linear nonconforming virtual
  element methods.
\newblock {\em SIAM J. Numer. Anal.}, 57(3):1058--1081, 2019.

\bibitem{11}
C.~Carstensen, A.~K. Dond, N.~Nataraj, and A.~K. Pani.
\newblock Error analysis of nonconforming and mixed fems for second-order
  linear non-selfadjoint and indefinite elliptic problems.
\newblock {\em Numer. Math.}, 133(3):557--597, 2016.

\bibitem{CFPP}
C.~Carstensen, M.~Feischl, M.~Page, and D.~Praetorius.
\newblock Axioms of adaptivity.
\newblock {\em Comput. Math. Appl.}, 67(6):1195--1253, 2014.

\bibitem{cc2}
C.~Carstensen and D.~Gallistl.
\newblock Guaranteed lower eigenvalue bounds for the biharmonic equation.
\newblock {\em Numer. Math.}, 126(1):33--51, 2014.

\bibitem{12}
C.~Carstensen, D.~Gallistl, and M.~Schedensack.
\newblock Adaptive nonconforming {C}rouzeix-{R}aviart {FEM} for eigenvalue
  problems.
\newblock {\em Math. Comp.}, 84(293):1061--1087, 2015.

\bibitem{cc1}
C.~Carstensen and J.~Gedicke.
\newblock Guaranteed lower bounds for eigenvalues.
\newblock {\em Math. Comp.}, 83(290):2605--2629, 2014.

\bibitem{carstensen2012explicit}
C.~Carstensen, J.~Gedicke, and D.~Rim.
\newblock Explicit error estimates for {Courant, Crouzeix-Raviart and
  Raviart-Thomas} finite element methods.
\newblock {\em J. Comput. Math.}, 30(4):337--353, 2012.

\bibitem{8}
C.~Carstensen and F.~Hellwig.
\newblock Constants in discrete {P}oincar{\'e} and {F}riedrichs inequalities
  and discrete quasi-interpolation.
\newblock {\em Comput. Methods Appl. Math.}, 18(3):433--450, 2018.

\bibitem{carstensen2018prove}
C.~Carstensen and S.~Puttkammer.
\newblock How to prove the discrete reliability for nonconforming finite
  element methods.
\newblock {\em arXiv preprint arXiv:1808.03535}, 2018.

\bibitem{cascon2008quasi}
J.~M. Cascon, C.~Kreuzer, R.~H. Nochetto, and K.~G. Siebert.
\newblock Quasi-optimal convergence rate for an adaptive finite element method.
\newblock {\em SIAM J. Numer. Anal.}, 46(5):2524--2550, 2008.

\bibitem{ciarlet1978finite}
P.~G. Ciarlet.
\newblock {\em The finite element method for elliptic problems}.
\newblock North-Holland, 1978.

\bibitem{10}
T.~Dupont and R.~Scott.
\newblock Polynomial approximation of functions in {S}obolev spaces.
\newblock {\em Math. Comp.}, 34(150):441--463, 1980.

\bibitem{Evans}
L.~C. Evans.
\newblock {\em Partial differential equations}, volume~19.
\newblock American Mathematical Society, Providence, RI, second edition, 2010.

\bibitem{HUANG2021113229}
J.~Huang and Y.~Yu.
\newblock A medius error analysis for nonconforming virtual element methods for
  {P}oisson and biharmonic equations.
\newblock {\em J. Comput. Appl. Math.}, 386, 2021.
\newblock \href {http://dx.doi.org/https://doi.org/10.1016/j.cam.2020.11322}
  {\path{doi:https://doi.org/10.1016/j.cam.2020.11322}}.

\bibitem{dG}
O.~A. Karakashian and F.~Pascal.
\newblock A posteriori error estimates for a discontinuous {G}alerkin
  approximation of second-order elliptic problems.
\newblock {\em SIAM J. Numer. Anal.}, 41(6):2374--2399, 2003.

\bibitem{17}
K.~Kim.
\newblock A posteriori error analysis for locally conservative mixed methods.
\newblock {\em Math. Comp.}, 76(257):43--66, 2007.

\bibitem{mora2015virtual}
D.~Mora, G.~Rivera, and R.~Rodr{\'\i}guez.
\newblock A virtual element method for the {S}teklov eigenvalue problem.
\newblock {\em Math. Models Methods Appl. Sci.}, 25(08):1421--1445, 2015.

\bibitem{product}
A.~Sommariva and M.~Vianello.
\newblock Product {G}auss cubature over polygons based on {G}reen's integration
  formula.
\newblock {\em BIT Numer. Math.}, 47(2):441--453, 2007.

\bibitem{sutton2017virtual}
O.~J. Sutton.
\newblock {\em Virtual element methods}.
\newblock PhD thesis, University of Leicester, 2017.

\bibitem{16}
R.~Verf{\"u}rth.
\newblock {\em A review of a posteriori error estimation and Adaptive
  Mesh-Refinement Techniques}.
\newblock {W}iley-{T}eubner, New York, 1996.

\end{thebibliography}
\end{document}